\documentclass{article}
\usepackage[utf8]{inputenc}
\usepackage[english]{babel}
\usepackage{amsmath, amscd,  amsthm, amstext, amssymb, verbatim, version, mathrsfs, bbm, geometry}
\usepackage[all, cmtip]{xy}
\usepackage{dsfont}

\author{Anthony Blanc\footnote{European Postdoctoral Institute, IHES. \texttt{anthony.blanc@ihes.fr}}
}

\title{Topological K-theory of complex noncommutative spaces}

\theoremstyle{plain}
\newtheorem{prop}{Proposition}[section]
\newtheorem{cor}[prop]{Corollary}
\newtheorem{lem}[prop]{Lemma}
\newtheorem{theo}[prop]{Theorem}
\newtheorem{df}[prop]{Definition}
\newtheorem{conj}[prop]{Conjecture}

\theoremstyle{definition}
\newtheorem{rema}[prop]{Remark}
\newtheorem{ex}[prop]{Example}
\newtheorem{nota}[prop]{Notation}

\newcommand{\Q}{\mathbb{Q}}
\newcommand{\Z}{\mathbb{Z}}
\newcommand{\N}{\mathbb{N}}
\newcommand{\C}{\mathbb{C}}
\newcommand{\s}{\mathbb{S}}

\newcommand{\lef}{\mathbb{L}}
\newcommand{\R}{\mathbb{R}}
\newcommand{\M}{\mathcal{M}}
\newcommand{\K}{\mathcal{K}}
\newcommand{\A}{\mathbf{A}}
\newcommand{\ocal}{\mathcal{O}}
\newcommand{\D}{\mathcal{D}}

\newcommand{\lmo}{\longrightarrow}
\newcommand{\lmos}[1]{\stackrel{#1} {\longrightarrow}}
\newcommand{\mo}{\rightarrow}
\newcommand{\mos}[1]{\stackrel{#1} {\rightarrow}}
\newcommand{\moi}{\hookrightarrow}
\def\dar[#1]{\ar@<2pt>[#1]\ar@<-2pt>[#1]}
\newcommand{\mbb}[1]{\mathbb{#1}}
\newcommand{\mbf}[1]{\mathbf{#1}}
\newcommand{\mcal}[1]{\mathcal{#1}}

\newcommand{\mrm}[1]{\mathrm{#1}}

\newcommand{\tel}{\otimes^{\mathbb{L}}}
\newcommand{\sml}{\wedge^{\mathbb{L}}}
\newcommand{\te}{\otimes}
\newcommand{\sm}{\wedge}

\newcommand{\uh}{\underline{\mrm{h}}}
\newcommand{\uk}{\underline{k}}
\newcommand{\und}[1]{\underline{#1}}
\newcommand{\map}{\mathrm{Map}}
\newcommand{\Hom}{\mathrm{Hom}}
\newcommand{\End}{\mathrm{End}}
\newcommand{\uaut}{\underline{\mathrm{Aut}}}
\newcommand{\ugl}{\underline{\mathrm{Gl}}}
\newcommand{\uend}{\underline{\mathrm{End}}}
\newcommand{\uma}{\underline{\mathrm{M}}}

\newcommand{\homi}{\underline{\mathrm{Hom}}}
\newcommand{\rhomi}{\mathbb{R}\underline{\mathrm{Hom}}}
\newcommand{\U}{\mbb{U}}
\newcommand{\V}{\mbb{V}}

\newcommand{\sps}{Sp_{S^1}}
\newcommand{\spss}{Sp_{S^2}}
\newcommand{\spcon}{Sp^{\mathrm{con}}}
\newcommand{\sinf}{\Sigma^{\infty}}

\newcommand{\mota}{\mbb{M}_{a}}
\newcommand{\mloc}{\mbb{M}_{loc}(k)}

\newcommand{\unit}{\mathbbm{1}}
\newcommand{\dgcat}{dgCat_{k}}
\newcommand{\dgcattf}{dgCat^{\mrm{tf}}_{k}}
\newcommand{\dgcatc}{dgCat_{\C}}

\newcommand{\dgmor}{dgMor_{k}}
\newcommand{\dgmorc}{dgMor_{\C}}
\newcommand{\affk}{\mathrm{Aff}_k}
\newcommand{\schk}{\mrm{Sch}_k}
\newcommand{\affc}{\mathrm{Aff}_{\C}}
\newcommand{\afflissc}{\mathrm{Aff^{liss}_{\C}}}
\newcommand{\calgc}{\mrm{CAlg}_{\C}}
\newcommand{\schc}{\mrm{Sch}_{\C}}

\newcommand{\lissc}{\mathrm{Sch^{liss}_{\C}}}

\newcommand{\spec}{\mathrm{Spec}}

\newcommand{\hlamb}{H\Lambda-Mod_\s}

\newcommand{\kc}{\tilde{\mathbf{K}}}
\newcommand{\kn}{\mathbf{K}}
\newcommand{\ka}{\mathbf{k}}
\newcommand{\ukc}{\und{\tilde{\mathbf{K}}}}
\newcommand{\ukn}{\und{\mathbf{K}}}

\newcommand{\kh}{KH}
\newcommand{\ukh}{\und{KH}}

\newcommand{\kcst}{\tilde{\mathbf{K}}^{\mathrm{st}} }
\newcommand{\kst}{\mathbf{K}^{\mathrm{st}} }
\newcommand{\ktop}{\mathbf{K}^{\mathrm{top}} }
\newcommand{\kctop}{\tilde{\mathbf{K}}^{\mathrm{top}} }
\newcommand{\ktopu}{K_{\mathrm{top}} }

\newcommand{\hc}{\mathrm{HC}}
\newcommand{\hcn}{\mathrm{HN}}
\newcommand{\hh}{\mathrm{HH}}
\newcommand{\hp}{\mathrm{HP}}
\newcommand{\hpa}{\mathrm{HP^{alg}}}
\newcommand{\uhp}{\und{\mathrm{HP}}}
\newcommand{\uhh}{\und{\mathrm{HH}}}
\newcommand{\uhcn}{\und{\mathrm{HN}}}
\newcommand{\hdrna}{\mathrm{H_{DR}^{naive}}}
\newcommand{\hdran}{\mathrm{H_{DR}^{naive,an}}}
\newcommand{\hdel}{\mathrm{H}_\mcal{D}}
\newcommand{\hb}{\mathrm{H_{B}}}

\newcommand{\hbs}{\mathrm{H_{\s,B}}}

\newcommand{\ao}{\mathbf{A}^1}
\newcommand{\po}{\mathbf{P}^1}
\newcommand{\gm}{\mathbf{G}_m}
\newcommand{\cuu}{\C[u^{\pm 1}]}

\newcommand{\re}[1]{\vert #1\vert}
\newcommand{\ret}[1]{\vert #1\vert^{\mrm{top}}}
\newcommand{\resp}[1]{\vert #1 \vert_{\s}}
\newcommand{\redel}[1]{\vert #1 \vert_{\del}}
\newcommand{\regam}[1]{\vert #1 \vert_{\gam}}
\newcommand{\resh}[1]{\vert #1 \vert_{S^2}}

\newcommand{\del}{\Delta}
\newcommand{\gam}{\Gamma}
\newcommand{\sh}{\stackrel{h}{\amalg}}
\newcommand{\ph}{\stackrel{h}{\times}}
\newcommand{\wh}{\stackrel{h}{\wedge}}
\newcommand{\cone}{\mathrm{Cone}}
\newcommand{\cocone}{\mathrm{Cocone}}
\newcommand{\chc}{\mathrm{Ch_c}}
\newcommand{\ch}{\mathrm{Ch}}
\newcommand{\chst}{\mathrm{Ch^{st}}}

\newcommand{\chutop}{\mathrm{Ch_{utop}}}
\newcommand{\chtop}{\mathrm{Ch^{top}}}
\newcommand{\chctop}{\mathrm{Ch^{top}_c}}
\newcommand{\bul}{\bullet}
\newcommand{\bu}{\mathbf{bu}}
\newcommand{\BU}{\mathbf{BU}}

\newcommand{\spr}{SPr(\affc)}
\newcommand{\sprliss}{SPr(\afflissc)}
\newcommand{\sprs}{SPr(\schc)}

\newcommand{\spret}{SPr(\affc)^{\mathrm{\acute{e}t}}}
\newcommand{\sprlisset}{SPr(\afflissc)^{\mathrm{\acute{e}t}}}
\newcommand{\sprset}{SPr(\schc)^{\mathrm{\acute{e}t}}}
\newcommand{\sprslisset}{SPr(\lissc)^{\mathrm{\acute{e}t}}}

\newcommand{\shpro}{Sh_{\mrm{pro}}(\schc)}
\newcommand{\shproliss}{Sh_{\mrm{pro}}(\lissc)}
\newcommand{\sprpro}{SPr(\schc)^{\mathrm{pro}}}
\newcommand{\sprproliss}{SPr(\lissc)^{\mathrm{pro}}}
\newcommand{\sppro}{Sp(\schc)^{\mathrm{pro}}}

\newcommand{\spafk}{Sp(\affk)}
\newcommand{\spaf}{Sp(\affc)}
\newcommand{\spafliss}{Sp(\afflissc)}
\newcommand{\spretao}{SPr^{\mathrm{\acute{e}t}, \ao}}
\newcommand{\spetao}{Sp^{\mathrm{\acute{e}t}, \ao}}

\newcommand{\splissna}{Sp^{\mathrm{Nis, \ao}}_{\mathrm{Liss}}}

\newcommand{\shc}{\mathcal{SH}_{\C}} 

\newcommand{\m}[1]{\widehat{#1}} 

\newcommand{\lpe}{\mathrm{\mathcal{L}_{perf}}}
\newcommand{\lqc}{\mathrm{\mathcal{L}_{qcoh}}}

\newcommand{\parf}{\mathrm{Perf}}
\newcommand{\uparf}{\und{\mathrm{Parf}}}
\newcommand{\pspa}{\mathrm{PsParf}}
\newcommand{\upspa}{\und{\mathrm{PsParf}}}
\newcommand{\proj}{\mathrm{Proj}}
\newcommand{\uproj}{\und{\mathrm{Proj}}}
\newcommand{\psproj}{\mathrm{PsProj}}
\newcommand{\upsproj}{\und{\mathrm{PsProj}}}
\newcommand{\vect}{\mathrm{Vect}}

\newcommand{\sspsp}{ssp_{\s}}

\newcommand{\tc}{\mbb{T}}
\newcommand{\bcl}{\mathbf{B}}

\newcommand{\drcx}{\Omega_X^*}

\usepackage[pdftex]{hyperref}
\hypersetup{colorlinks=true, citecolor=black, linkcolor=black}

\begin{document}

\maketitle

\begin{abstract}
The purpose of this work is to give a definition of a topological K-theory for dg-categories over $\C$ and to prove that the Chern character map from algebraic K-theory to periodic cyclic homology descends naturally to this new invariant. This topological Chern map provides a natural candidate for the existence of a rational structure on the periodic cylic homology of a smooth proper dg-algebra, within the theory of noncommutative Hodge structures. The definition of topological K-theory consists in two steps : taking the topological realization of algebraic K-theory and inverting the Bott element. The topological realization is the left Kan extension of the functor "space of complex points" to all simplicial presheaves over complex algebraic varieties. Our first main result states that the topological K-theory of the unit dg-category is the spectrum $\BU$. For this we are led to prove a homotopical generalization of Deligne's cohomological proper descent, using Lurie's proper descent. The fact that the Chern character descends to topological K-theory is established by using Kassel's Künneth formula for periodic cyclic homology and the proper descent. In the case of a dg-category of perfect complexes on a separated scheme of finite type, we show that we recover the usual topological K-theory of complex points. We show as well that the Chern map tensorised with $\C$ is an equivalence in the case of a finite dimensional associative algebra -- providing a formula for the periodic homology groups in terms of the stack of finite dimensional modules. 
\end{abstract}

\tableofcontents

\section{Introduction}

The idea of associating algebraic invariants to geometrical objects culminated with the study of derived categories in algebraic geometry. Kontsevich's noncommutative geometry goes further in really defining a noncommutative space as being the category of functions defined on this space. Therefore by a noncommutative space we mean a differential $\Z$-graded category (or dg-category for short). This categorical point of view on geometry is a very powerful one, allowing one to treat on an equal footing objects having a priori different origin. For example in algebraic geometry the derived category of perfect complexes of quasi-coherent sheaves \cite{bondalorlovicm}, \cite{orlovdereq}; in symplectic geometry the Fukaya category of a symplectic variety \cite{fooo}, \cite{seidelfuk}; in representation theory the derived category of complexes of representations of a quiver \cite{kellerclus}, \cite{krauserep}; in singularity theory the category of matrix factorizations of an isolated hypersurface singularity \cite{orlovderivedcatsing}, \cite{tobimf}, \cite{efimovmf}; in algebraic analysis the derived category of deformation-quantization modules \cite{kontsdefq}, \cite{shapiradef}, etc. The first two examples are particulary important in the mathematical formulation of mirror symmetry, for example in the foundational paper \cite{kkp}. In this latter paper, the authors formulate the Homological Mirror Symmetry conjecture in terms of noncommutative (nc for short) Hodge structures (see §3 of loc.cit.). The main conjecture (Conj 2.24 of loc.cit.) claims that there exists a nc-Hodge structure on the periodic cyclic homology of any smooth and proper dg-algebra over $\C$, which furthermore comes from a commutative one. The definition of a rational pure nc-Hodge structure consists of different structures attached to periodic cyclic homology : the de Rham structure corresponding to the Hodge filtration and the Betti structure corresponding to the rational structure given by rational Betti cohomology (see Def 2.5 of loc.cit.). The first motivation for pursuing topological K-theory came from finding a candidate for the rational structure for a nc-Hodge structure (see §2.2.6 of loc.cit.). The severals algebraic flavoured invariants associated to dg-categories (algebraic K-theory, cyclic homology, Hochschild cohomology, Hall algebras,...) were intensively studied by Keller, Tsygan, Tabuada, Cisinski and Toën. Topological K-theory can itself be thought of as the noncommutative analog of rational Betti cohomology of algebraic varieties. 

\vspace*{.8cm}

\begin{center}
\textit{Main results.}
\end{center}

The aim of the paper is to give a construction of topological K-theory for dg-categories over $\C$ which satisfies all the expected properties. These properties are summarized in the following main theorem. We denote by $\dgcatc$ the category of $\C$-dg-categories, $Sp$ the category of symmetric spectra, $\ktopu$ is the usual nonconnective topological K-theory, $\kn$ the algebraic K-theory, $\hcn$ the negative cyclic homology, $\hp$ the periodic cyclic homology, $\BU=\ktopu(\ast)$ is the usual spectrum of topological K-theory, and $\unit$ the dg-category with one object and $\C$ as endomorphism ring. 

\begin{theo}\label{motiv}
There exists a functor $\ktop : \dgcatc \lmo Sp$, called the topological K-theory of noncommutative spaces, which satisfies the following properties. 
\begin{description}
\item[a.] $\ktop(\unit) \simeq \BU$ in the homotopy category of spectra. 
\item[b.] If $X$ is a separated $\C$-scheme of finite type, then there exists a functorial isomorphism 
$$\ktop(\lpe(X))\simeq \ktopu(X(\C)),$$ 
where $\lpe(X)$ is the dg-category of perfect complexes of quasi-coherent $\ocal_X$-modules on $X$. 
\item[c.] $\ktop$ commutes with filtered colimits, is Morita-invariant and sends exact sequences of dg-categories to distinguished triangles in the triangulated homotopy category of spectra. 
\item[d.] For any dg-category $T$, there exists a functorial commutative square in the homotopy category of spectra, 
$$\xymatrix{ \kn(T) \ar[r]^-{\ch } \ar[d]^-{} & \hcn(T) \ar[d]^-{ } \\ \ktop(T) \ar[r]^-{ \chtop} & \hp(T) }$$
such that in the case of a separated scheme of finite type $X$, the map $\chtop$ for $T=\lpe(X)$ is isomorphic to the usual topological Chern character. 
\end{description}
\end{theo} 

Our definition of topological K-theory for dg-categories is really inspired by Friedlander--Walker's definition of semi-topological K-theory for quasi-projective complex algebraic varieties (see \cite{fwcomp}, \cite{fwsemi}), itself inspired by the work of Thomason \cite{thomet}. 
 
The definition we give of topological K-theory starts with algebraic K-theory and proceeds in two steps. If $T$ is a $\C$-dg-category, we denote by $\kn(T)$ the nonconnective algebraic K-theory of $T$ (defined by Schlichting \cite{schl}). We also have a presheaf of spectra on the site of complex affine $\C$-schemes,
$$\ukn(T) : \spec(A) \longmapsto \kn(T\te_\C A).$$
The first step consists in applying the topological realization of simplicial presheaves. Denote by $\affc$ the category of complex affine schemes of finite type, by $SSet$ the category of simplicial sets, and by $\spr$ the category of simplicial presheaves on $\affc$. Then the topological realization is a functor 
$$ssp : \spr\lmo SSet$$
It is the left Kan extension of the functor "space of complex points" $\affc\lmo SSet$ (where space is understood as simplicial set) along the Yoneda embedding $\affc\lmo \spr$. The topological realization extends naturally to presheaves of spectra and is a left Quillen functor for the $\ao$-étale local model structure on simplicial presheaves. We denote by $\re{-}:=\lef ssp$ the left derived functor with respect to this model structure. 

\begin{df} \emph{(see Def \ref{defkst})} ---
The semi-topological K-theory of a $\C$-dg-category $T$ is the spectrum $\kst(T):=\re{\ukn(T)}$. 
\end{df}

We can see that there exists a canonical map $\kn(T)\lmo \kst(T)$. When $T=\lpe(X)$ is the dg-category of perfect complexes on a smooth complex algebraic variety, the map $\kn_0(X)\lmo \kst_0(X)$ is the quotient map given by the algebraic equivalence relation, i.e. two algebraic vector bundles $E\lmo X$ and $E'\lmo X$ are algebraically equivalent if there exists a connected complex algebraic curve $C$, a vector bundle $E''\lmo C\times X$ such that we recover $E$ and $E'$ by restricting $E''$ to some $\C$-points of $C$. 

Using the proper topology (we could also have used the \emph{cdh}-topology of Voevodsky), and the proper local model structure on simplicial presheaves, we show that the topological realization of a presheaf remains unchanged if we restrict this presheaf to smooth schemes. Denote by $l^*$ the restriction to smooth schemes. 

\begin{theo}\label{restintro}\emph{(see Thm \ref{restliss})} --- 
Let $F\in SPr(\affc)$. Then there exists a canonical isomorphism $\re{l^* F} \simeq \re{F}$ in the homotopy category of simplicial sets.   
\end{theo}

To prove this, as announced in the abstract we prove a homotopical generalization of Deligne's cohomological proper descent which can be expressed by saying that $ssp$ is a left Quillen functor with respect to the proper local model structure on simplicial presheaves.

\begin{prop} \emph{(see Prop \ref{quipro})} --- 
For every proper hypercovering $Y_\bul\lmo X$ of a scheme, the induced map, $hocolim_{\del^{op} }  \re{Y_\bul}\lmo \re{X}$ is an isomorphism in the homotopy category of simplicial sets.  
\end{prop} 

Theorem \ref{restintro} allows us to prove the following. Denote by $\bu$ the connective cover of $\BU$. 

\begin{theo} \emph{(see Thm \ref{annupoint})} --- There exists an isomorphism $\kst(\unit)\simeq \bu$ in the homotopy category of spectra. 
\end{theo}

Using the Tabuada--Cisinski theorem on the corepresentability of K-theory inside the motivic category of dg-categories, we define a canonical structure of ring spectrum on $\ukn(\unit)$ and a structure of $\ukn(\unit)$-module on $\ukn(T)$ for any dg-category $T$. Therefore $\kst(T)$ is a $\bu$-module.
 
The second step in the process of the definition of topological K-theory is the inversion of a Bott element $\beta\in \pi_2(\bu)$. 

\begin{df} \emph{(see Def \ref{deftop})} ---
The topological K-theory of a $\C$-dg-category $T$ is the spectrum $\ktop(T)=\kst(T)[\beta^{-1}]$. 
\end{df} 

This definition is of course motivated by the agreement of the Bott inverted Friedlander--Walker semi-topological K-theory with the usual topological K-theory of complex points (see \cite[Thm 5.8]{fwrat}), which was also motivated by Thomason's work on the Bott inverted algebraic K-theory with finite coefficients. The point $b.$ in theorem \ref{motiv} is proved by using Riou's Spanier--Whitehead duality in the motivic homotopy category of smooth schemes (see prop.\ref{compvarlisse}). Finally the point $d.$ of theorem \ref{motiv}, namely the construction of the topological Chern map is achieved using Kassel's Künneth formula for periodic cyclic homology and once again Thm \ref{restintro}. 

Finally, in §\ref{sectkstmt}, we give a convenient description of the connective semi-topological K-theory of any dg-category in terms of the stack $\M^T$ of perfect $T^{op}$-dg-modules. This stack is a stack of $E_\infty$-spaces because of the sum of dg-modules. It turns out that its topological realization $\re{\M^T}$ is a group-like $E_\infty$-space and one can consider it as a connective spectrum. The proof is based on the existence of an $\ao$-equivalence between the stack in $E_\infty$-spaces $\M^T$ and the Waldhausen construction of the category of perfect $T^{op}$-dg-modules (see  Prop \ref{mk} and Thm \ref{kstmt}).

\vspace*{.8cm}

\begin{center}
\textit{Applications and related works}. 
\end{center}

As was said above, the semi-topological K-theory of quasi-projective complex algebraic varieties has already been defined by Friedlander--Walker. However let us mention that their operation of "semi-topologization" seems to differ from our topological realization. An indirect comparison is nevertheless given by the point $b.$ of theorem \ref{motiv}. 

The construction of topological K-theory of dg-categories over $\C$ and the existence of the topological Chern map can be used in order to state the lattice conjecture claiming that this new invariant gives a rational structure on the periodic homology of a smooth proper dg-category. 

\begin{conj} \emph{(Lattice conjecture)} --- For every smooth proper $\C$-dg-category $T$, the Chern map $\chtop\sm_\s H\C : \ktop(T)\sm_{\s} H\C\lmo \hp(T)$ is an isomorphism in the homotopy category of spectra. 
\end{conj}

By the point $b$ and $d$ in theorem \ref{motiv}, the lattice conjecture is true for dg-categories of the form $\lpe(X)$ for $X$ any separated scheme of finite type over $\C$.  

As an application in §\ref{finalg} we show the lattice conjecture is true for finite dimensional associative $\C$-algebras, but if we replace $\ktop$ by another invariant $\kctop$ which is the Bott inverted \emph{connective} semi-topological K-theory. The proof consists in showing the invariance of $\kctop$ under infinitesimal extension. It has the following consequence about the periodic homology groups of a finite dimensional algebra. If $B$ is such an algebra, denote by $\vect^B$ the stack of projective right modules of finite type over $B$, and by $\re{\vect^B}^{ST}$ the stabilization of its topological realization with respect to the $B$-module $B$.

\begin{prop} \emph{(see Cor \ref{formhp1})} --- The Chern map $\kctop(B)\lmo \hp(B)$ induces an isomorphism of $\C$-vector spaces for $i=0,1$, 
$$ colim_{k\geq 0} \pi_{i+2k} \re{\vect^B}^{ST}\te_\Z \C\simeq\hp_i(B)$$
where the colimit is induced by the action of the Bott element $\beta$ on homotopy groups, $\pi_i\re{\vect^B}^{ST} \lmos{\times \beta} \pi_{i+2} \re{\vect^B}^{ST}$.
\end{prop} 

In the case where $B$ is moreover of finite homological dimension, i.e. smooth as a dg-category, we prove a similar formula but in terms of the stack $\vect_B$ of finite dimensional modules over $B$, see Cor \ref{cons2}. The stack $\vect_B$ has already been thought of by mathematicians as reflecting the noncommutative geometrical property of the algebra $B$, or in other words as a noncommutative spectrum. Therefore our formula gives a somehow geometrical interpretation of the periodic cyclic homology groups in this case. 

We believe the lattice conjecture is true for objects from very diverse origins like smooth proper DM-stacks over $\C$ and categories of matrix factorizations. We notice that the latter example requires the use of $2$-periodic dg-categories which have their own homotopy theory and the topological realization has also to be adapted to this $2$-periodic context. 

Another interesting question raised by topological K-theory is the behavior of the Chern map $\ktop(T)\lmo \hp(T)$ with respect to the Gauss--Manin connection supported by the periodic homology of a family of dg-categories parametrized by a smooth affine variety (see \cite{caract}, \cite{tsygm}). We believe the map $\chtop$ is flat with respect this connection. We can also mention \cite[§2.2.5]{kkp} for the relation with the associated variation of nc-Hodge structures. 

The point $d$ in theorem \ref{motiv} allows to propose a general definition of the Deligne cohomology of smooth and proper dg-categories over $\C$. In the commutative case, Deligne cohomology of degree $2p$ of a smooth projective variety with coefficients in the $p$th Deligne complex $\Z(p)_\D$ is given by the extension of integral $(p,p)$-classes with its $p$th intermediate Griffiths's jacobian. This suggests the following definition of Deligne cohomology for a smooth and proper $\C$-dg-category $T$, 
$$\hdel(T):=\ktop(T)\ph_{\hp(T)} \hcn(T).$$
As usual, by cohomology we mean the spectrum whose stable homotopy groups give the actual cohomology groups. This new invariant comes with a map $\kn(T)\lmo \hdel(T)$ which can be thought of as a mix between usual and secondary characteristic classes. Let us mention also Marcolli--Tabuada's work \cite{tabmarjac} on intermediate jacobians for dg-categories. 

Finally let us mention another possible relation with the work of Freed \cite{freed} where for problems of dimension reduction of certains Chern--Simons theories, the author raises the question to find a refinement of Hochschild homology defined over the integers. The topological K-theory endowed with its Chern map seems to be a good candidate for this refinement.

\paragraph*{Acknowledgements ---} I want to warmly thank Bertrand Toën for having accepted me as his student, for sharing his mathematical ideas with me and for his constant support during the preparation of this work. I want to thank Denis-Charles Cisinski, Benoit Fresse, Benjamin Hennion, Dmitry Kaledin, Bernhard Keller, Pranav Pandit, Frédéric Paugam, Marco Robalo, Abd\'{o} Roig-Maranges, Carlos Simpson, Boris Tsygan, Michel Vaquié and Gabriele Vezzosi for their support and the motivating discussions we had on the subject of this paper. I thank the referee for his useful comments that did contribute to correct and improve this article from its early version. 

\vspace*{.8cm}

\begin{center}
\textit{Notations, conventions.}
\end{center}

\begin{itemize}
\item We handle set theoretical issues by choosing two Grothendieck universes $\U\in\V$ in the sense of \cite{sga4-1}. We suppose that $\U$ contains the set $\N$ of natural numbers. 

\item In general we consider objects (sets, simplicial sets, topological spaces,...) which are $\V$-small. We denote by $Set$ the category of $\V$-small sets, $SSet$ the category of $\V$-small simplicial sets, $Top$ the category of $\V$-small topological spaces, $Sp$ the category of $\V$-small symmetric spectra.
\item If $C$ is any category, we denote by $C^{op}$ the opposite category of $C$. If $D$ is any another category, we denote by $C^D$ the category of functors and natural transformations $D\lmo C$. 
\item Hovey's book \cite{hoveymod} is our reference for definitions and results on model category. If $M$ is a model category, we denote by $Ho(M)$ the homotopy category of $M$. Most of the time we call simply \emph{equivalences} the weak equivalences. If $M$ is a cofibrantly generated model category, and $C$ any category, we will use several times the projective model structure on functors $M^C$. Its existence is proven in \cite[§11.6]{hirs}. 
\item The category $SSet$ is endowed with its standard model structure (see \cite[Chap 3]{hoveymod}). The category $Sp$ is endowed with its standard stable model structure (see \cite{hss}), which is Quillen equivalent to the usual stable model category of spectra, but has the advantage to be model monoidal. The category $Top$ is endowed with its standard model structure (see e.g. \cite[§2.4]{hoveymod}). 
\item All the spectra we consider are symmetric spectra. We will practise the abuse to forget the word "symmetric" in front of spectra to light the reading, but it will always mean symmetric spectra. 
\item When we deal with symmetric ring spectra, by convention it will mean stricly associative symmetric ring spectra. If $A$ is a symmetric ring spectrum, the expression $A$-modules means by convention right $A$-modules. The notation $A-Mod_\s$ stands for the category of right $A$-modules spectra, which we sometimes call just $A$-modules. 
\item If $C$ is a category and $x,y\in C$ two objects of $C$, we denote by $\Hom_C(x,y)$ the set of morphisms or maps from $x$ to $y$ in $C$. If $C$ is enriched in a category $V$ different than the category of sets, we denote by $\homi_C(x,y)$ the object of maps from $x$ to $y$, specifying if needed the category $V$. If $C$ is moreover endowed with a model structure compatible with the enrichment, we denote by $\rhomi_C$ its derived internal hom. In the particular case of $V=SSet$, we denote the internal homs by $\map_C$ and $\R\map_C$. 
\item In a category with final object we denote "this" final object by $\ast$.
\item We denote by $\del$ the standard simplicial category of finite ordinals $[0], [1], [2], \hdots$ with increasing maps as morphisms. If nothing is specified $S^1$ refers to the standard model $S^1=\del^1/\partial \del^1$ of the simplicial circle, pointed by its zero simplex.

\item By convention, all schemes are \emph{separated of finite type over the base}. If $k$ is a commutative noetherian ring, we denote by $\affk$ the category of affine $k$-scheme of finite type over $k$, we denote by $\schk$ the category of separated $k$-schemes of finite type over $k$. 

\item If $C$ is a $\V$-small category and $V$ a locally $\V$-small category, we denote by $Pr(C, V)$ the category of presheaves on $C$ with values in $V$. In the particular case of $V=SSet$, we denote $Pr(C, V)=:SPr(C)$. In the particular case of $V=Sp$, we denote $Pr(C, Sp)=:Sp(C)$. 

\item \emph{In the particular cases of $V=SSet, Sp$ endowed with their standard model structure, the category $Pr(\affk, V)$ is by default endowed with its projective model structure, for which the equivalences and fibrations are levewise ones (cf. \cite[§11.6]{hirs}). \textbf{Unless otherwise stated, the notations $SPr(\affk)$ and $\spafk$ refer to the projective model structures, also called global model structures by opposition with the local ones}.}
\item If $M$ is a model category, and $F:I\lmo M$ a diagram in $M$ (with $I$ a category), we denote by $hocolim_I F$ the homotopy colimit of $F$, i.e. the left derived functor of the functor $colim_I : M^I\lmo M$ for the projective model structure. We denote by $holim_I F$ the homotopy limit of $F$, this is the right derived functor of the functor $lim_I : M^I\lmo M$ for the injective model structure. We denote by $A\ph_C B$ the homotopy pullback and by $A\sh_C B$ the homotopy pushout. 
\item With a left Quillen functor $f : M\lmo N$ between model categories, we will make several times the abuse to say that $\lef f : Ho(M)\lmo Ho(N)$ commutes with homotopy colimits to says that for all category $I$, and all $I$-diagram $F:I\lmo M$, the map $hocolim_I \lef f(F)\lmo \lef f(hocolim_I F)$ is an isomorphism in $Ho(N)$. We will practise the analog abuse for homotopy limits. 
\item In this text, unless stated otherwise, the sets of adjoint pairs of functors written horizontally are written such that every functor $F$ is left adjoint of the functor just below $F$. 

\end{itemize}

\section{Preliminaries}

In this section we set up some notations, definitions and results that will be used in the definition of the topological K-theory and its Chern character. The first part deals with particular models for homotopy coherently associative monoids (and its commutative analog) and its link with the homotopy theory of symmetric spectra. In the second part we set definitions and recall the properties of algebraic K-theory of noncommutative spaces (connective and nonconnective). Finally in the third part we use Cisinski--Tabuada's result on nonconnective K-theory of noncommutative spaces to redefine the ring structure on algebraic K-theory and the Chern map in a linear fashion with respect to the ring spectrum of algebraic K-theory.

\subsection{Monoids up to homotopy and connective spectra}\label{monoids}

We will use particular models for homotopy associative monoids and homotopy commutative monoids known as $\del$ and $\gam$-spaces respectively. They are quite convenient if someone wants to handle algebraic structures up to coherent homotopy while staying in the realm of model categories, without refering to Lurie's $\infty$-operads and $\infty$-categories. We recall basic results about $\del$-spaces, $\gam$-spaces, group completion and the link between the homotopy theory of very special $\gam$-spaces and the homotopy theory of connective spectra. At the end we recall how to define the Waldhausen K-theory spectrum using the group completion of $\gam$-spaces. 

Let $\Gamma$ be the skeletal category of finite pointed sets with objects the sets $\und{n}=\{0,\hdots, n\}$ with $0$ as basepoints for all $n\in\N$ and with morphisms all pointed maps of sets. 

\begin{df}\label{defdel} Let $M$ be a model category. 
\begin{itemize}
\item A $\del$-object (resp. a $\gam$-object) in $M$ is a functor $\del^{op}\lmo M$ sending $[0]$ to $\ast$ (resp. a functor $\gam\lmo M$ sending $\und{0}$ to $\ast$). Morphisms being natural transformations of functors we denote by $\del-M$ (resp. by $\gam-M$) the category of $\del$-objects (resp. of $\gam$-objects) in $M$. For $E\in \del-M$ (resp. $F\in\gam-M$), we adopt the notations $E([n])=E_n$ and $F(\und{n})=F_n$. 

\item A $\del$-object $E$ in $M$ is called special if all the Segal maps are weak equivalences in $M$, i.e. if for all $[n]\in\del$ the map 
$$ p_0^\ast \times \hdots \times p_{n-1}^\ast : E_n\lmo E_1^{\ph n} = E_1\ph \hdots \ph  E_1$$ 
is a weak equivalence in $M$ where $p_i : [1]\lmo [n]$, $p_i(0)=i$ and $p_i(1)=i+1$. We denote by $s\del-M$ the full subcategory of $\del-M$ consisting of special $\del$-objects in $M$. 

\item A $\gam$-object $F$ in $M$ is called special for all $\und{n}\in \gam$ the map 
$$ q^1_\ast \times \hdots \times q^n_\ast : F_n\lmo F_1^{\ph n}$$ 
is a weak equivalence in $M$, where $q^i : \und{n}\lmo \und{1}$, $q^i(j)=1$ if $j=i$ and $q^i(j)=0$ if $j\neq i$. We denote by $s\gam-M$ the full subcategory of $\gam-M$ consisting of special $\gam$-objects in $M$. 

\item If $E\in s\del-M$, we say that $E$ is very special if the map 
$$p_0^\ast \times d_1^\ast : E_2\lmo E_1\ph E_1$$ 
is a weak equivalence in $M$, where $d_1 : [1] \lmo [2]$ is the face map which avoids $1$ in $[2]$. We denote by $vs\del-M$ the full subcategory of $s\del-M$ consisting of very special $\del$-objects.

\item If $F\in s\gam-M$, we say that $F$ is very special if the map 
$$q^1_\ast \times \mu_\ast : F_2\lmo F_1\ph F_1$$
is a weak equivalence in $M$, where $\mu : \und{2}\lmo \und{1}$ is the map defined by $\mu(1)=1$ and $\mu(2)=1$. We denote by $vs\gam-M$ the full subcategory of $s\gam-M$ consisting of very special $\gam$-objects.
\end{itemize}
\end{df}

\begin{rema} 

\begin{itemize}
\item If we take $M=SSet$, the $\del$-objects and $\gam$-objects are usually called $\del$-spaces and $\gam$-spaces, e.g. in \cite{bf}. 

\item The special $\del$-objects in a model category $M$ are particular models for associative monoids up to coherent homotopy in $M$ (or $A_\infty$-monoids). If we take $M=Set$ the category of $\U$-small sets, with isomorphims of sets as weak equivalences, then the category $s\del-Set$ is equivalent to the category of monoids in sets. Namely a $\del$-set $E$ is map to the monoid $E_1$ with composition law given by 
$$\xymatrix{E_1\times E_1 \ar[r]^-{(d_0^*\times d_2^*)^{-1}} & E_2 \ar[r]^-{d_1^*} & E_1 }.$$
The associativity and unital conditions are recovered via the simplicial identities. 
The special $\gam$-objects are particular models for  commutative monoids up to coherent homotopy in $M$ (or $E_\infty$-monoids).  The composition law is recovered by the map $\mu$, and the commutativity is encoded by the map $\und{2}\lmo \und{2}$ interchanging $1$ and $2$. 

\item In the special case where $M$ is the projective model category of simplicial presheaves on a site, we can then replace homotopy products by products in definition \ref{defdel}. In this case, a special $\del$-object $E$ is very special if and only if the monoid $\pi_0 E_1$ is a group (i.e. every element has a two-sided inverse) (see \cite[Lem 1.3]{teze}). 

\end{itemize}
\end{rema}

Recall that we have at least three interesting model structures on $\del-M$ for any left proper combinatorial model category $M$ : 

\begin{itemize}
\item The \textit{projective} or \textit{strict model structure} for which weak equivalences and fibrations are levelwise weak equivalences and levelwise fibrations respectively. We denote by $\del-M$ this model structure. In all the sequel, by default "equivalence in $\del-M$" will refers to levelwise weak equivalence.
 
\item The \textit{special model structure} which is the Bousfield localization of the strict one with respect to the set of maps 
$$(\sqcup_{i=0}^{n-1} h_{p_i}:h_{[1]}\sqcup \cdots \sqcup h_{[1]}\lmo h_{[n]})_{n\geq 1}\quad  \square \quad (\textrm{generating cofibrations of M})$$ 
where $\square$ is the box product of \cite[Thm 3.3.2]{hoveymod}. We denote it by $\del-M^{sp}$. The fibrant objects for this model structure are the special $\del$-objects which are moreover levelwise fibrant. 

\item The \textit{very special model structure}  which is a Bousfield localization of the special one with respect to the map 
$$(h_{p_0} \sqcup h_{d_1}  : h_{[1]} \sqcup h_{[1]} \lmo h_{[2]}) \quad \square \quad (\textrm{generating cofibrations of M})$$
We denote it by $\del-M^{vsp}$. The fibrant objects for this model structure are the very special $\del$-objects which are moreover levelwise fibrant. 
\end{itemize}

We have derived identity functors
$$\xymatrix{Ho(\del-M) \ar@<2pt>[r]^{\lef id} &  Ho(\del-M^{sp}) \ar@<2pt>[r]^-{\lef id'} \ar@<2pt>[l]^{\R id} & Ho(\del-M^{vsp})\ar@<2pt>[l]^{\R id'} }.$$
We denote by $mon:=\R id \lef id$ the free homotopy associative monoid functor and by $(-)^+:=\R id'\lef id'$ the homotopy group completion. 
\\

We have similar model structures for $\gam$-objects : a projective model structure $\gam-M$, a special model structure $\gam-M^{sp}$, and a very special model structure $\gam-M^{vsp}$ with corresponding free homotopy commutative monoid and homotopy commutative group completion functors : 
$$\xymatrix{Ho(\gam-M) \ar[r]^{com}  & Ho(\gam-M^{sp}) \ar[r]^-{(-)^+}  & Ho(\gam-M^{vsp})}.$$
(We give the group completion the same notation as for associative monoids because one can show they are actually equivalent ; which can be expressed by the following). 
\\

We have a fully faithful functor from homotopy commutative monoids to homotopy associative monoids given by composition with the functor
$$\alpha : \del^{op}\lmo \gam,$$
defined on objects by $\alpha ([n])=\und{n}$. And for any map $f:[n]\lmo [m]$ in $\del$ we define \\ $\alpha(f)=g:\und{m}\lmo \und{n}$ by 
$$g(i)=\left\lbrace
\begin{array}{ccc}
0 & \text{if  }& 0\leq i\leq f(0) \\ 
j & \text{if  }& f(j-1)< i\leq f(j) \\ 
0  & \text{if  }& f(n)<i \\ 
\end{array} 
\right.$$

One can verify that $\alpha (p_i)=q^{i+1}$ for $i=0,\hdots, n-1$, and $\alpha (d_1)=\mu$ so that the fully faithful functor 
$$\alpha^\ast : \gam-M \lmo \del-M, $$
sends special $\gam$-objects to special $\del$-objects and also very special objects to such. Hence we obtain a diagram

$$\xymatrix{ Ho(\gam-M)   \ar[d]_-{\alpha^\ast} \ar[r]^-{com} &  Ho(\gam-M^{sp})  \ar[d]_-{\alpha^\ast} \ar[r]^-{(-)^+} & Ho(\gam-M^{vsp} )  \ar[d]_{\alpha^\ast} \\
 Ho(\del-M)  \ar[r]^-{mon} &  Ho(\del-M^{sp})  \ar[r]^-{(-)^+} &  Ho(\del-M^{vsp} )  }$$
The left square is not commutative anymore but we can actually show that the right square is commutative up to canonical isomorphism.

\begin{rema}\label{remgroup}
Working with $M=SSet$ or with the global model category of simplicial presheaves on a category we have the following. 
By Segal's Theorem \cite[Prop 1.5]{seg}  the group completion functor $(-)^+$ has as model the composite functor 
$$\xymatrix{Ho(\del-SSet^{sp}) \ar[r]^-{\re{-} } & Ho(SSet_\ast) \ar[r]^-{\R\Omega_\bul} & Ho(\del-SSet^{vsp})  },$$
where $\re{-}$ is the realization of bisimplicial sets and for a pointed fibrant simplicial set $(X,x)$ the simplicial set $\Omega_n X$ is the simplicial set of maps from $\del^n$ to $X$ which send all vertices on $x$. We have indeed more : the composite functor $(-)^+\circ mon$ has as model the functor $\R\Omega_\bul \circ \re{-}$. 
\end{rema}

\begin{ex}\label{exespk}
The following example will be important to us in this paper on K-theory. If $C$ is any Waldhausen category, we have a $\del$-space
$$\K_\bul (C):=NwS_\bul C$$
where $Nw$ is the nerve of weak equivalences and $S_\bul$ is Waldhausen's S-construction. The level $1$ is $NwS_1C$ which is equivalent to $NwC$. This $\del$-space is not special in general. Algebraic K-theory is indeed a way to make it very special. The algebraic K-theory space of $C$ is defined by the pointed simplicial set 
$$K(C):=\Omega\re{NwS_\bul C},$$
where $\Omega$ means $\Omega_1$ in the notation of \ref{remgroup}, i.e. the simplicial set of loops. The basepoint is taken to be the zero object of $C$. Algebraic K-theory defines a functor 
$$K : WCat \lmo SSet_\ast$$
from $\V$-small Waldhausen categories to $\V$-small simplicial sets. By Segal's theorem (see Rem.\ref{remgroup}) the algebraic K-theory of $C$ is the level $1$ of the group completion 
$$(mon \, \K_\bul (C))^+\simeq \R\Omega_\bul \re{NwS_\bul C}.$$
Moreover one can verify directly using adjoint functors that 
$$\pi_0 (mon \, \K_\bul (C))^+_1\simeq (mon \, \pi_0 \K^{(1)} (C))^+.$$
The free monoid of $\pi_0 \K^{(1)} (C)$ is the monoid in which we identify $a$ with the product of $a'$ and $a''$ each time there is a cofibration sequence $a'\moi a\twoheadrightarrow  a''$. It follows that this product is commutative and coincides with the sum in $C$, and that $\pi_0 K(C) =K_0(C)$ is the Grothendieck group of $C$. 
\end{ex} 

\vspace*{.4cm}

We recall the equivalence between the homotopy theory of very special $\gam$-spaces and the homotopy theory of connective spectra. This was first proven by Segal (\cite{seg}) and upgraded in the language of model categories in \cite{bf}. \cite[Thm 5.8]{bf} can be directly generalised from $\gam$-spaces and spectra to $\gam$-objects in $M=SPr(C)$ and presheaves of spectra on $C$. Moreover, following \cite[Ex 2.39]{schw-sym}, we can replaced ordinary spectra by symmetric spectra. We denote by $\spcon(C)$ the subcategory of presheaves of connective spectra. We have a pair of adjoint functor 
$$\xymatrix{ \gam-SPr(C) \ar@<2pt>[r]^-{\mcal{B}} & \ar@<2pt>[l]^-{\mcal{A}} \spcon(C) }$$
Recall from from \cite[§5]{bf} that a $\gam$-space can be extend to a functor from symmetric spectra to symmetric spectra. The functor $\mcal{B}$ is defined on an object $E\in \gam-SPr(C)$ by
$$\mcal{B}E=E(\s),$$
the value of $E$ on the sphere spectrum, which is a connective spectrum for every $\gam$-object $E$. This functor is really identical to Segal's functor from special $\gam$-spaces to spectra, defined using iterations of realization of simplicial spaces. The functor $\mcal{B}$ preserves weak equivalences between all $\gam$-spaces, not just cofibrants. We endow the category $Sp(C)$ of presheaves of symmetric spectra on $C$ with the projective model structure and we denote by $Ho(\spcon(C))$ the subcategory of $Ho(Sp(C))$ consisting of connective symmetric spectra. 

\begin{theo}\label{bf} The adjoint pair $(\mcal{B},\mcal{A})$ is a Quillen pair for the very special model structure on $\gam-SPr(C)$. Moreover it is a Quillen equivalence, inducing an equivalence of categories,
$$\xymatrix{ Ho(\gam-SPr(C)^{vsp} ) \ar@<2pt>[r]^-{\lef \mcal{B}}  & Ho(\spcon(C)) \ar@<2pt>[l]^-{\R \mcal{A} } }.$$
\end{theo}

\begin{rema}\label{kconspec}
Following \cite[§1.2]{teze}, we give a model for the algebraic K-theory spectrum of a Waldhausen category. This model has the advantage of being canonically delooped using the functor $\mcal{B}$ of theorem \ref{bf}, compared to the way Waldhausen has defined his K-theory spectrum. We just remind the construction and for more details and a comparison with Waldhausen's spectrum we refer to \cite[§1.2]{teze}. 

Let $C$ be any Waldhausen category. The axioms setting the structure of a Waldhausen category imply that $C$ has finite sums. We can therefore construct a special $\gam$-object in the category $WCat$ of $\V$-small Waldhausen categories denoted by $B_W C$ such that there is an equivalence of categories $(B_W C)_1 \simeq C$ and the composition law is given by the sum in $C$. Moreover there is an equivalence of categories $(B_W C)_n\simeq C^n$ for any $n\geq 1$. Since the algebraic K-theory space functor $K:WCat \lmo SSet_\ast$ commutes with finite products, we obtain a special $\gam$-space by taking $K$ levelwise : 
$$K^\gam(C):=K(B_W C).$$
Because $\pi_0K^\gam(C)_1\simeq K_0(C)$ is a group, the $\gam$-object $K^\gam(C)$ is very special and thus gives a connective symmetric spectrum
$$\kc(C):=\mcal{B} (K^\gam(C)).$$
This defines a functor 
$$\kc:WCat\lmo \spcon.$$
For any Waldhausen category $C$, there exists a map of special $\gam$-spaces, 
\begin{equation}\label{mapgam}
NwB_W C \lmo K^{\gam} (C)
\end{equation}
given by the adjoint of the map $S^1\sm  NwB_W C \lmo \re{NwS_\bul B_W C}$. 
We recall (see \cite[Lem 1.10]{teze}) that in the special case where all cofibrations are split in the Waldhausen category $C$, then the map of very special $\gam$-spaces,
$$(NwB_W C)^+ \lmo K^{\gam} (C)$$
is a levelwise equivalence. 
\end{rema}



\subsection{Algebraic K-theory of noncommutative spaces}\label{algk}

\emph{In all the sequel, $k$ is a commutative associative unital noetherian base ring.}

\vspace*{.2cm}

Actually our definition of topological K-theory is developed only over the complex field $\C$. Thus from §\ref{rea}, the base ring will be $\C$. We fix some notations for dg-categories. We denote by $C(k)$ the category of $\U$-small unbounded cochain complexes of $k$-modules. We denote by $\dgcat$ the category of $\U$-small $k$-dg-categories, i.e. of $\U$-small $C(k)$-enriched categories, and $C(k)$-enriched functors between them. By convention, the expression dg-category always refers to a $\U$-small $k$-dg-category. 

We endow $\dgcat$ with the standard model structure \cite[Thm 1.8]{tabth} for which equivalences are quasi-equivalences, we denote by $\dgcat$ this model structure. We denote by $\dgmor$ the Morita model structure on $\dgcat$ \cite[Thm 2.27]{tabth}, for which the equivalences are Morita equivalences. If $T$ is a dg-category, we denote by $T^{op}$ its opposite dg-category, and by $T^{op}-Mod$ the category of $T^{op}$-dg-modules, endowed with its projective model structure. Its homotopy category is denoted by $D(T)$ and called the derived category of $T$. A map $T\lmo T'$ in $\dgcat$ is a Morita equivalence if the induced map $D(T)\lmo D(T')$ is an equivalence of categories.

We consider $\parf(T)$ the category of cofibrant and perfect objects in $T^{op}-Mod$, i.e. the sub-category of $T^{op}-Mod$ consisting of cofibrant $T^{op}$-dg-modules which are perfect (or compact) as objects of the derived category $D(T)$. We endow $\parf(T)$ with a structure of a Waldhausen category induced by the model structure of $T^{op}-Mod$, i.e. a map is a weak equivalence (resp. a cofibration) in $\parf(T)$ if it is so in $T^{op}-Mod$. The axioms of a Waldhausen category structure are satisfied essentially because the homotopy pushout of two perfect dg-modules over a third perfect dg-module is again perfect. Moreover, the Waldhausen category $\parf(T)$ satisfies the saturation axiom, the extension axiom, has a cylinder functor which satisfies the cylinder axiom. 

Let $f:T\lmo T'$ be a map in $\dgcat$, then $f$ induces a Quillen pair
$$\xymatrix{ T^{op}-Mod \ar@<2pt>[r]^-{f_!}  &T'^{op}-Mod \ar@<2pt>[l]^-{f^*} }$$
where $f^*$ is defined on objects by composition with $f$. As a left Quillen functor, the direct image $f_!$ preserves perfect dg-modules and induces an exact functor still denoted by $f_!$
$$f_!:\parf(T)\lmo \parf(T').$$
This defines a lax functor $\dgcat \lmo WCat$ and applying the canonical strictification procedure we obtain a functor
$$\parf : \dgcat \lmo WCat.$$

\begin{df} 
\begin{description}
\item[a)] The \emph{algebraic K-theory space of a dg-category $T$} is the pointed simplicial set $K(T):=K(\parf(T))$. This defines a functor 
$$K : \dgcat \lmo SSet_\ast.$$

\item[b)] The \emph{connective algebraic K-theory spectrum of a dg-category $T$} is the spectrum $\kc(T):=\kc(\parf(T))$. This defines a functor
$$\kc: \dgcat\lmo \spcon.$$
\end{description}
\end{df}

\begin{rema}\label{algebra}
If $T=A$ is any associative $k$-algebra, one can consider vector bundles on $Spec(A)$, or in other words, projective (right) $A$-modules of finite type. This forms a Waldhausen category $Vect(A)$ with weak equivalences being isomorphisms and cofibrations being monomorphisms. One can show (using \cite[Thm 1.11.7]{tt}) that there is a weak equivalence of simplicial sets $K(Vect(A))\simeq K(\parf(A))$, and thus a weak equivalence on the associated connective K-theory spectra too. 
\end{rema}

We now recall the main properties of connective K-theory : filtered colimits, Morita invariance and additivity on split short exact sequences of dg-categories. 

Recall that 
\begin{itemize}
\item a short sequence $T'\mos{i} T\mos{p} T''$ of dg-categories is called \emph{exact} if the sequence of triangulated perfect derived categories 
$\xymatrix{D_{pe}(T')\ar[r]^-{i_!} & D_{pe}(T)\ar[r]^-{p_!} &D_{pe}(T'')}$ is exact, i.e. that $i_!$ is fully faithfull and that $p_!$ induces an equivalence up to factors between $D_{pe}(T'')$ and the Verdier quotient $D_{pe}(T)/D_{pe}(T')$.
\item  A short sequence $T'\mos{i} T\mos{p} T''$ of dg-categories is called \emph{strictly split exact} if it is exact and if moreover the functor $i_!$ has a right adjoint denoted by $i^*$, the functor $p_!$ has a right adjoint denoted by $p^*$ such that $i^*i_!\simeq id_{\parf(T')}$ and $p_!p^*\simeq id_{\parf(T'')}$ via the adjunction maps. 
\item A short sequence $T'\mos{i} T\mos{p} T''$ of dg-categories is called \textit{split exact} if it is isomorphic in $Ho(\dgmor)$ to a strictly split short exact sequence. 
\end{itemize}

\begin{prop}\label{split}
\begin{description}
\item[a.] The functor $\kc$ commutes with filtered homotopy colimits in $\dgcat$. 
\item[b.] The functor $\kc$ sends Morita equivalences in $\dgcat$ to isomorphisms in $Ho(Sp)$.  
\item[c.] Let $T'\mos{i} T\mos{p} T''$ be a \textbf{split} short exact sequence of dg-categories. Then the morphism 
$$i_!+p^\ast :\kc(T')\oplus \kc(T'')\lmo \kc(T)$$
is an isomorphism in $Ho(Sp)$.
\end{description}
\end{prop}

\begin{proof} 
These facts are well known, but for details see \cite[Prop 2.8]{teze}. 
\end{proof}

We recall how to define nonconnective algebraic K-theory of dg-categories using Tabuada--Cisinski construction \cite{nck}. This definition coincides with Schlichting's construction thanks to \cite[Prop 6.6]{nck}.

The main ingredient of nonconnective K-theory is the countable sum completion functor or flasque envelope : 
$$\mcal{F} : \dgcat\lmo \dgcat,$$
(see \cite[§6]{nck} for this construction). It comes with a quasi-fully faithfull functor $T\lmo \mcal{F}(T)$.  The essential property of the dg-category $\mcal{F}(T)$ is that it admits countable sums, and thus verifies $\kc(\mcal{F}(T))=0$. One can define the suspension functor of dg-categories 
$$\mcal{S} : \dgcat\lmo \dgcat,$$
by $\mcal{S}(T):=\mcal{F}(T)/T$, where the quotient is calculated in $Ho(\dgcat)$. 
The sequence of spectra $(\kc(\mcal{S}^n(T))_{n\geq 0}$ forms actually a spectrum in the monoidal category $Sp$ of symmetric spectra (see \cite[Prop 7.2]{nck}) and we take the $0$th-level of the associated $\Omega$-spectrum to define nonconnective K-theory. 

\begin{df}(\cite[Prop 7.5]{nck})
The \emph{nonconnective algebraic K-theory of a dg-category $T$} is the spectrum 
$$\kn(T):=hocolim_{n\geq 0} \kc(\mcal{S}^n (T))[-n].$$
This defines a functor 
$$\kn :\dgcat \lmo Sp.$$
\end{df}

By definition, for every $T\in \dgcat$ we have a natural map $\kc(T)\lmo \kn(T)$. The main properties of nonconnective K-theory are the following. 

\begin{prop}\label{propalgk}
\begin{description}
\item[a.] For every triangulated dg-category $T$, the natural map $\kc(T)\lmo \kn(T)$ induces an isomorphism on $\pi_i$ for all $i\geq 0$. Therefore $\kc(T)$ is the connective covering of the spectrum $\kn(T)$. 
\item[b.] The functor $\kn$ commutes with filtered homotopy colimits in $\dgcat$. 
\item[c.] The functor $\kn$ sends Morita equivalences to isomorphism in $Ho(Sp)$. 
\item[d.] Let $T'\mos{i} T\mos{p} T''$ be an exact sequence of dg-categories. Then the induced sequence
$$\xymatrix{ \kn(T') \ar[r]^{i_!} & \kn(T) \ar[r]^{p_!} & \kn(T'')}$$
is a distinguished triangle in $Ho(Sp)$. 
\end{description}
\end{prop}

In fact, it can be shown that there is essentially a unique way to extend $\kc$ to a nonconnective invariant satisfying all these properties (see \cite[Thm 1.9]{marco2}).

\begin{proof} 
These facts are well known and result from Schlichting's theory. For details and references, see \cite[Prop 1.15]{teze}. 
\end{proof}

\begin{nota}\label{notaksch}

If $X$ is a scheme of finite type over $k$, we denote by $\lpe(X)$ the dg-category of perfect complexes of quasi-coherent $\ocal_X$-modules. Following \cite[8.3]{dgmor}, this dg-category can be defined starting with the injective $C(k)$-model category $C(QCoh(X))$ of complexes of quasi-coherent sheaves of $\ocal_X$-modules on $X$ (in the sense of \cite{hoveysh}) and then taking the associated $k$-dg-category of cofibrant-fibrant object $\lqc(X):=Int(C(QCoh(X)))$. This latter dg-category is only $\V$-small, but its full sub-dg-category consisting of perfect complexes which is defined to be $\lpe(X)$ is equivalent to a $\U$-small dg-category. We recall that a complex of quasi-coherent $\ocal_X$-modules is said to be perfect if it is locally quasi-isomorphic to a bounded complex of locally free $\ocal_X$-modules. These objects are the compact objects of the derived category $D_{\mrm{qcoh}}(X)$ of quasi-coherent sheaves on $X$. 

The assignment $X\mapsto \lqc(X)$ can be arranged into a functor into $\V$-small dg-categories, 
$$\lqc : \schk^{op}\lmo\dgcat^{\V}.$$
In order to obtain such a functor, we adopt the definition of quasi-coherent modules given in \cite[§1.3.7]{hag2}. The definition of perfect complexes is easily formulated in this setting, these are collections of quasi-coherent complexes which are locally quasi-isomorphic to a bounded complex of (locally) free modules. 
In this way, for any morphism $f:X\lmo Y$ in $\schk$ the pullback functor $f^*:\lqc(Y)\lmo \lqc(X)$ preserves perfect complexes and we obtain a map $f^*:\lpe(Y)\lmo \lpe(X)$. This defines a functor 
$$\lpe:\schk^{op}\lmo \dgcat.$$
Following Schlichting \cite[§6.5]{schl} and the work of Thomason--Trobaugh, we define the algebraic K-theory of a scheme $X\in \schk$ to be the algebraic K-theory of its dg-category of perfect complexes $\kn(X):=\kn(\lpe(X))$. This defines a functor
$$\kn : \schk^{op} \lmo Sp.$$
Schlichting proved (Thm 5 of loc.cit) that for any quasi-compact quasi-separated scheme $X$ over $k$, the algebraic K-theory $\kn(X)$ coincides with Thomason--Trobaugh nonconnective K-theory of $X$ \cite[Def 6.4 p.360]{tt}. 
\end{nota}



\subsection{Algebraic Chern character}\label{caracalg}

We set notations for cyclic homology of noncommutative spaces. All the different versions of cyclic homology are defined out a single object called the mixed complex associated to a dg-category which was defined by Keller (\cite{kelcyclic}). The mixed complex is a functor  
$$Mix=Mix(-\mid k): \dgcat\lmo \Lambda-Mod,$$
with values in the category of $\Lambda$-dg-modules, where $\Lambda$ is the $k$-dg-algebra generated by a single element $B$ in degree $-1$ satisfying the relation $B^2=0$ and $d(B)=0$. A $\Lambda$-dg-module is commonly called a mixed complex. If $T$ is a locally cofibrant $k$-dg-category, we define a precyclic complex of $k$-module in the sense of \cite[§2.1]{kelinv}, denoted by $C(T)$ by 
$$C(T)_n=\bigoplus_{x_0,x_1,\hdots,x_n\in T} T(x_n,x_0) \te_k T(x_{n-1},x_n)\te_k T(x_{n-2},x_{n-1})\te_k \hdots \te_k T(x_0,x_1) $$
where the sum runs over all sequences of cardinal $n+1$ of objects in $T$. The differential $d:C(T)_n\lmo C(T)_{n-1}$ is defined by 
$$d(f_n\te f_{n-1}\te \hdots \te f_0) = f_{n-1}\te \hdots \te f_1\te f_0 f_n + \sum _{i=1}^n (-1)^n f_n\te \hdots\te f_{i-1}f_i \te \hdots \te f_0,$$
where the $f_i$ are homogeneous elements. The cyclic operator $t:C(T)_n\lmo C(T)_n$ is defined by 
$$t(f_n\te \hdots \te f_0) = (-1)^{n+1} f_0\te f_n\te \hdots \te f_1.$$
The mixed complex associated to the precyclic complex $C(T)$ (in the sense of  §2.1 of loc.cit.) is by definition the mixed complex of $T$, denoted by $Mix(T)$. The object $Mix(T)$ is a relative invariant in the sense that it depends crucially on the ground ring. By abuse of notation we omit to mention the ground ring in the notation $Mix(T)$ and, by default, it refers to mixed complex and cyclic homology relatively to the ground field $k$ over which our dg-categories are defined. Keller has proved that $Mix$ is a localizing invariant in the sense of Tabuada \cite{tabth}. 

\begin{theo}\label{cyclic} \emph{(Keller \cite{kelcyclic})} --- The functor $Mix$ commutes with filtered homotopy colimits in $Ho(\dgcat)$, sends Morita equivalences to equivalences and sends exact sequences of dg-categories on distinguished triangles in the derived category of mixed complexes. 
\end{theo} 

We denote by 
$$H : C(\Z)\lmo Sp $$
the standard functor from (unbounded, cohomological) complexes of $\Z$-modules to symmetric spectra. We refer to \cite[after Cor B.1.8]{shischcat} or to \cite[§1.4.1]{teze} for a construction of $H$. We consider $C(\Z)$ and $Sp$ as monoidal categories with their usual monoidal structure given by the dg-tensor product and the smash product respectively. The functor $H$ is then a lax monoidal functor and induces a functor on monoids and modules over monoids. Therefore we have a ring spectrum $Hk$, an $Hk$-algebra spectrum $H\Lambda$ and an induced functor 
$$H: \Lambda-Mod\lmo \hlamb$$ 
with values in the category of $H\Lambda$-modules spectra. We endow $\Lambda-Mod$ and $\hlamb$ with the model structure given by \cite[sect 4]{ss-alg} where a map is an equivalence if it is an equivalence of the underlying modules and spectra respectively. By definition $H$ preserves equivalences. 

We denote by $\hh:=H\circ Mix$ the composite functor 
$$\xymatrix{ \dgcat \ar[r]^-{Mix} \ar@/^2pc/[rr]^-{\hh} & \Lambda-Mod \ar[r]^-{H} & \hlamb }$$
and we call it Hochschild homology over $k$. It is still a localizing invariant (i.e. it verifies the conditions of theorem \ref{cyclic}). It is also a lax monoidal functor because $H$ and $Mix$ are. In the following definition we consider $Hk$ as an $H\Lambda$-module via the natural augmentation $H\Lambda\lmo Hk$.

\begin{df} Let $T\in \dgcat$. 
\begin{itemize}
\item The \emph{negative cyclic homology of $T$} is the symmetric spectrum 
$$\hcn (T):=\rhomi_{H\Lambda} (Hk,\hh(T)).$$
\item The $Hk$-module $\hcn(T)$ is a module over the $Hk$-algebra $\rhomi_{H\Lambda} (Hk,Hk)\simeq Hk[u]$ with $u$ a generator of degree $-2$. We define \emph{the periodic cyclic homology of $T$} or \emph{periodic homology of $T$} as the symmetric spectrum
$$\hp(T):=\hcn(T)\wh_{Hk[u]} Hk[u,u^{-1}].$$
\end{itemize}
We obtain functors
$$\hcn : \dgcat\lmo Hk[u]-Mod_\s,$$
$$\hp : \dgcat\lmo Hk[u,u^{-1} ]-Mod_\s.$$
We have by definition a map of $Hk[u]$-modules $\hcn(T)\lmo \hp(T)$. 
\end{df}

\begin{rema}
This definition of periodic homology is not the standard definition given in \cite[p.210]{kasselcyclic}, but by \cite[Thm 6.1.24]{rosenbergk} we know that for an associative unital algebra, the two definitions coincide. 
\end{rema}

\begin{rema}\label{remhhsch}
The cyclic homology of schemes was defined and studied by Weibel \cite{weibelhcsch} using a sheafification of the Hochschild complex of an algebra. If $X$ is a scheme we denote by $\hh(X)$ Weibel's Hochschild homology of $X$ (the spectral version). Following Keller \cite[§1.10 p.10]{kelcyclic} we define the Hochschild homology of a scheme $X$ to be the Hochschild homology spectrum $\hh(\lpe(X))$ of its dg-category of perfect complexes (see remark \ref{notaksch}). By the comparison result \cite[Thm 5.2]{kellercyclicschemes}, we know that if $X$ is quasi-compact separated over a field $k$, then there exists a canonical isomorphism 
$$\hh(\lpe(X))\simeq \hh(X)$$
in the derived category of $\Lambda$-module spectra. Therefore there are no ambiguities on the meaning of the Hochschild homology of a separated scheme $X$ of finite type over a field $k$. 
\end{rema}

\begin{nota}\label{notapref}
Let $\mbf{E}:\dgcat\lmo V$ be a functor, and we define a new functor
$$\und{\mbf{E}} : \dgcat\lmo Pr(\affk, V)$$
by $\und{\mbf{E}}(T)(\spec(A))=\mbf{E}(T,A):=E(T\tel_k A):=E(QT\te_k A)$ where $Q:\dgcat\lmo \dgcat$ is a functor which satisfies the following properties : 
\begin{itemize}
\item There exists a natural transformation in $T$, $Q(T)\lmo T$ which is an equivalence (for the standard model model structure). 
\item For all $T\in \dgcat$, the dg-category $Q(T)$ is flat\footnote{A dg-category $T$ is flat over $k$ if for all $x,y\in T$, the $k$-modules $T(x,y)^n$ are flat $k$-modules for all $n$.} over $k$. 
\end{itemize}
Such a functor is for example given by a cofibrant replacement functor. In this paper, we are primarly interested in the case $k=\C$ complex numbers, so we could just take $Q$ to be the identity functor. We'll use this notation for all classical invariant of dg-categories $\mbf{E}=\kn, \kc, \hh, \hc, \hcn, \hp$. 
\end{nota}

\vspace*{.4cm}

We now use the motivic category of Tabuada (see \cite{tabhkt}, \cite{nck}) and both Tabuada's and Cisinski-Tabuada's theorem of corepresentability of K-theory inside the motivic category. Following Cisinski--Tabuada, these results enable us to define the algebraic Chern character in a $\ka$-linear fashion, where $\ka$ is the ring in presheaves of spectra given by $\ka (\spec(A))=\kn(A)$ the algebraic K-theory of the commutative algebra $A$. The corepresentability theorem gives us a natural ring structure on $\ka$ and of $\ka$-module on $\ukn(T)$ for all dg-category $T$. 

The universal localizing motivator $\mcal{M}_{loc}(k)$ of Tabuada (\cite[Def 11.2]{tabhkt}) can be described (using \cite[Prop 12.4]{tabhkt}) as the derivator associated to the projective model category of presheaves of spectra on finite type dg-categories, localized with respect to Morita equivalences and short exact sequences. If $\mrm{Hodgcat}$ stands for the derivator associated to $\dgmor$, there exists a morphism of derivators
$$\mcal{U}_l : \mrm{Hodgcat} \lmo \mcal{M}_{loc}(k) $$
called the universal localizing invariant, which is universal among functors which commute with filtered homotopy colimits, send Morita equivalences to equivalences and short exact sequences to triangles (see \cite[Thm 11.5]{tabhkt}). In \cite{tabcisym}, the authors build a symmetric monoidal structure on $\mcal{M}_{loc}(k)$ out of the Day convolution product and the derived tensor product of dg-categories, using the fact that a special class of finite type dg-categories are invariant under the latter product. Moreover this class must satisfy some flatness properties (see \cite{tabcisym} 7.1, properties a) to f)). In this paragraph, we just deal with the homotopy category $\mcal{M}_{loc}(k)(\ast)=Ho(\mloc)$. We state a purely categorical version of Tabuada's universal property, which is sufficient for our purpose. We also replace the category of symmetric spectra $Sp$ by the category $\spafk$ of presheaves of symmetric spectra. We will thus consider presheaves over dg-categories of finite type with values in the category $\spafk$. 

\vspace*{.4cm}

We denote by $\dgcattf\subseteq \dgcat$ a subcategory of $\U$-small dg-cateogories which verify properties a) to f) in \cite[§7.1]{tabcisym}. Roughly speaking, this means that any dg-category in $\dgcattf$ is of finite type in the sense of \cite[Def 2.4]{modob}, locally flat over $k$, and that the subcategory $\dgcattf$ is stable by tensor product and by the endofunctor $Q$ we fixed above in notation \ref{notapref}. We consider the category $Pr(\dgcattf, \spafk)$ of presheaves over $\dgcattf$ with values in $\spafk$. We endow the category $\spafk$ with the injective model structure for which we refer to the general existence theorem \cite[Prop A.2.8.2]{htt}. We endow the category $Pr(\dgcattf, \spafk)$ with the projective model structure. We denote by 
$$\uh : \dgcattf\lmo Pr(\dgcattf, \spafk)$$
the Yoneda embedding given by $\uh_T(T')(\spec(A))=\Hom_{\dgcat}(T',T\te_k A)$. Let $\mcal{R}$ be the set of maps in $Pr(\dgcattf, \spafk)$ which are of the form
\begin{enumerate}
\item $\uh_T\lmo \uh_{T'}$, with $T\lmo T'$ being a Morita equivalence. 
\item $\cone(\uh_i)\lmo \uh_{T''}$, for $\xymatrix{T'\ar[r]^-i & T\ar[r]^-p & T''}$ an exact sequence of dg-categories. 
\end{enumerate}
We define $\mloc:=L_{\mcal{R}} Pr(\dgcattf, \spafk)$ as the $\spafk$-enriched left Bousfield localization in the sense of \cite{barwenr}. Therefore $\mloc$ is a $\spafk$-model category in the sense of \cite[Def 4.2.18]{hoveymod}, and we denote by $\rhomi_{\mloc}$ its derived Hom enriched in $\spafk$. By \cite[§7.5]{tabcisym}, it is also a monoidal model category in the sense of \cite[Def 4.2.6]{hoveymod}, and we denote by $\sm$ its monoidal product. The unit is given by the object $\uk$ which is defined by $\uk(T)=\Hom_{\dgcat}(T, k)$ and constant as a presheave on $\affk$. The following proposition follows from general properties of left Bousfield localization (see \cite[Prop 3.3.18]{hirs}). 

\begin{prop}\label{propuni} 
The Yoneda embedding $\uh:\dgcattf\lmo \mloc$ has the following properties. 
\begin{enumerate}
\item For all $T\in \dgcat$, $\uh_T$ is cofibrant in $\mloc$. 
\item $\uh$ sends Morita equivalences to equivalences in $\mloc$. 
\item $\uh$ sends exact sequences to fibration sequences in $\mloc$. 
\end{enumerate}
Moreover $\uh$ is universal with respect to these three properties. More precisely, for all $\spafk$-model category $V$, the functor 
$$\uh^* : \Hom_!(\mloc, V) \lmo \Hom_*(\dgcattf, V)$$
is an equivalence of categories, where $\Hom_!$ is the category of $\spafk$-enriched left Quillen functors, and $\Hom_*$ is the category of functors which satisfy properties 1 to 3 above. 
\end{prop}

\begin{rema}
Let $V$ be a $\spafk$-model category and $E : \dgcat\lmo V$ a functor which send Morita equivalences to equivalences and commutes with filtered colimits. Then $E$ restricts to a functor $E_\mid : \dgcattf\lmo V$. The extension of $E_\mid$ to $\mloc$ is denoted by 
$$E_{\mid,!} : \mloc\lmo \spafk.$$
Every dg-category $T$ can be written as a filtrant colimit of dg-categories in $\dgcattf$. Therefore there exists a filtrant diagram $\{T_\alpha\}$ of dg-categories in $\dgcattf$ and an equivalence of dg-categories $T\simeq colim_\alpha T_\alpha$ in $\dgcat$. On the other hand, the object $\uh_T\in \mloc$ is defined by 
$$\uh_T(T')(A)=\Hom_{\dgcat}(T',T\tel_k A)$$ 
and represents the dg-category $T$ in $\mloc$. Since the standard Yoneda functor is equivalent to the homotopical Yoneda embedding (see \cite[Lem 4.2.2]{hag1}), we then have an equivalence  $\uh_T\simeq colim_\alpha \uh_{T_\alpha}$ in $\mloc$. We therefore have equivalences 
$$E_{\mid,!}(\uh_T)\simeq colim_\alpha E_{\mid,!}(\uh_{T_\alpha})\simeq colim_\alpha E_\mid (T_\alpha)\simeq E(T).$$
We conclude the extension $E_{\mid,!}$ coincides up to equivalence with $E$ on $\dgcat$. This can be rephrased in the following way. We denote by $\Hom_{Mor}(\dgcattf, V)$ the projective model category of functors which send Morita equivalences to equivalences, and by $\Hom_! (\dgcat, Sp)$ the projective model category of functors which commute with filtered colimits and send Morita equivalences to equivalences. Then the inclusion functor $\dgcattf\lmo \dgcat$ induces an equivalence of categories 
\begin{equation}\label{inddg}
Ho(\Hom_! (\dgcat, Sp))\lmos{\sim} Ho(\Hom_{Mor}(\dgcattf, Sp)).
\end{equation}
Every functor $E\in \Hom_{Mor}(\dgcattf, Sp)$ can be uniquely extended up to isomorphism to an object of $Ho(\Hom_! (\dgcat, Sp))$. By abuse of notation we denote by $\uhh$ the functor $\uhh_!$. 

\end{rema}

\begin{theo}\label{repk}\emph{(Cisinski--Tabuada \cite[Thm 7.16]{nck})} ---
For all $T\in \dgcattf$, there exists a functorial isomorphism in $T$,  
$$\rhomi_{\mloc}(\und{k}, \uh_T)\simeq \ukn(T),$$
\end{theo}

Because of formula (\ref{inddg}), the isomorphism of theorem \ref{repk} lifts to a unique isomorphism in $Ho(\Hom_! (\dgcat, Sp))$. 

We consider the object $\uh_k$ in $\mloc$. It is given by $\uh_k(T)(\spec(A))=\Hom_{\dgcat} (T,A)$ for all $T\in \dgcat$ and all $\spec(A)\in \affk$. The object $\uh_k$ admits a structure of monoid in $\mloc$ because by definition of the monoidal structure in $\mloc$ we have an isomorphism $\uh_k\sm \uh_k \simeq \uh_{k\tel k}$ in $\mloc$. For all $T\in \dgcattf$, we have an $\uh_k$-module $\uh_T$ with $\uh_k$-action given by the natural isomorphism $\uh_T\sm \uh_k \simeq \uh_{T\te_k k} \simeq \uh_T$. We then have : 
\begin{itemize}
\item A monoidal model category $\mcal{V}=\spafk$. 
\item Two $\spafk$-enriched monoidal model categories 
$$\mcal{M}=\mloc \quad \textrm{and} \quad \mcal{N}=Pr(\affk, \hlamb).$$
Units are denoted by $\unit_{\mcal{M}}$ and $\unit_{\mcal{N}}$. They are cofibrant objects in $\mcal{M}$ and $\mcal{N}$ respectively. Thus we have two monoidal functors $\mcal{V}\lmo \mcal{M}$ and $\mcal{V}\lmo \mcal{N}$ given by the product with units. Their right adjoints $\homi(\unit, -)$ ($\mcal{V}$-enriched Hom) are lax monoidal functors and send therefore monoids on monoids. 
\item A $\mcal{V}$-enriched left Quillen lax monoidal functor $F : \mcal{M}\lmo \mcal{N}$ given by $F=\und{\hh}$. (In this case $F(\unit_{\mcal{M}}) \simeq\unit_{\mcal{N}}$).
\item A cofibrant monoid $a=\uh_k$ in $\mcal{M}$ and a cofibrant $a$-module $m$ given by $m=\uh_T$. 
\item Two morphisms in $\mcal{V}$ given by the functoriality of $\lef F$ : 
$$t:\rhomi_{\mcal{M}} (\unit_{\mcal{M}}, a)\lmo \rhomi_{\mcal{N}}(F(\unit_{\mcal{M}}), F(a)),$$
$$u : \rhomi_{\mcal{M}} (\unit_{\mcal{M}}, m)\lmo \rhomi_{\mcal{N}}(F(\unit_{\mcal{M}}), F(m)),$$
where $\rhomi_{\mcal{M}}$ and $\rhomi_{\mcal{N}}$ are the $Ho(\mcal{V})$-enriched Homs. 
\end{itemize}

The general results of  \cite{ss-alg} on the existence of model structures on categories of monoids and modules require to verify a special axiom called the monoid axiom, which is not satisfied by the model category $\mloc$. However, by Hovey's results \cite[Thm 3.3, 2.1]{hoveymon}, monoids and modules objects in $\mloc$ can be endowed with nice homotopical properties, which are sufficient for our purpose. Roughtly speaking, monoids do not form a model category in the usual sense, but a so-called semi model category, and we can therefore consider the homotopy category of monoids. In this situation, if $M$ is a monoid, there exists a functorial cofibrant replacement $QM\lmos{\sim} M$ and if $M$ is cofibrant, there exists a functorial fibrant replacement $M\lmos{\sim} RM$. The homotopy category of monoids is then the quotient of cofibrant-fibrant monoids by the usual homotopy relation. We can therefore state the following. 

\begin{itemize}
\item The objects $\rhomi_{\mcal{M}} (\unit_{\mcal{M}}, a)$ and $\rhomi_{\mcal{N}}(F(\unit_{\mcal{M}}), F(a))$ are monoids.  
\item The object $\rhomi_{\mcal{M}} (\unit_{\mcal{M}}, m)$ is a $\rhomi_{\mcal{M}} (\unit_{\mcal{M}}, a)$-module. 
\item The object $\rhomi_{\mcal{N}}(F(\unit_{\mcal{M}}), F(m))$ is a $\rhomi_{\mcal{N}}(F(\unit_{\mcal{M}}), F(a))$-module and therefore a $\rhomi_{\mcal{M}} (\unit_{\mcal{M}}, a)$-module in $\mcal{V}$ via $t$. 
\item The map $u$ is a map of $\rhomi_{\mcal{M}} (\unit_{\mcal{M}}, a)$-modules. 
\end{itemize}
Applying this in our context, we obtain the following for any dg-category $T\in \dgcattf$. 

\begin{itemize}
\item A presheaf of ring spectra $\ka:=\ukn(\ast)\simeq \rhomi_{\mloc} (\uk, \uh_k)$. 
\item A presheaf of $\ka$-modules $\ukn(T)\simeq \rhomi_{\mloc} (\uk, \uh_T)$.
\item A presheaf of $\ka$-modules $\uhcn(T)=\rhomi (\uk, \uhh(T))\simeq \rhomi (\uhh(\uk), \uhh(\uh_T))$, where $\rhomi$ is the derived Hom of $Pr(\affk, \hlamb)$. 
\item A map of $\ka$-modules $\ukn(T)\lmo \uhcn(T)$. 
\end{itemize}
We denote by $\ka-Mod_\s$ the category of $\ka$-modules objects in $\spafk$. 

\begin{df}\label{defchernlin}
Let $T\in \dgcattf$. The algebraic Chern character associated to $T$ is the map in $Ho(\ka-Mod_\s)$,
$$\ch_T : \ukn(T)\lmo \uhcn(T)$$ 
which is defined above. Taking appropriate fibrant replacements, it defines a map in $Ho(\ka-Mod_\s^{\dgcattf})$, 
$$\ch : \ukn\lmo \uhcn.$$
By formula (\ref{inddg}), the map $\ch$ lifts to a unique map in $Ho(\ka-Mod_\s^{\dgcat})$ (up to isomorphism). 
\end{df}

\begin{rema}
There exists an additive version $\mota(k)$ of the model category $\mloc$ where we replace exact sequences of dg-categories by \emph{split} exact sequences. Connective K-theory then becomes corepresentable as a functor $\mota(k)\lmo \spafk$, and the same construction as above gives a connective Chern map,
$$\chc : \ukc \lmo \uhcn$$
in $Ho(\tilde{\ka}-Mod_\s^{\dgcat})$, where $\tilde{\ka}$ is the presheaf of connective K-theory. By construction, the map 
$$\xymatrix{\ukc\ar[r] &  \ukn \ar[r]^-{\ch} & \uhcn}$$ 
is isomorphic to $\chc$ in $Ho(Sp^{\dgcat})$. 
\end{rema}

\begin{rema}
By definition, the Chern map of definition \ref{defchernlin} coincide up to isomorphism in $Ho(\spafk^{\dgcat})$ with Tabuada Chern map as defined in \cite[p.4]{tabpro}, and we therefore know it coincides with the classic version of the Chern map by \cite[Thm 2.8]{tabpro}. 
\end{rema}

\section{Topological realization over complex numbers}\label{rea}

The topological realization of a simplicial presheaf defined over $\C$-schemes is a topological space functorially associated to this presheaf. In some way the topological realization of a simplicial presheaf is the analog of the "espace étalé" associated to a sheaf over a topological space. We begin by considering the functor which associated to any $\C$-scheme its space of complex points. This functor extends in a canonical way to all simplicial presheaves. The resulting functor is the topological realization, and could also be named Betti realization. Topological realization has already been studied by Simpson (\cite{simpsontop}), by Morel--Voevodsky (\cite{mv}) on the level of the motivic homotopy category, and by Dugger--Isaksen (\cite{hyptop}) where the authors proved its compatibility with the homotopy theory of stacks (in the étale or Nisnevich topology). In this section we set up definitions, properties and results about the topological realization. The most original part is the last one in which it is proven that the topological realization of a simplicial presheaf is unchanged if we first restrict this presheaf to smooth schemes. This uses a homotopical generalization of Deligne's cohomological proper descent. 

\subsection{Generalities about the topological realization}\label{genrel}

We recall that $\affc$ stands for the category of affine $\C$-schemes of finite type. We denote by 
$$sp : \affc \lmo Top$$
the functor which associates to an affine $\C$-scheme $X$ its topological space of complex points. Let $X=\spec(A)$ with presentation $A=\C[X_1,\hdots, X_n] /\mathfrak{a}$, with $\mathfrak{a}$ an ideal of the polynomial algebra $\C[X_1,\hdots, X_n]$. Then this presentation defines a closed immersion $i:X\hookrightarrow \A^n$ and $sp(X)=i(X)(\C)$ endowed with the induced transcendantal topology of the space $\A^n(\C)$. Composing with the singular complex functor $S : Top\lmo SSet$, we obtain a simplicial set valued realization denoted by $ssp$. We have the commutative triangle, 
$$\xymatrix{ \affc \ar[r]^-{sp} \ar[dr]_-{ssp} & Top \ar[d]^-{S} \\ & SSet }$$

\begin{df} 
We still denote by $ssp$ (resp. $sp$) and we call \emph{topological realization} the $SSet$-enriched left Kan extension
\footnote{
We briefly recall the notion of enriched left Kan extension. Let $V$ be a monoidal category, $F:C\lmo D$ a functor with $C$ any category and $D$ a $V$-enriched category, which is moreover cocomplete in the enriched sense, $G:C\lmo C'$ a functor with $C'$ another cocomplete $\V$-enriched category. The $V$-enriched left Kan extension of $F$ along $G$ is the unique $V$-enriched functor $\tilde{F} : C'\lmo D$ (up to isomorphism) which commutes to $\V$-enriched colimits and such that the triangle 
$$\xymatrix{ C\ar[r]^-G \ar[dr]_-F & C' \ar[d]^-{\tilde{F}} \\ & D }$$
commutes up to isomorphism. In our case, we don't need to consider general enriched colimits, but only colimits of the form (\ref{colimenr}) which appears below. 
}
of the functor $ssp : \affc\lmo SSet$ (resp. of the functor $sp$) along the Yoneda embedding $h:\affc\hookrightarrow SPr(\affc)$. We then obtain functors 
$$\xymatrix{ \affc \ar@{^{(}->}[r]^-h & SPr(\affc)\ar[r]^-{sp} \ar[dr]_-{ssp} & Top \ar[d]^-{S} \\ & & SSet  }$$
\end{df}

General properties of the functor $ssp$ are summarized in the following proposition. Analogous properties holds for $sp$. 

\begin{prop}\label{proprel}
\begin{enumerate}
\item The functor $ssp : SPr(\affc) \lmo SSet$ commutes with colimits and its right adjoint is the functor $R$ such that for any simplicial set $K$, $R(K) : X\mapsto \map (ssp(X), K)$. The adjoint pair $(ssp, R)$ is a Quillen pair for the global model structure on $SPr(\affc)$. 
\item The functor $ssp$ commutes with finite homotopy products. In particular $ssp$ is a symmetric monoidal functor for the cartesian model structure on $\spr$ and $SSet$. 
\item For all $K\in SSet$, and all $E\in SPr(\affc)$, there exists a canonical isomorphism $ssp(K\times E)\simeq K\times ssp(E)$. For all $K\in SSet_*$ and all $E\in SPr(\affc)_*$ there exists a canonical isomorphism $ssp(K\sm E)\simeq K\sm ssp(E)$.
\item Let $E\in \spr$. We denote by $\pi_0^{pr} E$ the presheaf obtained by applying $\pi_0$ levelwise. Then there exists a canonical isomorphism of sets $\pi_0 ssp(E)\simeq \pi_0 ssp(\pi_0^{pr} E)$. 
\end{enumerate}
\end{prop}

\begin{proof} 
\begin{enumerate}
\item $R$ is right adjoint to $ssp$ by definition. By its definition, $R$ sends (trivial) fibrations to (trivial) fibrations. The pair $(ssp, R)$ is therefore a Quillen pair. 
\item Comes from formula (\ref{colimenr}) of remark \ref{remcolim} below, and the fact that in a topos (here the topos of simplicial presheaves on $\affc$) products commute with colimits. 
\item The proof of the pointed case is similar of the unpointed one. It comes from the fact that $K\times E = \coprod_K E$ and that the functor $ssp$ is a left adjoint and therefore commutes with colimits. 
\item Let $A\in Set$. We have canonical isomorphisms
\begin{align*}
\Hom_{Set}(\pi_0 ssp(\pi_0^{pr} E), A) & \simeq \Hom_{SSet}(ssp(\pi_0^{pr} E), A) \\
& \simeq  \Hom_{\spr}(\pi_0^{pr} E, R(A)) \\
& \simeq \Hom_{\spr}(E, R(A)) \\
& \simeq \Hom_{SSet}(ssp(E), A)\\
& \simeq \Hom_{Set}(\pi_0 ssp(E), A).
\end{align*}
We conclude by the Yoneda lemma.   
\end{enumerate}
\end{proof}

\begin{rema}\label{remcolim}
Every preasheaves of sets $F\in Pr(\affc)$ can be written as a colimit of representable presheaves. More precisely, the map
\begin{equation}\label{yo}
\xymatrix{colim_{X\in \affc/F} h_X \ar[r]^-\sim & F}
\end{equation}
is an isomorphism of presheaves. Therefore, since $ssp$ commutes with colimits, we have the formula
\begin{equation}\label{formssp}
\xymatrix{ssp(F)\simeq colim_{X\in \affc/F} ssp(X)}.
\end{equation}
We'll use in \ref{pi0rel} a homotopical version of the previous fact. Now take $E\in \spr$ a simplicial presheaf. The formula (\ref{yo}) is not yet available. This comes from the very homotopical flavour of the object $E$ and $\affc/E$ should really be consider as an $\infty$-category in order to obtain a similar formula, with a colimit in a higher sense. An alternative way is to write $E$ as a more complicated colimit (see \cite[p.8]{duguni}). Indeed $E$ can be written as the standard coegalisator
\begin{equation}\label{colimenr}
\xymatrix{ E\simeq colim ( \coprod_{X\in \affc} E(X)\times h_X  & \coprod_{Y\lmo Z} E(Z)\times h_Y ) \ar@<3pt>[l] \ar@<-3pt>[l] }.
\end{equation}
And thus we have an equivalence
$$\xymatrix{ ssp(E)\simeq colim ( \coprod_{X\in \affc} E(X)\times ssp(X)  & \coprod_{Y\lmo Z} E(Z)\times ssp(Y) ) \ar@<3pt>[l] \ar@<-3pt>[l] }.$$
\end{rema}

\subsection{$\ao$-étale model structure}

Let $C$ be a Grothendieck site with topology denoted by $\tau$. Jardine (see \cite{jardsimp}, and also \cite{dhi}) has proven the existence of a $\tau$-local model structure on the category $SPr(C)$ of simplicial presheaves on $C$ such that the equivalences are precisely the local equivalences. Thanks to \cite{dhi} this model structure can be defined as the left Bousfield localization of the global model structure $SPr(C)$ with respect to the set of maps of the form
$$hocolim_{\del^{op}} h_{U_\bul} \lmo h_X$$
in $SPr(C)$, for all $X\in C$ and all $\tau$-hypercovering $U_\bul\lmo X$ in $C$. Morally, this construction forces simplicial presheaves to verify $\tau$-homotopical descent. We denote by $SPr(C)^\tau$ the $\tau$-local model structure. We'll use it in the particular case of $C=\affc$ with the étale topology, and call this model structure the étale local model structure. The homotopy category $Ho(\spret)$ is a model for the homotopy theory of stacks. 

We endow $\affc$ with the étale topology. Everything below is also valid for the Nisnevich topology. Following \cite{mv}, we can build the étale $\ao$-homotopy theory of schemes out of the site $\affc$. The étale $\ao$-homotopy category of schemes is by definition the left Bousfield localization of the global model category $\spr$ by the set of maps of the form 
\begin{enumerate} 
\item $hocolim_{\del^{op}}h_{U_\bul} \lmo h_X$, for $U_\bul\lmo X$ an étale hypercovering of a scheme $X\in \affc$. 
\item a projection $h_X\times h_{\ao}\lmo h_X$, for $X\in \affc$. 
\end{enumerate}
We denote by $\spretao$ this model structure and call it the \emph{$\ao$-étale structure} on simplicial presheaves. Equivalences in $\spretao$ are called $\ao$-equivalences. We denote by 
$$sph : \spretao\lmo Top$$
the corresponding topological realization, and $ssph$ its simplicial analog. We denote by $Rh$ the right adjoint to $ssph$. An important property of the topological realization is the following.

\begin{theo}\label{di} \emph{(Dugger--Isaksen \cite[Thm 5.2]{hyptop})}  --- The adjoint pair $(sph, Rh)$ is a Quillen pair
$$\xymatrix{ \spretao \ar@<3pt>[r]^-{sph}  & Top \ar@<3pt>[l]^-{Rh} },$$
for the $\ao$-étale model structure. More precisely the functor $sp$ sends relations 1. and 2. (which defines the $\ao$-étale model structure) to equivalences of spaces. 
\end{theo}

\begin{rema}\label{remderrel}
\begin{itemize}
\item It follows from the above that the simplicial version $ssph : \spretao\lmo SSet$ is a left Quillen functor and preserves relations of type (1) and (2) which define $\spretao$. 
\item Another consequence of \ref{di} is that the derived functors 
$$\lef ssp:Ho(SPr(\affc)) \lmo Ho(SSet)$$
$$\textrm{et  }\lef ssph : Ho(\spretao)\lmo Ho(SSet)$$
verify $\lef ssp(E)\simeq \lef ssph(E)$ for all simplicial presheaf $E$. In other words, if $Q$ is a cofibrant replacement functor for the global model structure $SPr(\affc)$, then $\lef ssph (E)=ssp(QE)$. 
\end{itemize}
\end{rema}

\begin{df} The \emph{derived topological realization} is the functor $\lef ssp$ (or equivalently $\lef ssph$) and is denoted by $\re{-}$. We then have a commutative triangle,
$$\xymatrix{ Ho(\spr) \ar[rd]_-{\re{-}} \ar[rr]^-{\R id} && Ho(\spretao) \ar[ld]^-{\re{-}} \\ & Ho(SSet) &  }$$
We denote by $\ret{-}$ its topological analog with values in $Ho(Top)$. 
\end{df} 

\begin{prop}\label{proprelder}
\begin{enumerate}
\item The derived topological realization $\re{-}$ commutes with homotopy colimits. Its right adjoint is denoted by $(-)_B$. For all $K\in SSet$ and all $X\in\affc$, we have $K_B(X)=\R\map(ssp(X), K)$. 
\item The functor $\re{-}$ commutes with homotopy finite products. 
\item For all $K\in SSet$, and all $E\in \spretao$, we have a canonical isomorphism $K\times \re{E}\simeq \re{K\times E}$ in $Ho(SSet)$. For all $K\in SSet_*$ and all $E\in \spretao_\ast$, we have a canonical isomorphism $K\sm \re{E}\simeq \re{K\sm E}$ in $Ho(SSet_*)$. 
\end{enumerate}
\end{prop} 

\begin{proof} 
Proofs are essentially the same as for the non-derived case, and we just omit them. 
\end{proof}

\begin{nota}
By abuse of notation, we will omit the word derived in front of topological realization, even if we always mean the derived topological realization throughout. 
\end{nota}

\subsection{$\pi_0$ of the topological realization}\label{pi0rel}

We give an explicit description of the set $\pi_0\re{E}$ for any simplicial presheaf $E\in \spr$, which will be used below. This description is based on the formula (\ref{yo}) given by Yoneda.

\begin{lem}\label{pi0relens}
Let $F\in Pr(\affc)$ be a presheaf of sets. Then there exists a canonical isomorphism of sets,
$$\pi_0 ssp(F)\simeq F(\C) / \sim $$
where $\sim$ is the equivalence relation defined by $[x]\sim [y]$ if there exists a connected algebraic curve $C$, a map $f:h_C\lmo F$ in $Ho(\spr)$,  and two complex points $x',y'\in C(\C)$ such that $f(x')=x$ and $f(y')=y$. 
\end{lem}

\begin{proof} As there exists an isomorphism $\pi_0 ssp(F)\simeq \pi_0 sp(F)$, where $sp$ is the realization with values in $Top$, it suffices to prove the statement for $\pi_0 sp(F)$. From formula (\ref{yo}) in remark \ref{remcolim}, there exists a natural map $F(\C)\lmo sp(F)$ in $Top$ and therefore a map $F(\C)\lmo \pi_0sp(F)$ by taking $\pi_0$. We then have a commutative square
$$\xymatrix{ F(\C) \ar[r] \ar[d]_-\wr & \pi_0 sp(F) \ar[d]^-\wr \\ colim_{X\in \affc/F} X(\C) \ar[r] &  colim_{X\in \affc/F} \pi_0 sp(X) }$$
where the left vertical arrow is obtained by taking complex points in (\ref{yo}), and the right vertical arrow is obtained by taking $\pi_0 sp$ in formula (\ref{yo}). These two arrows are therefore isomorphisms. For all $X\in \affc$, the map $X(\C)\lmo \pi_0sp(X)$ is surjective. It follows that the top map in the square is surjective. Now let  $a,b\in F(\C)$ be two complex points having the same image under $F(\C)\lmo  \pi_0 sp(F)$. In $colim_{X\in \affc/F} X(\C)$, $a$ and $b$ correpond to two affine schemes $(X,x)$ and $(Y,y)$ above $F$, each provided with a complex point. By hypothesis, these two elements have the same image in $colim_{X\in \affc/F} \pi_0sp(X)$. Therefore there exists an affine scheme $Z\lmo F$ above $F$, and two maps $p:X\lmo Z$ and $q:Y\lmo Z$, and a continuous path between $p(x)$ and $q(y)$ in the space $sp(Z)$. It is a classical fact in algebraic geometry that, using Bertini theorem, we can deduce the existence of a connected algebraic curve $g:C\lmo Z$ above $Z$, two complex points $x',y'\in C(\C)$ such that $g(x')=p(x)$ and $g(y')=q(y)$. By composition we obtain a map $f:h_C\lmo F$ such that $f(x')=a$ and $f(y')=b$. The proof is then complete. 
\end{proof}

\begin{prop}\label{pi0cor}
Let $E\in \spr$ be a simplicial presheaf. Then there exists a canonical isomorphism of sets, 
$$\pi_0 \re{E} \simeq \pi_0(E(\C))/\sim$$
where $\sim$ is the equivalence relation defined by $[x]\sim [y]$ if there exists a connected algebraic curve $C$, a map $f:h_C\lmo F$ in $Ho(\spr)$,  and two complex points $x',y'\in C(\C)$ such that $f(x')=x$ and $f(y')=y$. 
\end{prop}

\begin{proof} 
We denote by $\pi_0^{pr}E$ the presheaf obtained by applying $\pi_0$ levelwise to the presheaf $E$. We apply the preceding lemma \ref{pi0relens} to the presheaf of sets $F=\pi_0^{pr}E$. We obtain an isomorphism $\pi_0 ssp(\pi_0^{pr} E) \simeq \pi_0^{pr} E(\C) /\sim \simeq (\pi_0 E(\C))/\sim$. By \ref{proprel}, there is an isomorphism $\pi_0 ssp(\pi_0^{pr} E)\simeq \pi_0 ssp(E)$. However, by remark \ref{remderrel}, there is an isomorphism $\pi_0(ssp(E))\simeq \pi_0\re{E}$, and the proof is complete.  
\end{proof}

\subsection{Topological realization of structured presheaves}\label{prestru}

\paragraph{Strict groups.}

Let $SGp_\C$ (resp. $SGp$) be the category of strict group objects in the cartesian monoidal category $\spr$ (resp. in the cartesian monoidal category $SSet$). We denote by
$$\mrm{B}:SGp\lmo SSet$$
the classifying space functor. For all $G\in SGp$, the space $\mrm{B}G$ is defined as the homotopy colimit
$$\mrm{B}G:=hocolim\, G^\bul$$
where $G^\bul$ is the bisimplicial set defined by $[n]\mapsto G^n$, with faces and degeneracies induced by product and neutral element in $G$. There exists a functorial cofibrant replacement of the object $G^\bul$ in the projective model category $SSet^{\del^{op}}$, which justifies that we consider $\mrm{B}$ as a functor between non-localized categories. For all $G\in SGp_\C$, we set $\mrm{B}G(\spec(A))=\mrm{B}(G(\spec(A))$. Because $\re{-}$ commutes with products, the topological realization of a presheaf of simplicial groups is a simplicial group. We consider the diagram
$$\xymatrix{Ho(SGp_\C) \ar[r]^-{\mrm{B}} \ar[d]^-{\re{-}} & Ho(\spr) \ar[d]^-{\re{-}} \\ Ho(SGp) \ar[r]^-{\mrm{B}} & Ho(SSet)  }$$

\begin{prop}\label{gp}
For all $G\in SGp_\C$, there exists a canonical isomorphism $\mrm{B}\re{G}\simeq \re{\mrm{B}G}$ in $Ho(SSet)$. 
\end{prop}

\begin{proof}
This follows directly from the fact that $\re{-}$ commutes with finite homotopy products and homotopy colimits. 
\end{proof}

\paragraph{Symmetric spectra.}\label{realsp}

We extend the topological realization functor $ssp :\spr \lmo SSet$ to presheaves of symmetric spectra by taking the $Sp$-enriched left Kan extension of the composite functor
$$\xymatrix{ \spr\ar[r]^-{ssp} & SSet\ar[r]^-{\sinf(-)_+} & Sp }$$
along the infinite suspension functor $\sinf(-)_+ : \spr\lmo \spaf$. We denote by $\sspsp$ this extension. We obtain a commutative diagram
$$\xymatrix{ \spr \ar[r]^-{ssp} \ar[d]_-{\sinf(-)_+} & SSet\ar[r]^-{\sinf(-)_+} & Sp  \\ \spaf \ar@/_1pc/[rru]_-{\sspsp} }.$$
The right adjoint to $\sspsp$ sends a symmetric spectrum $E$ to the presheaf $X\mapsto \homi_{Sp}(\sinf X_+, E)$. We have the same properties than with the simplicial realization. 
\begin{prop}\label{proprelsp}
\begin{itemize}
\item For all $F\in Sp$ and all $E\in \spaf$ we have a canonical isomorphism of spectra $F\sm \sspsp(E)\simeq \sspsp(F\sm E)$.
\item For all $E,E'\in \spaf$, we have a canonical isomorphism of spectra $\sspsp(E\sm E') \simeq \sspsp(E)\sm \sspsp(E')$, i.e. the functor $\sspsp$ is monoidal. 
\end{itemize}
\end{prop}

We denote by $\spetao$ the $\ao$-étale model structure on the category $\spaf$. Proposition \ref{di} implies that the functor
$$ssph_\s:\spetao\lmo Sp$$ 
is left Quillen for the $\ao$-étale model structure. Remark \ref{remderrel}  applies also for presheaves of spectra, that is for all $E\in \spaf$ we have a canonical isomorphism $\lef \sspsp (E)\simeq \lef ssph_\s (E)$. 

\begin{df} We call \emph{spectral topological realization} and we denote by $\resp{-}$ the two derived functors $\lef \sspsp$ and $\lef ssph_\s$. We have a commutative triangle,
$$\xymatrix{ Ho(\spaf) \ar[rd]_-{\resp{-}} \ar[rr]^-{\R id} && Ho(\spetao) \ar[ld]^-{\resp{-}} \\ & Ho(Sp) & }$$
We denote by $\hbs$ the right adjoint to $\resp{-}$. For all $E\in Ho(Sp)$, $\hbs(E)$ is the presheaf
$$\hbs(E) : X\longmapsto \rhomi_{Ho(Sp)} (\resp{X}, E).$$
\end{df} 

The notation $\hbs$ stands for Betti cohomology. Indeed if $E=H\C$ is the Eilenberg--Mac-Lane spectrum of $\C$, then the presheaf $\hbs(H\C)$ is nothing but Betti cohomology. Since $\sspsp$ is a monoidal functor, the spectral topological realization is also a monoidal functor.

\paragraph{$\del$ et $\gam$-espaces.}

Finally, we extend the topological realization to $\del$-objects in $\spr$. Of course, the same can be done for $\gam$-objects. We apply notations of \ref{monoids} for model structures on $\del$-objects. Let $E\in \del-\spr$ be a $\del$-presheaf. The topological realization of $E$, denoted by $ssp_{\del}(E)$ is the $\del$-space defined by the formula $ssp_{\del}(E)_n=ssp(E_n)$. This defines a functor
$$ssp_{\del} : \del-\spr\lmo \del-SSet.$$
This functor admits as right adjoin the functor which associates to a $\del$-space $F$ the $\del$-presheaf $R(F)_n=R(F_n)$ (where $R$ is the right adjoint to $ssp$). The associated functor
$$ssph_{\del}:\del-\spretao \lmo \del-SSet$$ 
is left Quillen for the $\ao$-étale model structure. 

\begin{df} 
We call \emph{realization of $\del$-presheaves} and we denote by $\redel{-}$ the derived functors $\lef ssp_{\del}$ and $\lef ssph_{\del}$. We have a commutative triangle,
$$\xymatrix{ Ho(\del-\spr) \ar[rd]_-{\redel{-}} \ar[rr]^-{\R id} && Ho(\del-\spretao) \ar[ld]^-{\redel{-}} \\ & Ho(\del-SSet) & }$$
We denote by $\regam{-}$ its $\gam$-version. 
\end{df} 

The realization $\redel{-}$ commutes with finite homotopy products, and therefore the realization of a special (resp. very special) $\del$-object, is a special (resp. very special) $\del$-object. We have a diagram
$$\xymatrix{ Ho(\del-\spr) \ar[r]^-{\redel{-}}  \ar[d]^-{mon} & Ho(\del -SSet)\ar[d]^-{mon}\\ 
Ho(\del-\spr^{sp}) \ar[r]^-{\redel{-}}  \ar[d]^-{(-)^+} & Ho(\del -SSet^{sp})   \ar[d]^-{(-)^+} \\
Ho(\del-\spr^{tsp} )  \ar[r]^-{\redel{-}} & Ho(\del -SSet^{tsp}) }$$

\begin{prop}\label{proprelfree}
For all $E\in Ho(\del-\spr)$ we have a canonical isomorphism in $Ho(\del-SSet)$, 
$$\redel{mon(E)} \simeq mon \redel{E}.$$
\end{prop} 

\begin{proof} 
It suffices to show that their right adjoints commute. But it follows from the fact that if  $F\in \del-SSet$ is special, then the $\del$-object $\R R(F)=\R\map(\re{-}, F)$ is special. 
\end{proof}

\begin{prop}\label{proprelplus}
For all \emph{special} $\del$-presheaf $E\in Ho(\del-\spr)$ we have a canonical isomorphism in $Ho(\del-SSet)$, 
$$\redel{E^+} \simeq \redel{E}^+.$$
\end{prop} 

\begin{proof} 
It suffices to show that their right adjoints commute. But it follows from the fact that if  $F\in \del-SSet$ is very special, then the $\del$-object $\R R(F)=\R\map(\re{-}, F)$ is very special. 
\end{proof}

Now we compare topological realization of $\gam$-presheaves and of presheaves of connective spectra. We have a diagram, 
$$\xymatrix{Ho(\gam-\spr) \ar[r]^-{\regam{-}} \ar[d]^-{\mcal{B}} & Ho(\gam-SSet) \ar[d]^-{\mcal{B}}  \\ Ho(\spaf) \ar[r]^-{\resp{-}} & Ho(Sp) }$$

\begin{prop}\label{proprelb}
For all $E\in Ho(\gam-\spr)$, we have a canonical isomorphism in $Ho(Sp)$, 
$$\mcal{B} \regam{E} \simeq \resp{\mcal{B}E}.$$
\end{prop} 

\begin{proof} 
It follows from the fact that $\mcal{B}$ is a homotopy colimit, and because $\resp{-}$ commutes with homotopy colimits. 
\end{proof}

We recall from \ref{monoids} that there exists a functor
$$\alpha^*: Ho(\gam-\spr) \lmo Ho(\del-\spr).$$
We notice that by definition of  $\alpha^*$, for all $E\in \gam-\spr$, we have a canonical isomorphism 
$$\redel{\alpha^* E} \simeq \alpha^*\regam{E}.$$

\subsection{Restriction to smooth schemes}\label{sectrestliss}

We show here that the topological realization of a simplicial presheaf is unchanged if we restrict this presheaf to smooth schemes. It is a consequence of a result by Suslin--Voevodsky concerning the \emph{cdh}-topology, and of one of our main results : the homotopical version of Deligne's cohomological proper descent. This latter result is the analog for the proper topology of Dugger--Isaksen theorem (see theorem \ref{di}). To prove this descent result, we mimic Dugger--Isaksen's proof and we also use Lurie's proper descent theorem. 

The property of "restriction to smooth schemes" will be used below to show that the topological K-theory of a smooth complex algebraic variety coincides with the topological K-theory (in the dg-sense) of its dg-category of perfect complexes. It will also be used to show that negative semi-topological K-theory of a smooth commutative algebra is zero. 

We denote by $\afflissc$ the category of smooth affine $\C$-schemes of finite type. We denote by $l:\afflissc\hookrightarrow \affc$ the inclusion. The restriction on simplicial presheaves is denoted by $l^*:\spr\lmo \sprliss$. 

\begin{theo}\label{restliss} --- 
Let $F\in SPr(\affc)$. Then there exists a canonical isomorphism $\re{l^* F} \simeq \re{F}$ in $Ho(SSet)$. 
\end{theo}

The proof occupies the next two subsections. We will use the proper topology on the category $\schc$ of separated $\C$-schemes of finite type. The idea consists in using a Quillen equivalence between the proper local model category $\sprpro$ of simplicial presheaves defined over schemes, and the proper local model category $\sprproliss$ of simplicial presheaves defined over smooth schemes. Having proved this step, it remains to show that the topological realization behaves well with respect to the proper topology, which is our descent result or generalization of Deligne's proper descent. 

We denote by $\lissc$ the category of separated smooth $\C$-schemes of finite type and by 
$$\xymatrix{ \afflissc\ar@{^{(}->}[r]^-{l} \ar@{^{(}->}[d]_k & \affc \ar@{^{(}->}[d]^-j \\ \lissc\ar@{^{(}->}[r]^-i & \schc }$$
the inclusions. We have restriction/extension functors between the étale local model categories, 
$$\xymatrix{ \sprset \ar@<-3pt>[r]_-{i^*} \ar@<3pt>[d]^-{j^*} & \sprslisset \ar@<-3pt>[l]_-{i_!}   \ar@<3pt>[d]^-{k^*} \\ \spret \ar@<-3pt>[r]_-{l^*} \ar@<3pt>[u]^-{j_!}  & \sprlisset \ar@<3pt>[u]^-{k_!} \ar@<-3pt>[l]_-{l_!} }$$
All this adjoint pairs are Quillen pairs by the general results on change of sites \cite[Prop 7.2]{dhi}. The restriction functor $j^*$ preserves étale local equivalences and it is a Quillen equivalence. Indeed this can be deduced from the analog statement for sheaves using the projective local model structure on simplicial sheaves of \cite[Thm 2.1]{blande}. Theorem \ref{restliss} can be equivalently state in the category $\sprs$. Indeed for all $F\in \spr$, we have an equivalence $\re{\lef j_! F} \simeq \re{F}$ by definition of the realization and because $\lef j_!$ is a left adjoint. However for all $E\in \sprs$, the counit map $\lef j_! j^* E\lmo E$ is an étale local equivalence. The induced map $\re{\lef j_! j^* E} \lmo \re{E}$ is therefore an equivalence by theorem \ref{di}. We also have by definition of the realization an equivalence $\re{\lef j_! j^* E}\simeq \re{j^*E}$, and thus an equivalence $\re{j^*E}\simeq \re{E}$. Consequently we'll work in the category $\sprs$. 

We endow $\schc$ with the proper topology, i.e. a family of maps is declared to be a covering if each of the map is proper and if this family is cojointly surjective. We define the proper topology on the category $\lissc$ by saying that a sieve $S\subseteq h_X$ of an object of $\lissc$ is a covering sieve if it contains a sieve generated by a proper covering of $X$ in $\schc$. We recall that the category $\lissc$ does not have all pullbacks in general. 

\begin{rema}\label{hiro}
By Hironaka's theorem on resolution of singularities in characteristic zero \cite{hironakaresII}, every proper covering sieve $S$ of a scheme $X$ admits a refinement by a proper covering sieve $T\subseteq S$ generated by maps with source being smooth. Indeed for all map $Z\lmo Y$ in $S$, by Hironaka's theorem, there exists a smooth scheme $Z'$ and a proper surjective map $Z'\lmo Z$.
\end{rema}

\begin{nota} 
We denote by $\sprpro$ the proper local model category, i.e. the left Bousfield localization of the projective model category $\sprs$ with respect to maps $$hocolim_{\del^{op}} h_{Y_\bul} \lmo h_X$$ 
for all proper hypercovering $Y_\bul \lmo X$ in $\schc$. The following is our proper descent result. 
\end{nota}

\begin{prop}\label{quipro}
For every proper hypercovering $Y_\bul\lmo X$ of a scheme $X\in \schc$, the induced map,
$$hocolim_{\del^{op} }  \re{Y_\bul}\lmo \re{X}$$
is an isomorphism in $Ho(SSet)$. We deduce that the functor
$$ssp : \sprpro \lmo SSet$$
is left Quillen. 
\end{prop}

The proof of proposition \ref{quipro} is relatively non-trivial and is based on the purely topological fact that a proper hypercovering induces an equivalence if we take its colimit. It will be given below. We use proposition \ref{quipro} to prove Thm \ref{restliss}. 

The functor $i:\lissc\hookrightarrow \schc$ gives rise to a continuous functor on the level of sites (with the proper topology) and to a set of adjunctions between sheaves of sets in the proper topology,
$$\xymatrix{\shpro \ar[r]^-{i^\ast} & \shproliss \ar@<7pt>[l]^-{e} \ar@<-10pt>[l]_-{i_!}   }.$$
The functor $i^*$ is the restriction to smooth schemes. For every sheaf $F\in \shproliss$, the sheaf $eF$ is a priori the sheaf associated to the presheaf $X\longmapsto \Hom(i_\ast h_X, F)$. By \cite[Prop 7.2]{dhi}, the ajoint pair
$$\xymatrix{\sprpro \ar@<-3pt>[r]_-{i^\ast} & \sprproliss \ar@<-3pt>[l]_-{i_!}   },$$
is a Quillen pair. The following is a variant of \cite[Thm 6.2]{sv2}. 

\begin{prop}\label{eqpreliss}
 The Quillen pair $(i_!, i^*)$ is a Quillen equivalence and induces an equivalence of categories, 
$$\xymatrix{Ho(\sprpro) \ar@<-3pt>[r]_-{i^*} & Ho(\sprproliss) \ar@<-3pt>[l]_-{\lef i_!}   }.$$
\end{prop}  

\begin{proof} 
It suffices to prove that the categories of sheaves of sets are actually equivalent. Indeed by \cite[Thm 2.1]{blande} there exists a projective local model structure on the category of simplicial sheaves $SSh(\schc)^{\mrm{pro}}$ and by Thm 2.2 of loc.cit. the Quillen pair formed by the inclusion of simplicial sheaves into simplicial presheaves and by the sheafification is a Quillen equivalence,
$$\xymatrix{ SSh(\schc)^{\mrm{pro}} \ar@{^{(}->}@<-3pt>[r] & \sprpro \ar@<-3pt>[l]_{\und{a}}  }$$
Moreover the category $SSh(\schc)^{\mrm{pro}}$ is the category of simplicial objects in the topos $\shpro$, and its model structure depends only on the underlying topos $\shpro$. The same fact holds for the model category $SSh(\lissc)^{\mrm{pro}}$. 

From there, it suffices to prove that the adjoint pair
$$\xymatrix{\shpro \ar@<3pt>[r]^-{i^\ast} & \shproliss \ar@<3pt>[l]^-{e} }$$
is an equivalence. But this is a direct consequence of \cite[Exposé III, Thm 4.1]{sga4-1}. 
\end{proof}

\begin{lem}\label{preserve}
The restriction functor $i^* : \sprpro\lmo \sprproliss$ preserves all local equivalences.
\end{lem} 
\begin{proof} 
Because of Blander's result \cite[Thm 2.2]{blande}, it suffices to prove it for simplicial sheaves. Then it is a direct consequence of the fact proven above that $i^*$ induces an equivalence of categories $i^*:\shpro\lmos{\sim} \shproliss$. Indeed the local equivalences are exactly the maps which induce isomorphisms on all homotopy sheaves for all basepoint. Moreover, the construction of homotopy sheaves in a category of simplicial objects in a topos is made in purely categorical terms, and thus they only depend on the underlying topos. 
\end{proof}

\begin{proof} of \ref{restliss}  --- Since the following triangle
$$\xymatrix{ \lissc \ar[rr]^-i \ar[dr]_-{ssp} & & \schc \ar[dl]^-{ssp} \\ & SSet }$$
is commutative, the following triangle of left Kan extension
$$\xymatrix{ Ho(SPr(\lissc)) \ar[rr]^-{\lef i_!} \ar[dr]_-{\re{-}} & & Ho(SPr(\schc)) \ar[dl]^-{\re{-}} \\ & Ho(SSet) }$$
is commutative. This means that for all $G\in SPr(\lissc)$, there exists an isomorphism $\re{\lef i_! G}\simeq \re{G}$ in $Ho(SSet)$. Therefore if $G=\R i^* F$ for $F\in SPr(\schc)$, we have canonical isomorphisms
\begin{align*}
\re{i^* F} & \simeq  \re{\R i^* F} &&  \textrm{(by lemma \ref{preserve})} \\
& \simeq  \re{\lef i_! \R i^* F} &&  \textrm{(from the above)} \\
& \simeq \re{F}  && \textrm{(by proposition \ref{eqpreliss} and proposition \ref{quipro})} \\
\end{align*}
\end{proof}

\begin{proof} of proposition \ref{quipro} --- The first assertion implies that $ssp$ is left Quillen for the proper local model structure because of general facts about left Bousfield localizations.
 
Let $A=\ret{X}$ and $B_\bul=\ret{Y_\bul}$. The map of simplicial spaces $B_\bul \lmo A$ is a proper hypercovering of topological spaces, i.e. an hypercovering relatively to the proper topology on the category $Top$, for which covering families are cojointly surjective families of proper continuous maps. In order to prove proposition \ref{quipro} it suffices to prove that for every proper hypercovering $B_\bul \lmo A$ between sufficiently nice topological spaces, the induced map $hocolim_{\del^{op}} B_\bul\lmo A$ is a weak equivalence of spaces. By sufficiently nice, we mean locally compact Haussdorf spaces with the homotopy type of CW-complexes and such that all the components of the coskeletons $cosk_k^A B_\bul$ have the homotopy type of CW complexes. These assumptions are satisfied by topological spaces which are complex points of separated schemes of finite type (see \cite{hironakatri}). The proof is then completed by proposition \ref{hypdescpro}. 
\end{proof}

\begin{prop}\label{hypdescpro}
 Let $B_\bul\lmo A$ be a proper hypercovering of topological spaces such that 
 \begin{itemize}
 \item The spaces $A$ and $(B_n)_{\geq 0}$ are Hausdorff, locally compact and have the homotopy type of CW-complexes. 
 \item All the components of the relative coskeletons $cosk^A_k B_\bul$ have the homotopy type of CW-complexes for all $k\geq 0$. 
 \end{itemize}
 Then the map $hocolim_{\del^{op}} B_\bul\lmo A$ is an equivalence in $Top$. 
\end{prop}

\begin{proof} 
The proof will consist of several steps which mimic the proof by Dugger-Isaksen of the open version of our statement in \cite[Thm 4.3]{hyptop} (i.e. if we replace the proper topology by the usual open covering topology). We resume these steps now. First we reduce the statement to the case of bounded hypercoverings (in the sense of \cite[Def 4.10]{dhi}) using the same argument as in the proof of \cite[Thm 4.3]{hyptop}. Second we reduce to the case of a simpler class of bounded proper hypercoverings which are nerves of proper surjective morphisms using the same argument as in the proof \cite[Lem 4.2]{hyptop}. Third, we prove it for this class of hypercoverings using Lurie's proper base change theorem \cite[Cor 7.3.1.18]{htt} to reduce it to the point. Then a lemma from simplicial homotopy theory provides the result for the point. The assumptions made on topological spaces is used in Lurie's theorem and also in Toën's theorem in order to calculate the derived global section of a constant simplicial presheaf.  

We introduce some terminology. Let $\mcal{C}$ be any complete and cocomplete category. For any $[n]\in \del$ and any simplicial object $C_\bul$ in $\mcal{C}$ we denote by $sk_n C$ its $n$-skeleton and $cosk_n C_\bul$ its $n$-coskeleton. If $\mcal{C}=Top$, one has $(cosk_n C_\bul)_i\simeq \map(sk_n \del^i, C_\bul)$, where $\map$ is a mapping space for $Top^{\del^{op}}$. There is an augmented version of these, if $C_\bul\lmo D$ is an augmented simplicial object to a constant simplicial object $D$, then we denote by $sk_n^D C_\bul$ and $cosk_n^D C_\bul$ the the skeleton and coskeleton functor for the category $(\mcal{C}\downarrow D)^{\del^{op}}$. We denote by 
$$M_n C=lim_{(\del^{op} \downarrow n)\setminus id} C_\bul$$
 the $n$th matching object of $C_\bul$, where $(\del^{op} \downarrow n)\setminus id$ is the category of maps to $[n]$ in $\del^{op}$ minus the identity map of $[n]$. There is an augmented version of the matching object. If $C_\bul\lmo D$ is an augmented simplicial object of $\mcal{C}$ into a constant simplicial object $D$, then one can compute the limit seeing $C$ as a functor from $(\del^{op} \downarrow n)\setminus id$ to the category $\mcal{C}\downarrow D$ of maps to $D$ in $\mcal{C}$. We denote it by $M^D_n C_\bul$. There are natural maps $C_n\lmo M_n C_\bul$ and $C_n\lmo M^D_n C_\bul$. 
Suppose $\mcal{C}$ is endowed with a Grothendieck topology so that we can talk about hypercoverings, for us it will be $Top$ with the proper topology. A hypercovering $C_\bul\lmo D$ is called \emph{bounded} if there exists an integer $N\geq 0$ such that the maps $C_n\lmo M_n^D C_\bul$ are isomorphisms for all $n> N$. The minimum $N$ with this property is called the \emph{dimension} of the hypercovering. A hypercovering is bounded of dimension $\leq N$ if and only if the unit map $C_\bul \lmo cosk^D_N C_\bul$ is an isomorphism. 

If $f:C\lmo D$ is a map in $\mcal{C}$, one can see it as a map of constant simplicial objects of $\mcal{C}$. Then we can take the $0$-coskeleton $cosk^D_0 C\lmo D$. This augmented simplicial object is called the nerve of $f$. We have $(cosk^D_0 C)_i=C\times_D\cdots \times_D C$, $n+1$ times. The faces and degeneracies are the projections and diagonals respectively. If $f:C\lmo D$ is a covering in $\mcal{C}$ then the nerve of $f$ is an hypercovering of $D$ of dimension $0$, and these are the only hypercoverings of dimension $0$. 

To reduce to the case of bounded hypercoverings, we observe that for any $k\geq 0$, the hypercovering $cosk^A_{k+1}B_\bul$ is bounded and that the unit map $B_\bul \lmo cosk^A_{k+1} B_\bul$ is an isomorphism on $(k+1)$-skeleton. By Lemma \ref{sque} below, this implies that the top map in the diagram
$$\xymatrix{ hocolim_{\del^{op}} B_\bul \ar[r] \ar[rd] & hocolim_{\del^{op}} cosk^A_{k+1} B_\bul \ar[d] \\ & A }$$ 
induces an isomorphism on the $\pi_k$ at any basepoint. Suppose the statement is proven for bounded hypercoverings, then the right vertical induces an isomorphism on $\pi_k$ because $cosk^A_{k+1} B_\bul$ is bounded. Hence the last map $hocolim_{\del^{op}} B_\bul \lmo A$ induces an isomorphism on $\pi_k$ at any basepoint, hence is a weak equivalence because $k$ is arbitrary.

\begin{lem}\label{sque}
Let $C_\bul\lmo D_\bul$ be a map of simplicial spaces which induces an isomorphisms on $(k+1)$-skeleton. Then the map 
$$\pi_i hocolim_{\del^{op}} C_\bul \lmo \pi_i hocolim_{\del^{op}} D_\bul$$
 is an isomorphism for every $0\leq i\leq k$ and any basepoint. 
\end{lem}

To prove this, we reduce to the case of simplicial simplicial sets using the singular functor, because for simplicial simplicial sets the homotopy colimit is weakly equivalent to the diagonal. We denote by 
$$\xymatrix{ SSet \ar@<2pt>[r]^-{Re} & Top \ar@<2pt>[l]^-{S}  }$$
the standard adjunction with right adjoint the singular functor $S$. It induces an adjunction on the level of simplicial objects just by taking these functors levelwise. The counit map $Re\circ S C_\bul \lmo C_\bul$ is a levelwise weak equivalence in $Top^{\del^{op}}$, hence the induced map 
$$hocolim_{\del^{op}} Re\circ S C_\bul \lmo hocolim_{\del^{op}} C_\bul$$ 
is a weak equivalence. But composing with the canonical weak equivalence $hocolim_{\del^{op}} Re\circ S C_\bul \simeq Re( hocolim_{\del^{op}}Sing C_\bul)$, we get a weak equivalence $Re( hocolim_{\del^{op}}S C_\bul)\simeq hocolim_{\del^{op}} C_\bul$. Then for every $i\geq 0$ we get a canonical isomorphisms 
$$\pi_i(hocolim_{\del^{op}}S C_\bul) \simeq \pi_i Re( hocolim_{\del^{op}}S C_\bul)\simeq \pi_i hocolim_{\del^{op}} C_\bul$$
at every basepoint. Therefore it suffices to prove the claim for $C_\bul \lmo D_\bul$ a map in $SSet^{\del^{op}}$. But in that case there is canonical weak equivalence $hocolim_{\del^{op}} C_\bul \simeq dC_\bul$ in $SSet$ where $d:SSet^{\del^{op}} \lmo SSet$ is the diagonal functor. Then if $C_\bul \lmo D_\bul$ is an isomorphism on $(k+1)$-skeleton, it is straighfoward that $\pi_i dC_\bul \lmo \pi_i dD_\bul$ is an isomorphism for every $0\leq i\leq k$. This finishes the proof of \ref{sque}.

Back to the proof of \ref{hypdescpro}, we will then proceed by induction on the dimension of the hypercovering, reducing the proof to the dimension $0$ case. 

Let $n\geq 0$ be an integer. Suppose we have proven the statement for hypercoverings of dimension $\leq n$ and let $B_\bul\lmo A$ be a bounded proper hypercovering of dimension $n+1$. Consider the unit map $B_\bul \lmo cosk_n^A B_\bul=:C_\bul$. Then $C_\bul$ is bounded of dimension $\leq n$. Consider the bisimplicial space which is the nerve of the map $B_\bul \lmo C_\bul$
$$E_{\bul\bul} :=( \xymatrix{B_\bul & B_\bul \times_{C_\bul} B_\bul \ar@<2pt>[l] \ar@<-2pt>[l]  & B_\bul \times_{C_\bul} B_\bul \times_{C_\bul} B_\bul \ar@<3pt>[l] \ar[l] \ar@<-3pt>[l] \cdots } ).$$
Considering $C_\bul$ as constant in one simplicial direction, we have a map $E_{\bul\bul}\lmo C_\bul$. The $k$th row of $E_{\bul\bul}\lmo C_\bul$ is the nerve of the map $B_k\lmo C_k$. Consider the diagonal $D_\bul :=dE_{\bul\bul}$. Then standard homotopy theory (see e.g. \cite{hirs}) proves that $hocolim_{\del^{op}} D_\bul$ is weakly equivalent to the space obtained by taking the homotopy colimit of each rows of $E_{\bul\bul}$, and then taking the homotopy colimit of the resulting simplicial space. But by induction hypothesis, the $k$th row being a dimension $0$ hypercovering of $C_k$, its homotopy colimit is weakly equivalent to $C_k$. The resulting simplicial object is $C_\bul$, which is of dimension $\leq n$, so by induction hypothesis $ hocolim_{\del^{op}} C_\bul \simeq A$. Hence we prove that $hocolim_{\del^{op}} D_\bul \simeq A$.

Now we prove that $B_\bul$ is a retract of $D_\bul$ over $A$, hence that $hocolim_{\del^{op}} D_\bul \simeq hocolim_{\del^{op}} B_\bul \simeq A$. There is a natural map $B_\bul \lmo D_\bul$ given by the horizontal degeneracy $E_{0,k} \lmo E_{k,k}$. Then we need a map $D_\bul \lmo B_\bul$. It is sufficient to find a map $sk^A_{n+1}D_\bul \lmo sk^A_{n+1}B_\bul$ because then the adjoint map $D_\bul \lmo cosk^A_{n+1} sk^A_{n+1} B_\bul \simeq B_\bul$ is the wanted map. Notice that because of the definition of $C_\bul$ the map $B_k\lmo C_k$ is an isomorphism for $k=0,\hdots, n$ and the map $sk^A_n B_\bul \lmo sk^A_n D_\bul$ is an isomorphism. Let $[0]\lmo [n+1]$ be any coface map, giving a face map $E_{n+1, n+1} \lmo E_{0,n+1}$ which gives the wanted map $sk^A_{n+1}D_\bul \lmo sk^A_{n+1}B_\bul$. One can check that $B_\bul \lmo D_\bul \lmo B_\bul$ is the identity which proves our claim. 

It remains to prove the statement for a dimension $0$ proper hypercovering $\pi : B_\bul\lmo A$ with spaces satisfying the assumptions of \ref{hypdescpro}. Such a hypercovering is the nerve of a proper surjective map $B_0\lmo A$. Therefore $B_n\simeq B_0\times_A\cdots\times_A B_0$, $n+1$ times. We will use a proper base change argument. For this we will study simplicial presheaves on the simplicial space $B_\bul$ and their behavior with respect to $\pi$. Deligne defined in \cite{deligneIII} a notion of sheaves on a simplicial space, constructing a site out of a simplicial space and taking sheaves on it. His construction can be directly use for simplicial presheaves. Indeed let $\tilde{B}_\bul$ be the category with objects the pairs $([n], V)$ with $[n]\in \del$ and $V\subseteq B_n$ an open subset. A morphism between $([n], V)$ and $([m], V')$ is the data of a morphism $a:[n]\lmo [m]$ in $\del$ and a continuous map $V'\lmo V$ such that the square
$$\xymatrix{V' \ar[r] \ar[d] & V\ar[d] \\ B_m \ar[r]^-{B(a)} & B_n   }$$
commute. Composition and identities are defined in the obvious way, and satisfy all the required conditions. The category $\tilde{B}_\bul$ is naturally endowed with the open covering topology induced by the topology of each $B_n$ and $\del$ is considered as discrete. Then we can consider the category $SPr(B_\bul)$ of simplicial presheaves on the site $\tilde{B}_\bul$. An object $F$ in this category is equivalent to the data of simplicial presheaves $F_n$ on $B_n$ for every $n\geq 0$, and for every map $a:[n]\lmo [m]$ in $\del$, a map of presheaves $u_a : F_n\lmo B(a)_\ast  F_m$, such that $u_{id_{[n]}} =id_{F_n} $ and for every $a:[n]\lmo [m]$ and $b:[m]\lmo [k]$ in $\del$, we have $u_{ba} =B(a)_\ast u_b u_a$. 

If $X$ is any space considered as a constant simplicial space, then $\tilde{X}$ is nothing but the site of opens of $X$. The map of simplicial spaces $\pi : B_\bul\lmo A$ gives a map of sites still denoted by $\pi : \tilde{B}_\bul\lmo \tilde{A}$. Consider the diagram of categories
$$\xymatrix{\tilde{B}_\bul \ar[rr]^-{\pi} \ar[rd]_-q && \tilde{A} \ar[ld]^-p \\ & \ast  &  }$$
where $\ast$ is the punctual category. Then taking simplicial presheaves we get a set of adjoint functors
$$\xymatrix{SPr(B_\bul) \ar@<2pt>[rr]^-{\pi_\ast}  \ar@<2pt>[rd]^-{q_\ast} && SPr(A) \ar@<2pt>[ll]^-{\pi^{-1}} \ar@<2pt>[ld]^-{p_\ast} \\ & SSet \ar@<2pt>[lu]^-{cst} \ar@<2pt>[ru]^-{cst} & }$$
where $cst(K)$ is the constant simplicial presheaf with value $K$ for any $K\in SSet$. The functors $p_\ast$ and $q_\ast$ are also famous under the name of global sections and are the right adjoints to $cst$. We endow $SPr(B_\bul)$ and $SPr(A)$ with the local model structure (with respect to the open covering topology) obtained as a Bousfield localization of the injective model structure (what is really important is the weak equivalences which are local equivalences, but we will need below to consider a homotopy limit, which explains why we need the injective one). Then the functors $\pi_\ast$, $p_\ast$ and $q_\ast$ are right Quillen. For any $K\in SSet$ we have $\pi^{-1}\circ cst(K)\simeq cst(K)$ and there are isomorphisms $\lef \pi^{-1}\simeq  \pi^{-1}$ and $\lef cst\simeq cst$. Therefore we have a canonical isomorphism $\R p_\ast \R\pi_\ast \simeq \R q_\ast$. 

A direct consequence of Toën's result \cite[Thm 2.13]{galhom} is that for any constant simplicial presheaf $K\in SPr(A)$ we have a canonical isomorphism $\R p_\ast (K)\simeq \R\map (S A, K)$ in $Ho(SSet)$. This uses the fact that our space $A$ has the homotopy type of a CW complex. Next we want to calculate the derived global sections $\R q_\ast (K)$ of a constant simplicial presheaf $K$ on $B_\bul$. The site $\tilde{B}_\bul$ is naturally endowed with a functor $\alpha : \tilde{B}_\bul\lmo \del$ where $\del$ is considered as a discrete site. The functor $\alpha$ is just the projection $\alpha ([n], V)=[n]$. We denote by $\beta :\del\lmo \ast$ the unique functor. We have induced functors on simplicial presheaves 
$$\xymatrix{SPr(B_\bul) \ar[r]^-{\alpha_\ast} \ar[rd]_-{q_\ast} & SPr(\del) \ar[d]^-{\beta_\ast} \\ & SSet }$$
These functors are right Quillen ($SPr(\del)$ is also endow with the injective model structure). The derived functor $\R \beta_\ast$ is then isomorphic to $holim_\del$. For any constant simplicial presheaf $K\in SPr(B_\bul)$ , we have $\R\alpha_\ast (K)=( \R q_{n\ast} K)_{n\geq 0}$ where $q_n:\tilde{B}_n\lmo \ast$. Using Thm 2.13 of loc.cit. (all spaces $B_n$ having the homotopy type of a CW), we obtain an isomorphism $\R\alpha_\ast (K)\simeq (\R\map(SB_n, K))_{n\geq 0}$ in $Ho(SPr(\del))$. Therefore we have a canonical isomorphism in $Ho(SSet)$
$$\R q_\ast (K)\simeq holim_{\del^{op}} \R\map(SB_\bul, K)\simeq \R\map(hocolim_{\del^{op}} SB_\bul, K).$$
Suppose for the time being that the following lemma is at our disposal.

\begin{lem}\label{pbc}
Let $K\in SPr(B_\bul)$ be a constant simplicial presheaf, with $K$ being a truncated simplicial set. Then the unit map 
$$K\lmo \R\pi_\ast \pi^{-1} (K)\simeq \R\pi_\ast(K)$$
 is an isomorphism in $Ho(SPr(A))$ (where $K$ denotes the same constant presheaf on $A$). 
\end{lem}

We will give a proof below. Applying the isomorphism $\R p_\ast \R\pi_\ast \simeq \R q_\ast$ to a truncated constant simplicial presheaf $K$ on $B_\bul$, we obtain a canonical isomorphism for every truncated simplicial set $K$
$$\R\map(SA, K)\simeq \R\map(hocolim_{\del^{op}} SB_\bul, K).$$
This implies that the map $hocolim_{\del^{op}} SB_\bul\lmo SA$ is an isomorphism in $Ho(SSet)$. By taking the Quillen equivalence $Re$ and the fact that $Re$ commutes with homotopy colimits, this implies that the map $hocolim_{\del^{op}} B_\bul\lmo A$ is an isomorphism in $Ho(Top)$, proving our claim. 

To sum up, it only remains to prove \ref{pbc}. 

To prove \ref{pbc} it suffices to prove that the unit map is an isomorphism on the stalk at any point $a\in A$. We have a cartesian square of simplicial spaces 
$$\xymatrix{B_\bul^a\ar[r]^-{\pi^a} \ar[d]_-\phi & \ast \ar[d]^-a \\ B_\bul \ar[r]^-{\pi}  & A }$$
We claim that the base change theorem holds for this square and for a truncated constant simplicial presheaf. 

\begin{lem}\label{pbcreal}
For any truncated constant simplicial presheaf $K\in SPr(B_\bul)$, the canonical map 
$$a^{-1} \R \pi_\ast (K)\lmo \R\pi^a_\ast \phi^{-1} (K)\simeq\R\pi^a_\ast (K) $$
is an isomorphism in $Ho(SSet)$. 
\end{lem}

Once the lemma is proven, it will just remain to prove that the map $K\lmo \R\pi^a_\ast (K)$ is an isomorphism in $Ho(SSet)$, which is exactly the statement of \ref{hypdescpro} for $A=\ast$ and $B_\bul$ a dimension $0$ hypercovering. Indeed we have an isomorphism $ \R\pi^a_\ast (K)\simeq holim_{\del^{op}} \R\map(SB_\bul^a, K)$. 

To prove \ref{pbcreal}, we will calculate the stalk $a^{-1} \R \pi_\ast (K)$ and relate it to $\R\pi^a_\ast (K)$. For any open subset $U\subseteq A$ we have a cartesian square of simplicial spaces 
$$\xymatrix{B_\bul^U\ar[r]^-{\pi^U} \ar[d]_-{\phi^U} & U \ar[d]^-i \\ B_\bul \ar[r]^-{\pi}  & A }$$
Then we have isomorphisms $\R\Gamma (U, \R\pi_\ast^U (K))\simeq \R\Gamma(B_\bul^U, K)\simeq holim_{\del^{op}} \R\map(SB_\bul^U, K)$ in $Ho(SSet)$. The stalk $a^{-1} \R \pi_\ast (K)$ is isomorphic to usual filtered colimit 
$$colim_{a\in U\subseteq A} \R\Gamma (U, \R\pi_\ast^U (K)) \simeq colim_{a\in U\subseteq A} holim_{\del^{op}} \R\map(SB_\bul^U, K).$$
Now we use the assumption that $K$ is truncated to deduce the fact this homotopy limit is isomorphic to a finite homotopy limit. Indeed if $K$ is $n$-truncated, then the simplicial set $\R\map(SB_\bul^U, K)$ is also $n$-truncated and one can calculate this homotopy limit by restricting to the subcategory of $\del^{op}$ given by simplexes of dimension $\leq n+1$. Then we can make this filtered colimit and this finite homotopy limit commute to get
$$a^{-1} \R \pi_\ast (K) \simeq holim_{\del^{op}} colim_{a\in U\subseteq A}  \R\map(SB_\bul^U, K).$$
Now we wish to have for all $n\geq 0$ an isomorphism $colim_{a\in U\subseteq A}  \R\map(SB_n^U, K)\simeq  \R\map(SB_n^a, K)$. To do so we apply Lurie's proper base change Theorem \cite[Cor 7.3.1.18]{htt} to the cartesian square of locally compact Hausdorff spaces
$$\xymatrix{ B_n^a\ar[r]^-{\pi_n^a} \ar[d]_-{\phi_n} & \ast \ar[d]^-a \\ B_n \ar[r]^-{\pi_n}  & A  }$$
We can apply this result here because our simplicial presheaf $K$ is truncated so that according to Cor 7.2.1.12 of loc.cit., $K$ satisfies hyperdescent if and only if $K$ satisfies ordinary descent. We obtain an isomorphism $a^{-1} \R\pi_{n\ast } (K)\simeq \R\pi_{n\ast}^a(K)$. With the same argument as before, we have $a^{-1} \R\pi_{n\ast } (K)\simeq colim_{a\in U\subseteq A}  \R\map(SB_n^U, K)$, which proves the expected isomorphism. This finishes the proof of Lemma \ref{pbcreal}.

Back to the proof of \ref{pbc}, it remains to prove that the map $K\lmo \R\pi^a_\ast (K)$ is an isomorphism in $Ho(SSet)$ for all truncated $K$. In view of what has been said, it is equivalent to the statement that $hocolim_{\del^{op}} B_\bul^a \lmo \ast$ is a weak equivalence of spaces. It is treated by the following lemma. 

\begin{lem}
Let $X$ be any non empty topological space (resp. a non empty simplicial set). Then the nerve $X_\bul \lmo \ast$ of the map $p:X\lmo \ast$ induces a weak equivalence $hocolim_{\del^{op}} X_\bul \lmo \ast$ in $Top$ (resp. in $SSet$). 
\end{lem}

The statement in $SSet$ implies the statement in $Top$. Indeed if $X\in Top$ we saw in the proof of \ref{sque} that $hocolim_{\del^{op}} X_\bul$ and $hocolim_{\del^{op}} SX_\bul$ have the same homotopy groups. 

Let $X\in SSet$. We prove that the map $X_\bul \lmo \ast$ is a simplicial homotopy equivalence in $SSet^{\del^{op}}$. Let $x:\ast \lmo X$ be a point. Then it suffices to find a homotopy $h:\del^1\times X_\bul\lmo X_\bul$ between $id_{X_\bul}$ and $xp$. We define $h_n:\del([n], [1]) \times X_n\lmo X_n$ by the following formula. Let $a:[n]\lmo [1]$ be a map in $\del$, it is essentially given by an integer $0\leq m\leq n$. We set $h_n(a, (x_0, \hdots, x_n))=(x_0, \hdots, x_m, x,\hdots, x)$. We then have a homotopy which verifies $h(0,-)=id_{X_\bul}$ and $h(1,-)=xp$. 

Recall the realization functor 
$$\re{-} : SSet^{\del^{op}} \lmo SSet$$
defined by the standard formula
$$\re{Y_\bul}:=coeq(\xymatrix{\bigsqcup_{n\in \del} \del^n\times Y_n& \bigsqcup_{p\mo q\in \del} \del^p\times Y_q \dar[l] } ).$$
This functor sends simplicial homotopy equivalences to simplicial homotopy equivalences. This implies that $\re{X_\bul}$ is contractible. Now we use the isomorphism $hocolim_{\del^{op}} X_\bul \simeq \re{X_\bul}$ in $Ho(SSet)$ (see \cite{hirs}) to conclude that $hocolim_{\del^{op}} X_\bul$ is contractible. 

Now the proof of \ref{hypdescpro} is complete. 
\end{proof}

\section{Topological K-theory of noncommutative spaces} 

We have now almost all we need to define the semi-topological and topological K-theory of noncommutatives spaces. The first part is dedicated to the definitions of semi-topological and topological K-theory. The semi-topological K-theory is roughly speaking the spectral topological realization of algebraic K-theory. These definition are possible modulo the calculation of the semi-topological K-theory of the point, which is roughly speaking the spectral topological realization of the stack of $E_\infty$-spaces of vector bundles, which is proved to be the usual connective spectrum $\bu$ in the second part. In the third part we give a convenient description of semi-topological K-theory in terms of the stack of perfect dg-modules. In the fourth part we prove that the Chern map descends to topological K-theory. Finally the last parts treat the examples of smooth schemes and of finite dimensional algebras, with some suprising consequences about the relation between the periodic homology groups of an algebra and the homotopy groups of the stabilized topological realization of the stack of noncommutative vector bundles.

\subsection{Definition of the (semi-) topological K-theory}\label{kstktop}

In §\ref{caracalg}, we defined a presheaf of symmetric ring spectra 
$$\ka : \affc^{op}\lmo Sp$$
such that for all $\spec(A)\in \affc$, we have a canonical isomorphism $\ka(\spec(A))\simeq\kn(A)$ in $Ho(Sp)$. For all $\C$-dg-category $T\in\dgcatc$, we defined a presheaf of symmetric $\ka$-module spectra
$$\ukn(T) : \affc^{op} \lmo Sp$$
such that for all $\spec(A)\in \affc$, we have a canonical isomorphism $\ukn(T)(\spec(A))\simeq \kn(T\tel_\C A)$ in $Ho(Sp)$. Thus we have an isomorphism $\ukn(\unit)\simeq \ka$ where $\unit$ is the $\C$-dg-category with one object and the ring $\C$ as endomorphisms. We recall that we denote by $\ka-Mod_\s$ the category of $\ka$-modules in the monoidal category $\spaf$. We therefore have $\ukn(T)\in \ka-Mod_\s$. 

In §\ref{realsp} we mentioned that the spectral realization functor, 
$$\resp{-} : Ho(\spaf)\lmo Ho(Sp)$$ 
is a monoidal functor. Because of the existence of an homotopy category of monoids and modules (see \cite[Thm 3.3 et Thm 2.1]{hoveymon}), $\resp{\ka}$ is a ring spectrum and the topological realization extends to $\ka$-modules with values in $\resp{\ka}$-modules, 
$$\resp{-} : Ho(\ka-Mod_\s)\lmo Ho(\resp{\ka}-Mod_\s).$$
We also have a connective version of K-theory $\tilde{\ka}$ and $\ukc(T)$. In the same way we obtain a topological realization for $\tilde{\ka}$-modules, 
$$\resp{-} : Ho(\tilde{\ka}-Mod_\s)\lmo Ho(\resp{\tilde{\ka}}-Mod_\s).$$

\begin{df}\label{defkst}
The \emph{semi-topological K-theory} (resp. the \emph{connective semi-topological K-theory}) of a $\C$-dg-category $T\in \dgcatc$ is the symmetric $\resp{\ka}$-module spectrum (resp. the symmetric $\resp{\tilde{\ka}}$-module spectrum),
$$\kst(T):=\resp{\ukn (T)} \qquad (\textrm{resp.  } \kcst(T):=\resp{\ukc (T)}).$$
Because of the existence of a cofibrant replacement functor in model categories of modules, we have two functors,
$$\kst : \dgcatc\lmo \resp{\ka}-Mod_\s,$$
$$\kcst : \dgcatc\lmo \resp{\tilde{\ka}}-Mod_\s.$$
We denote by $\kst_i(T):=\pi_i\kst(T)$ for all $i\in \Z$ the semi-topological K-groups. 
\end{df}

\begin{rema}
It is a priori necessary to also consider the connective version $\kcst$ in our study, because we do not know if  $\kcst(T)$ is the connective covering of $\kst(T)$. Indeed the first thought is to remark that the topological realization is a left adjoint while the connective cover is a right adjoint. 
\end{rema}

\begin{rema}\label{fibtop}
We remark the existence of a map from algebraic K-theory to semi-topological K-theory. If $T\in \dgcatc$, we have the unit map of adjoint pair $(\resp{-}, \hbs)$, 
$$\ukn(T)\lmo \hbs(\resp{\ukn(T)}).$$
Taking global sections, i.e. the value on $\spec(\C)$, we obtain a map in $Ho(Sp)$, 
$$\eta_T:\kn(T)\lmo \rhomi_{Ho(Sp)}(\sinf(\spec(\C))_+,\kst(T))\simeq \kst(T).$$
This defines a map $\eta:\kn\lmo \kst$ in $Ho(Sp^{\dgcatc})$. There is also a connective version of it denoted by $\tilde{\eta}:\kc\lmo \kcst$. For all scheme $X\in \schc$, taking $\pi_0$ we obtain a map 
$$\kc_0(X)\lmo \kcst_0(X).$$
Because of the formula \ref{pi0cor}, this map is the quotient map  of the equivalence relation on algebraic vector bundles on $X$ which identifies two vector bundles when they can be related by a connected algebraic curve. We obtain then a map $\kcst_0(X)\lmo \ktopu^0(sp(X))$ from our semi-topological K-group to the Grothendieck group of topological vector bundles on $sp(X)$. We'll prove below that this map is an isomorphism in the case of a smooth and proper $\C$-scheme of finite type. 
\end{rema}

\begin{rema}\label{restlisskst}
By applying the spectral version of theorem \ref{restliss}, we see that we can calculate the topological realization of definition \ref{defkst} by first taking the restriction to smooth schemes.
\end{rema}

The following two results are central in the definition of topological K-theory. We denote by $bu$ the usual topological K-theory spectrum. This means that for a topological space $X\in Top$, if $\ktopu^0(X)$ is the Grothendieck group of complex topological vector bundles on $X$, we have an isomorphism 
$\ktopu^0(X)\simeq\pi_0 \map_{Ho(Sp)} (\sinf SX_+, bu)$. A model for $bu$ as a symmetric spectrum will be given below.

\begin{theo}\label{bu}
There exists a canonical isomorphism $\kcst(\unit)\simeq bu$ in $Ho(Sp)$. 
\end{theo}

A proof will be given below in \ref{ktoppoint}. For the time being we use it in order to define topological K-theory.

\begin{theo}\label{annupoint}
For all smooth commutative algebra $B\in \calgc$, the canonical map
$$\kcst(B)\lmo \kst(B)$$ 
is an isomorphism in $Ho(Sp)$. In particular by theorem \ref{bu} we have an isomorphism $\kst(\unit)\simeq \bu$ in $Ho(Sp)$. 
\end{theo}

\begin{proof} It's a well known fact that the negative algebraic K-theory of a smooth commutative algebra vanishes (see \cite[Rem 7]{schl}). Therefore the map of presheaves of spectra $\ukc(B)\lmo \ukn(B)$ is an equivalence on smooth affine schemes. By theorem \ref{restliss}, we conclude that it induces an equivalence on spectral topological realization $\kcst(B)\simeq \kst(B)$. 
\end{proof}

\begin{nota} 
The last two theorems reformulate by saying we have isomorphisms
$$\resp{\tilde{\ka}}\simeq \resp{\ka}\simeq bu$$
in $Ho(Sp)$. We denote by $\bu$ the symmetric ring spectrum $\resp{\ka}$. We have functors
$$\kcst, \kst : \dgcatc\lmo \bu-Mod_\s.$$
\end{nota}

\begin{rema}\label{uniqbu}
It's a classical fact that $bu$ admits a model as a strict commutative ring in symmetric spectra, with addition corresponding to the sum of vector bundles and multiplication to external tensor product. A way to do it is to express $bu$ as the spectrum associated to the special $\gam$-space which is the topological realization of the stack of algebraic vector bundles (see \ref{ktoppoint}). This $\gam$-space has the structure of $\gam$-ring with multiplication given by external tensor product. This commutative ring structure on $bu$ is moreover \emph{unique} by \cite[Cor 1.4]{bakerrichter}, in the sense that for all commutative symmetric ring spectrum $A$ and for all ring map $f:A\lmo bu$ which induces an isomorphism on all homotopy groups, then there exists a map $g : A\lmo bu$ in the homotopy category of \emph{commutative} ring spectra such that $g$ is isomorphic to $f$.  

Here we can't consider $\bu$ as a commutative ring spectrum, neither $\ka$ as a \emph{commutative} ring spectrum because this ring structure is given by the ring structure of endomorphisms of unity in $\mbb{M}_{loc} (\C)$ (see \ref{caracalg}). We'll then just talk about associative unital ring spectra, still knowing at the same time that $\bu$ is equivalent in $Sp$ to a commutative ring spectrum and that this structure is unique. We feel here the limit imposed by the language of model categories and strict algebraic structures compared to the highly more flexible language of Lurie's monoidal $\infty$-categories. 
\end{rema}

\begin{nota}\label{choixbeta}
 By the Bott periodicity theorem, we know that the abelian group $\pi_2\bu$ is rank $1$ free. To define topological K-theory, we choose a Bott generator $\beta\in \kst_2(\unit)=\pi_2 \bu$ from the two existing. It's more convenient to first choose a generator $\alpha$ of the non-trivial part of $\kn_0(\po)$ and to take $\beta=\eta(\alpha)$ where $\eta:\kn_0(\po)\lmo \kst_0(\po)$ is the canonical map. Then $\beta$ gives a generator of $\kst_2(\unit)$ by the canonical map $\kst_0(\po)\lmo \ktopu(S^2) \simeq \ktopu^0(\ast)\oplus \beta \ktopu^{-2}(\ast)$. 
\end{nota}

\begin{nota}
We recall that given a ring spectrum $A$, an integer $k$, and an element $a\in\pi_k A$, one can define the ring spectrum $A[a^{-1}]$ localized with respect to $a$. It is endowed with a map $i_a:A\lmo A[a^{-1}]$ and verifies the following universal property. For all ring spectrum $B$, and all ring map $A\lmo B$, the simplicial set $\map_{A-Alg} (A[a^{-1}], B)$ is non-empty if and only if $a$ is invertible in the $\pi_*(A)$-module $\pi_*(B)$. This property characterizes the object $A[a^{-1}]$ up to equivalence. This is the noncommutative version of \cite[Prop 1.2.9.1]{hag2} applied to the monoidal model category $Sp$ of symmetric spectra. We conclude by Cor 1.2.9.3 of loc.cit. that the functor induced by composition with $i_a$, 
$$i_a^* : Ho( A[a^{-1}]-Mod_\s) \lmo Ho(A-Mod_\s)$$
is fully faithfull and its image consists of the $A$-modules $M$ such that multiplication by $a$ is invertible in the $\pi_*A$-module $\pi_*M$. If $M$ is a $A$-module and $a\in\pi_kA$ an element, the localized module $M[a^{-1}]$ is defined as the $A[a^{-1}]$-module $M\sm_A A[a^{-1}]$. 
\end{nota}

\begin{rema}\label{remloccomass}
Let $A$ be a ring spectrum and $a\in \pi_k A$. Denote by $A_{ass}[a^{-1}]$ the localization of $A$ with respect to $a$ in the sense of associative ring spectra and by $A_{com}[a^{-1}]$ the localization of $A$ with respect to $a$ but in the sense of commutative ring spectra. Then by the universal property, there exists a map $ c: A_{ass}[a^{-1}]\lmo A_{com}[a^{-1}]$ which is an equivalence. Indeed there exists a commutative diagram 
$$\xymatrix{Ho(A_{ass}[a^{-1}]-Mod_r) \ar@{^{(}->}[r]  \ar[d]^-{c^*} & Ho(A-Mod_r) \ar[d]^-{\wr} \\ Ho( A_{com}[a^{-1}]-Mod) \ar@{^{(}->}[r] & Ho(A-Mod)  }$$
where the indice $r$ stands for the category of right modules. Since $A$ is commutative, the map from right $A$-modules to two-sided $A$-modules is an equivalence. Then the two homotopy categories of modules $Ho(A_{ass}[a^{-1}]-Mod_r)$ and $Ho( A_{com}[a^{-1}]-Mod)$ are equivalent to the subcategory of $Ho(A-Mod)$ formed by the $A$-modules for which multiplication by $a$ is an equivalence. We conclude that $c^*$ is an equivalence of categories and that $c$ is an equivalence. 
\end{rema}

\begin{rema}
We consider the ring spectrum $\bu$, the Bott generator we choosed $\beta\in \pi_2\bu$ and its localization $\bu[\beta^{-1}]$. A priori the latter localization is calculated in the sense of associative rings. But after remark \ref{uniqbu}, we know that there exists a ring map $\bu\lmos{\sim} bu$ with $bu$ a commutative ring model, which is an equivalence. From \ref{remloccomass} we deduce that there is no ambiguity on the ring $\bu[\beta^{-1}]$ : it is equivalent to the commutative localization of $bu$ and is therefore equivalence to the usual colimit
$$\bu[\beta^{-1}]\simeq \xymatrix{  colim (bu \ar[r]^-{\cup \beta}  & bu \ar[r]^-{\cup \beta} & \cdots )}$$
where $\cup \beta$ is the map multiplication by $\beta$. The ring structure of $\bu[\beta^{-1}]$ is therefore the usual structure, and there is an isomorphism $\bu[\beta^{-1}]\sm_\s H\C \simeq H\cuu$ (with $u$ of degree $2$) in the homotopy category of $H\C$-algebras. We adopt the notation $\BU:=\bu[\beta^{-1}]$. 
\end{rema}

\begin{df}\label{deftop}
The \emph{topological K-theory of a $\C$-dg-category $T\in \dgcatc$} is the symmetric spectrum
$$\ktop(T):=\kst(T)[\beta^{-1}].$$
This defines a functor
$$\ktop : \dgcatc\lmo \BU-Mod_\s.$$
We denote by $\ktop_i(T):=\pi_i\ktop(T)$ for all $i\in\Z$ the topological K-groups. 
\end{df}

\begin{rema}
Following remark \ref{fibtop}, we compose the unit map $\kn\lmo \kst$ with the structural map $\kst\lmo \ktop$ and we obtain a map denoted by $\theta : \kn\lmo \ktop$ from algebraic K-theory to topological K-theory. 
\end{rema}

Topological K-theory inherits the properties of algebraic K-theory of proposition \ref{propalgk}. 

\begin{prop} 
\begin{description}
\item[a.] Topological K-theory commutes with filtrant homotopy colimits of dg-categories. 
\item[b.] Topological K-theory sends Morita equivalences to equivalences of spectra. 
\item[c.] For all exact sequence of dg-categories $T'\mo T\mo T''$, the induced sequence
$$\ktop(T')\lmo \ktop(T)\lmo \ktop(T'')$$
is a distinguished triangle in $Ho(Sp)$. 
\end{description}
\end{prop} 

\begin{proof} 
\begin{description}
\item[a.] Comes from the fact that the topological realization $\resp{-}$ and the operation of localization by $\beta$ are left adjoints and thus commute with homotopy colimits. 
\item[b.] By functoriality. 
\item[c.] Comes from the fact that $\resp{-}$ and the operation of localization by $\beta$ are exact functors. 
\end{description}
\end{proof}


\subsection{The case of the point}\label{ktoppoint}

We give a particular model for the spectrum $bu$, as a symmetric spectrum. For all commutative algebra $A\in \calgc$, we denote by $\proj(A)$ the Waldhausen category of projective $A$-modules of finite type, i.e. of finite rank vector bundles on $\spec(A)$ (see remark \ref{algebra}). The equivalences in this Waldhausen category $\proj(A)$ are by definition the isomorphisms and the cofibrations are admissible monomorphisms. We define a pseudo-functor $\affc^{op} \lmo WCat$, by setting for all map of algebras $A\lmo B$, the induced exact functor is the tensor product
\begin{align*}
\proj(A) & \lmo \proj(B) \\
E & \longmapsto E\te_A B
\end{align*}
which verifies the usual associativity conditions up to isomorphism. We denote by $\proj$ the canonical strictification of this pseudo-functor. We denote by $\vect=Nw\proj$ the simplicial presheaf obtained by taking levelwise the nerve of equivalences in $\proj(A)$. The direct sum of modules induces a structure of homotopy coherent commutative monoid on $\vect$. More precisely, using the construction $B_W$ mentioned in the end of \ref{monoids}, we have a $\gam$-simplicial presheaf
$$\vect_\bul :=NwB_W \vect \in \gam-\spr$$
such that for all integer $n\geq 0$, there is a levelwise equivalence $\vect_{(n)}\simeq \vect^n$. We define the connective symmetric spectrum $bu$ as
$$bu:=\mcal{B}\regam{\vect_\bul}^+.$$

\begin{rema}
We use notation from \ref{prestru} concerning classifying spaces of groups. We denote by $Gl_n : \affc^{op}\lmo SSet$ the discrete simplicial presheaf of linear groups. We denote by $\coprod_{n\geq 0} \mrm{B}Gl_n$ the $\gam$-simplicial presheaf whose commutative monoid structure is given by block sum of matrices. It is well known that there exists an étale local equivalence of $\gam$-simplicial presheaf,
$$\vect_\bul \simeq \coprod_{n\geq 0} \mrm{B}Gl_n$$
By theorem \ref{di}, proposition \ref{proprelder} and \ref{gp}, we have equivalences
$$\re{\vect_\bul}^+\simeq \re{\coprod_{n\geq 0} \mrm{B}Gl_n}^+ \simeq (\coprod_{n\geq 0} \re{\mrm{B}Gl_n})^+ \simeq (\coprod_{n\geq 0} \mrm{B}\re{Gl_n})^+\simeq  (\coprod_{n\geq 0} \mrm{B} Gl_n(\C))^+$$
where $Gl_n(\C)$ stands the topological space of complex points, and we take its classifying space as a topological group. The topological group is homotopy equivalent to the unitary group $U_n(\C)$, and we have an equivalence
$$\re{\vect_\bul}^+\simeq (\coprod_{n\geq 0} \mrm{B} U_n(\C))^+.$$
But the group completion $\coprod_{n\geq 0} \mrm{B} U_n(\C)$ is known (see \cite[App Q]{fmfilt}) to be equivalent to
$$(\coprod_{n\geq 0} \mrm{B} U_n(\C))^+\simeq \mrm{B}U_\infty \times \Z,$$ 
where $BU_\infty$ is the colimit of the $\mrm{B}U_n(\C)$ with respect to the natural inclusion $\mrm{B}U_n(\C)\hookrightarrow \mrm{B}U_{n+1}(\C)$, with the structure of $\gam$-objects still given by the block sum of matrices and the usual law for $\Z$. In consequence, by theorem \ref{bf}, we have an equivalence of spectra $\mcal{B}\regam{\vect_\bul}^+ \simeq \mcal{B} (BU_\infty \times \Z)$, which is the common definition of $bu$. 
\end{rema}

\begin{proof} du théorème \ref{bu} --- We have a chain of canonical isomorphims in $Ho(Sp)$, 
\begin{align*}
\kcst(\ast)& =\resp{\ukc(\ast)} \\
& \simeq \resp{\kc(\vect)} \qquad &&\textrm{(by remark\ref{algebra})} \\
& = \resp{\mcal{B} K^\gam(\vect)} \qquad & &\textrm{(by definition, see end of \ref{monoids})} \\
&\simeq \mcal{B} \regam{ K^\gam(\vect)} \qquad &&\textrm{(by proposition \ref{proprelb}) } \\
&\simeq \mcal{B} \regam{(\vect_\bul)^+} \qquad &&\textrm{(\cite[Lem 1.10]{teze}, recalled at the end of \ref{monoids}) } \\
&\simeq \mcal{B} \regam{\vect_\bul}^+ \qquad &&\textrm{(by proposition \ref{proprelplus})} \\
& =bu. 
\end{align*}
\end{proof}


\subsection{Topological K-theory via the stack of perfect modules}\label{sectkstmt}

Semi-topological K-theory, as initiated by Toën (see \cite{sat}, \cite{kal}, \cite{kkp}) was firstly defined as the topological realization of the stack of pseudo-perfect dg-modules associated to a dg-category. We consider two stacks associated to a dg-category $T$. The stack $\M_T$ of pseudo-perfect dg-modules and the stack $\M^T$ of perfect dg-modules. The stack $\M^T$ was studied by Toën--Vaquié in \cite{modob}. We show below that semi-topological K-theory of a dg-category can be recovered as the topological realization of the stack $\M^T$ of perfect dg-modules. This result is based on the existence of an $\ao$-homotopy equivalence between $\M^T$ and the $S$-construction of the category of perfect $T$-dg-modules. The stack $\M_T$ of pseudo-perfect dg-modules gives rise to a dual theory with respect to $\kst$ which can therefore be called \emph{topological K-homology}. This $\ao$-homotopy equivalence also proved that semi-topological K-homology is recovered as the topological realization of $\M_T$.

For all $\C$-dg-category $T$ we define two presheaves of Waldhausen categories. We denote by 
\begin{itemize}
\item $\uparf(T) : \spec(A)\longmapsto \parf(T\tel_\C A)=\parf(T,A)$. 
\item $\upspa(T) : \spec(A) \longmapsto \pspa(T\tel_\C A)=\pspa(T,A)$, where the latter is the category of $T\tel_\C A$-dg-modules which are perfect relative to $A$, which we call pseudo-perfect $T\tel_\C A$-dg-modules (voir \cite[Def 2.7]{modob}). 
\end{itemize}
These strict functors are obtained as canonical stictification of pseudo-functors for which functoriality is given by direct image. They give rise to two stacks, 
\begin{itemize}
\item $\M^T =Nw\uparf(T) : \spec(A)\longmapsto Nw\parf(T,A)$. 
\item $\M_T=Nw\upspa(T) : \spec(A)\longmapsto Nw\pspa(T,A)$. 
\end{itemize}
where $Nw$ stands for the nerve of the subcategory of equivalences. The direct sum of dg-modules induces a structure of homotopy coherent commutative monoid on $\uparf(T)$ and on $\upspa(T)$. We apply the functor $B_W$ defined in \ref{monoids} and we obtain special $\gam$-objects in $\spr$, 
$$\M^T_\bul = NwB_W\uparf(T), \qquad
\M_T^\bul =NwB_W\upspa(T).$$

Let $T\in \dgcatc$ be a dg-category. All statements in this subsection are also true for the stack $\M_T$ if we replace K-theory of perfect dg-modules by K-theory of pseudo-perfect dg-modules. We choose to write the details just for the stack $\M^T$. We use notations of example \ref{exespk} ; we have a $\del$-simplicial presheave given by 
$$\K^T_\bul:=\K_\bul (\uparf(T))=NwS_\bul \uparf(T)$$

By abuse of notation, we sometimes consider $\M^T$ as a $\del$-object applying the functor $\alpha^*$ defined in \ref{monoids}. We define a map of $\del$-objects

$$\lambda_\bul  : \M^T_\bul\lmo \K^T_\bul$$

by letting 
$$\lambda_n(a_1, \hdots, a_n) = (a_1\moi a_1\oplus a_2\moi \cdots \moi a_1\oplus \cdots \oplus a_n)$$
where we omit to precise the choices of sum diagrams which are part of the data, and which plays a role in that $\lambda_\bul$ is indeed simplicial. The fact that $\lambda_\bul$ is a map of simplicial objects is justified by the formulas in \cite[§1.2 p.29]{teze}. 

\begin{prop}\label{mk} 
The map $\lambda_\bul  : \M^T_\bul \lmo \K^T_\bul$ is a levelwise $\ao$-equivalence in $\del-\spretao$. 
\end{prop}

\begin{rema}
This result can be heuristically rephrased by saying that, in general, cofibrations are not split in the category $\parf(T)$, but if we look at the presheaf $\uparf(T)$ in the $\ao$-homotopy theory, then these cofibrations are all split up to $\ao$-homotopy equivalence. This explains the link with K-theory which is precisely the invariant through which cofibrations are split by the additivity theorem. 
\end{rema}

\begin{proof} of proposition \ref{mk} --- 
We introduce some notations. Let $n\geq 1$ be an integer. We denote by $[n]$ the category associated to the ordered set $\{1<2<\cdots <n\}$. We set
\begin{align*}
M_n :\affc^{op}  & \lmo Cat \\
\spec(A) & \longmapsto M_n(A)=\parf(T,A)^{[n-1]}.
\end{align*}
The latter object is the presheaf of sequences of length $n-1$ of composable maps in $\uparf(T)$. A map from $a_1\mo \cdots \mo a_n$ to $b_1\mo \cdots \mo b_n$ in $\parf(T,A)^{[n-1]}$ is by definition the data of commutative squares in $\parf(T,A)$, 
$$\xymatrix{a_1\ar[r] \ar[d]  & a_2 \ar[r] \ar[d] & \cdots \ar[r] & a_n \ar[d]  \\ b_1\ar[r]  & b_2 \ar[r] & \cdots \ar[r] & b_n}$$
We have $M_1=\uparf(T)$. For all $n\geq 1$ and all $A\in \calgc$, the category $M_n(A)$ is endowed with the projective model structure. Let $X_n=NwM_n$ be the simplicial presheaf which classifies sequences of length $n-1$ of composable maps in $\uparf(T)$. For all $n\geq 1$, we have a natural inclusion, 
$$\K^T_n \hookrightarrow X_n.$$
Because every map in $\parf(T,A)$ factorizes as a cofibration followed by a quasi-isomorphism, this last map is a global equivalence in $\spr$. Therefore, to prove our result, it suffices to prove that the map still denoted by $\lambda_n : \M^T_n \lmo X_n$ is an $\ao$-equivalence for all $n\geq 1$. We proceed by reccurence on $n$. We use level $2$ and $n-1$ to show level $n$. 
For level $1$ we have natural isomorphisms $\M^T_1 =X_1=Nw\parf(T,-)$. For level $2$ the map $\lambda_2$ acts on $0$-simplexes by
$$\lambda_2 (a,b)=(a\mo a\oplus b).$$
We then define an explicit $\ao$-homotopy inverse to $\lambda_2$ denoted by $\mu_2$ and defined on $0$-simplexes by
$$\mu_2(i:x\mo y):=(x,\cone(i)).$$
This defines naturally a map of simplicial presheaves $\mu_2 :  X_2\lmo \M^T_2$. We then have
$$\mu_2 \circ \lambda_2 (a,b)=\mu_2(a\mo a\oplus b)=(a,\cone(a\mo a\oplus b))\simeq (a,b),$$
where the last map is a quasi-isomorphism. For all $A\in \calgc$ we have an homotopy $\mu_2 \circ \lambda_2\Rightarrow id$ as endomorphisms of $\M^T_2(A)$. In the other direction we have
$$\lambda_2\circ \mu_2(i:x\mo y)=\lambda_2(x,\cone(i))=(x\mo x\oplus \cone(i)).$$
We then define an $\ao$-homotopy $h:\ao\times X_2\lmo X_2$ for every $A\in \calgc$ by
\begin{align*}
h_A :  A\times X_2(A)  & \lmo X_2(A) \\
(f,i:x\mo y) & \longmapsto (f i:x\mo y).
\end{align*}
The map $h$ is an $\ao$-homotopy between $id_{X_2}$ and the endomorphism $Z$ of $X_2$ defined by 
$$Z(i:x\mo y)=(0:x\mo y).$$
The endomorphism $Z$ is conjugated by an autoequivalence of $X_2$ with the map $\lambda_2\circ \mu_2$. This autoequivalence is given by the shift
\begin{align*}
t : X_2 & \lmo X_2\\
 (i:x\mo y)  & \longmapsto (y\mo \cone(i)),
\end{align*}
The map $t$ verifies $t^{(3)} (i)=i[1]$, and is therefore an autoequivalence of $X_2$. The inverse of $t$ is given by
$$t^{-1} (i:x\mo y)=\cocone(i)\mo x, $$
where the last map is given by the definition of the cocone. We have
\begin{align*}
tZt^{-1} (f) & = tZ(\cocone(i) \mo x) \\
& = t(0: \cocone(i)\mo x )\\
&= x\mo \cone(0:\cocone(i) \mo x). \\
\end{align*}
The module $\cone(0:\cocone(i) \mo x)$ is canonically quasi-isomorphic to $x\oplus \cone(i)$ with the sum differential. Thus we have a quasi-isomorphism $tZt^{-1} \simeq \lambda_2\circ \mu_2$. To sum up, the map $h$ is an $\ao$-homotopy $id_{X_2}\Rightarrow Z$, we have $tZt^{-1} \simeq \lambda_2\circ \mu_2$ and an homotopy $\mu_2 \circ \lambda_2\Rightarrow id$ which implies that $\lambda_2$ is an $\ao$-equivalence. 

Let now $n\geq 2$. We use the notion of pullback of model categories defined in \cite{dhall}. Consider the functor 
\begin{align*} 
F:M^{(n)}  & \lmo M^{(n-1)} \underset{M^{(1)}}{\ph} M^{(2)} \\
(a_1\mo a_2\mo \cdots\mo a_n) & \longmapsto ((a_1\mo \cdots\mo a_{n-2} \mo a_n), (a_{n-1}/a_{n-2} \mo a_n/a_{n-2}), a_{n-1}/a_{n-2}, id, id)
\end{align*}
where $a_{n-2} \mo a_n$ is the composite map $a_{n-2} \mo a_{n-1} \mo a_n$, and the notation quotient stands for the homotopy cofiber or the cone. 
\begin{lem}\label{lemrec}
The functor $F$ verifies the two assumptions of \cite[Lem 4.2]{dhall}. We deduce that the map induced by $F$,  
$$q_n:X_n \lmo X_n \underset{X_1}{\ph} X_2$$
is a global equivalence in $\spr$ (the pullback being calculated in the global model category $\spr$).  
\end{lem}

\begin{proof} of lemma \ref{lemrec} --- Toën's proof that $q_3$ is an equivalence (\cite{dhall} right after the proof of lemma 4.2) generalizes to sequences of maps of arbitrary length\footnote{We remark that there is a shift of indices between our notations and the paper \cite{dhall}.}. The major distinction comes from the fact that we work with a presheaf of categories of perfect objects and not with a stable model category. Nevertheless, the same proof makes sense for such perfect objects. Moreover, \cite[Lem. 4.2]{dhall} can be apply levelwise, and in the global model structure on $\spr$, a square is homotopy cartesian in and only if it is levelwise homotopy cartesian in $SSet$. 
\end{proof}

We therefore end up with a square of simplicial presheaves,
$$\xymatrix@R=1cm @C=3cm { \M^T_n \ar[r]^-{\lambda_n } \ar[d]^-{p} & X_n \ar[d]^-{q_n} \\ \M^T_{n-1} \ph \M^T_{2} \ar[r]^-{\lambda_n \ph  \lambda_2 } & X_{n-1} \underset{X_1}{\ph} X_1 }$$
where $p$ is the map $p(a_1, \hdots, a_n)=((a_1,\hdots, a_{n-2},a_{n-1}\oplus a_n), (a_{n-1}, a_n))$. This latter is an equivalence by its very definition. We can check directly that the square is commutative up to global homotopy. The maps $\lambda_{n-1}$ and $\lambda_2$ are $\ao$-homotopy equivalences by reccurence hypothesis. The $\ao$-homotopy equivalences being stable by homotopy pullback, the map $\lambda_2 \ph  \lambda_2$ is an $\ao$-homotopy equivalence. We conclude by the 2-out-of-3 property that the map $\lambda_n$ is an $\ao$-homotopy equivalence. The proof of \ref{mk} is then complete. 
\end{proof}

\begin{prop}\label{mttsp}
Let $T\in \dgcatc$ be a dg-category over $\C$. Then the special $\gam$-space $\regam{\M^T_\bul}$ is very special. 
\end{prop} 

\begin{proof} 
We use formula of proposition \ref{pi0cor}. We have an isomorphism of sets $\pi_0\re{\M^T_1}\simeq \pi_0\M^T(\C)/\sim$ where two class of dg-modules $[E]$ and $[E']$ are equivalent if there exists a connected algebraic curve which connects the two dg-modules in $\M^T(\C)$. Let $E$ be a perfect $T^{op}$-module. Let
$$\delta : \ao\lmo \M^T$$
the map such that for all $A\in \calgc$, 
$$\delta_A(f)=\cone(E\lmos{\times f} E).$$
Then we have $\delta_A(0)=\cone(0:E\lmo E)=E\oplus E[1]$ and $\delta_A(1)=\cone(id_E)$ which is canonically isomorphic to $0$. We proved that the identity $[E\oplus E[1]]=[0]$ is valid in the monoid $\pi_0\re{\M^T_1}$, which is therefore a group. 
\end{proof}

\begin{theo}\label{kstmt}
Let $T\in \dgcatc$. Then there exists a canonical isomorphism,
$$\kcst(T)\simeq \mcal{B} \regam{\M^T_\bul}$$ 
in $Ho(Sp)$. 
\end{theo} 

\begin{proof} 
We have a chain of equivalences
$$\kcst(T) = \resp{\ukc(T)} =\resp{\mcal{B} K^\gam(\parf(T,-))} \simeq \mcal{B}\regam{K^\gam(\parf(T,-))}$$ 
where the latter equivalence comes from proposition \ref{proprelb}. Set $K^\gam(\parf(T,-))=:K^\gam (T,-)$. We consider the map of $\gam$-simplicial presheaves, 
$$\sigma : \M^T_\bul \lmo K^\gam(T,-)$$
defined as the map (\ref{mapgam}). We want to show that $\sigma$ induces an equivalence on topological realizations. Since we deal with special $\gam$-objects , it suffices to prove that we have an equivalence on the level $1$. We have a commutative diagram in $Ho(SSet)$, where we intentionally omit the indices $\del$ and $\gam$ from the notation,

$$\xymatrix{\re{\M^T_1} \ar[r]^-{\re{\sigma}_1} \ar[dd]_-\wr &\re{K^\gam(T,-)_1} \ar[rd]^-\sim \\ & &  \re{K(T,-)} \\ \re{\M^T_\bul}^+_1  \ar[r]^-{\re{\lambda}^+_1} & \re{\K^T_\bul}^+_1 \ar[ur]^-\sim}$$
where the map of $\del$-objects $\lambda$ is the one from proposition \ref{mk}. By this latter proposition and theorem \ref{di}, the induced map $\re{\lambda}^+_1$ is an equivalence in $SSet$. The left vertical map is an equivalence by proposition \ref{mttsp}. On the other hand, the $\del$-objects $\re{K^\gam(T,-)}$ and $\re{\K^T_\bul}^+$ have the same level $1$, which is equivalent to $\re{K(T,-)}$. We then deduce that $\re{\sigma}_1$ is an equivalence in $SSet$ and therefore that $\re{\sigma}$ is an equivalence, which proves the expected formula. 
\end{proof}

\begin{theo}\label{mtps}
Let $T\in \dgcatc$. Then the special $\gam$-space $\regam{\M_T^\bul}$ is very special and there exists a canonical isomorphism, 
$$\resp{\kc(\upspa(T))}\simeq \mcal{B} \regam{\M_T^\bul}$$
in $Ho(Sp)$. 
\end{theo}

\begin{proof} 
The proofs of \ref{mk}, \ref{mttsp}, and \ref{kstmt} works the same if we replace perfect dg-modules by pseudo-perfect ones. 
\end{proof}



\subsection{Topological Chern character}

In this part we give the construction of the topological Chern character or topological Chern map. Let $T\in\dgcatc$ be a $\C$-dg-category. 
We recall that we defined in \ref{defchernlin} an algebraic Chern map which is a map of $\ka$-modules spectra, functorial in $T$, 
$$\ch_T : \ukn(T)\lmo \uhcn(T). $$
Composing this map with the map of $\ka$-modules $\uhcn(T)\lmo \uhp(T)$, we obtain a map of $\ka$-modules, 
$$\ukn(T)\lmo \uhp(T).$$
We then apply the spectral topological realization $\resp{-}$ to obtain a map of $\bu$-modules, 
$$\kst(T)=\resp{\ukn(T)}\lmo \resp{\uhp(T)}.$$
We now use a Künneth type formula for periodic cyclic homology. The presheaf $\uhp(T)$ is given by $\spec(A)\longmapsto \hp(T\tel_\C A)$. Kassel's theorem \cite[Thm 2.3]{kasselcyclic} and \cite[Prop 2.4]{kasselcyclic} implies that for any smooth commutative $\C$-algebra $A$ the natural map of spectra, 
$$\hp(T)\sml_{H\cuu}\hp(A)\lmo \hp(T\tel_\C A), $$
is an equivalence in $Sp$. This implies that the map of presheaves of spectra
$$\hp(T)\sml_{H\cuu} \uhp(\ast) \lmo \uhp(T)$$
is an equivalence on smooth affine schemes in $\spaf$. Theorem \ref{restliss} implies that the map induced on spectral topological realization, 
$$\hp(T)\sml\resp{\uhp(\ast)} \simeq\resp{\hp(T)\sml_{H\cuu} \uhp(\ast)} \lmo \resp{\uhp(T)}$$
is an equivalence in $Sp$. Therefore we have an isomorphism
$$\resp{\uhp(T)}\simeq \hp(T)\sml_{H\cuu} \resp{\uhp(\ast)}$$
in $Ho(Sp)$. By composition we have map 
\begin{equation}\label{mor1}
\kst(T)\lmo \hp(T)\sml_{H\cuu}\resp{\uhp(\ast)}
\end{equation}
which defines a natural transformation between objects of $Ho(Sp^{\dgcatc})$. In consequence to obtain a map with target $\hp(T)$, we have to choose a map $\resp{\uhp(\ast)}\lmo H\cuu$. By adjunction, it remains to choose a map $\uhp(\ast)\lmo (H\cuu)_{\s,B}$. The presheaf of spectra $(H\cuu)_{\s,B}$ is given by
$$X\longmapsto \rhomi_{Ho(Sp)} (\resp{X}, H\cuu) \simeq \rhomi_{Ho(Sp)} (\resp{X}, H\C)\sml_{H\C} H\cuu,$$
which is the $2$-periodic Betti cohomology of the scheme $X$ with coefficients in $\C$. We denote by $\hb(-,\C)\sm H\cuu$ the latter presheaf. We denote by $\hb(-,\C)$ the usual Betti cohomology with coefficients in $\C$, i.e. the presheaf of $H\C$-modules spectra
$$X\longmapsto \rhomi_{Ho(Sp)} (\resp{X}, H\C)=:\hb(X,\C),$$
whose homotopy groups are the Betti cohomology $\C$-vector spaces of $X$. We denote by $\hpa$ the presheaf $\uhp(\ast) : \spec(A)\mapsto \hp(A)$. We thus have to choose a map 
$$\hpa\lmo \hb(-, \C)\sm H\cuu.$$ 
We consider the standard antisymmetrisation map, 
$$\hpa\lmo \hdrna$$
which goes from periodic cyclic homology to naive de Rham cohomology. By naive de Rham cohomology we mean the presheaf of spectra $X\mapsto \hdrna(X)$, such that $\hdrna(X)$ is the spectrum associated to the algebraic de Rham complex of $X$, i.e. the complex of $\C$-vector spaces of algebraic differential forms everywhere defined on $X$. We denote by $\hdran$ the analytic analog of $\hdrna$ constructed out of analytic differential forms. The inclusion of algebraic differential forms into analytic differential forms induces a map $\hdrna\lmo \hdran$. The evident map from $\C$ to the analytic de Rham complex induces a map $\hb(-,\C)\lmo \hdran$ which is an equivalence on smooth schemes\footnote{This is true by the classical fact that for a smooth complex variety, the complex of sheaves of analytic differential forms is an injective resolution of the constant sheaf $\und{\C}$.}. We then have maps of presheaves of spectra,
\begin{equation}\label{mapcoh}
\hpa\lmo \hdrna \lmo \hdran\longleftarrow \hb(-,\C)\lmo \hb(-,\C)\sm H\cuu,
\end{equation}
and the map which goes from right to left is an equivalence on smooth schemes\footnote{In fact the first two maps are also equivalences on smooth affine schemes by the HKR theorem and the Grothendieck theorem respectively, but we won't need this fact.}.

\begin{nota}
We denote by $\sppro$ the proper local model structure on the category $Sp(\schc)$ of presheaves of symmetric spectra on the category of separated schemes of finite type over $\C$. I.e. the local model structure with respect to the proper topology on the category $\schc$. 
\end{nota}

We naturally consider the maps (\ref{mapcoh}) as maps in the homotopy category $Ho(\sppro)$ of presheaves of spectra on $\schc$ (with respect to the proper local model structure) ; recall remark \ref{remhhsch} for the cyclic homology of schemes. By remark \ref{hiro}, the map $\hb(-,\C) \lmo \hdran$ is a proper local equivalence, and is an isomorphism in $Ho(\sppro)$. We thus obtain a map $\hpa\lmo \hb(-,\C)$ in $Ho(\sppro)$ and a map
$$\hpa\lmo \hb(-, \C)\sm H\cuu$$
in $Ho(\sppro)$. By the spectral version of proposition \ref{quipro}, the map 
$$\hpa\lmo \hb(-, \C)\sm H\cuu = \hbs (H\cuu)$$ 
gives by adjunction the expected map
$$\mcal{P} : \resp{\hpa}=\resp{\uhp(\ast)}\lmo H\cuu$$
in $Ho(Sp)$. By composing the map (\ref{mor1}) with what we just found we obtain a map
$$\ch^{\mrm{st}}_T : \kst(T)\lmo \hp(T)$$
defined as the composite
$$\xymatrix{  \kst(T)\ar[d]_-{\resp{\ch_T}}  \ar[rr]^-{\ch^{\mrm{st}}_T} &&  \hp(T) \\
 \resp{\uhp(T)} \ar[r]^-\sim & \hp(T)\sml_{H\cuu} \resp{\uhp(\ast)} \ar[r]^-{id\sml \mcal{P}} & \hp(T)\sml_{H\cuu} H\cuu \ar[u]_-\wr }$$
By what has been said in \ref{kstktop}, the spectral topological realization can be extended on categories of modules, 
$$\resp{-} : \ka-Mod_\s\lmo \bu-Mod_\s.$$
We deduce from this that in the previous rectangle, all maps are maps of $\bu$-modules and we obtain in this way a map
$$\chst : \kst\lmo \hp$$
in $Ho(\bu-Mod_\s^{\dgcatc})$. By abuse of notation we denote by $\chst : \kst(T)\lmo \hp(T)$ omitting the indice $T$ in the notation. We remark that for any $\C$-dg-category $T$ we have a commutative square
\begin{equation}\label{carst}
\xymatrix{\kn(T) \ar[r]^-{\ch} \ar[d]_-\theta & \hp(T) \ar[d]^-{id} \\ \kst(T) \ar[r]^-{\chst}  & \hp(T)  }
\end{equation} 
in $Ho(\bu-Mod_\s)$, where $\theta$ is the natural map defined at remark \ref{fibtop} and the composite
$$\xymatrix{\hp(T)\ar[r] & \resp{\uhp(T)}\ar[r] & \hp(T)\sml_{H\cuu} \resp{\uhp(\ast)} \ar[r]^-{id\sml \mcal{P}} &\hp(T)}$$
is equal to the identity in $\End_{Ho(Sp)}(\hp(T))$. Now it remains to verify that the image $\chst(\beta)$ of the Bott generator is invertible in the ring $\hp(\ast)=H\cuu$. We use notations \ref{choixbeta}. We follow the Bott generator in the K-theory of $\po$. As particular case of the square (\ref{carst}), we have a commutative square of abelian groups
$$\xymatrix{\kn_0(\po) \ar[r]^-{\ch} \ar[d]_-{\eta} & \hp_0(\po) \ar[d]^-{id} \\ \kst_0(\po) \ar[r]^-{\chst} & \hp_0(\po) }$$
We recall that $\eta(\alpha)=\beta$ and we choose for example a generator $u\in\hp_0(\po)$ such that $\ch(\alpha)=u$. We then verify that $\chst(\beta)=\chst(\eta(\alpha))=\ch(\alpha)=u$. 
By universal property we obtain that for any $\C$-dg-category $T$ a map
$$\ch^{\mrm{top}}_T : \ktop(T)\lmo \hp(T).$$
This defines a map
$$\chtop : \ktop\lmo\hp$$
in $Ho(\BU-Mod_\s^{\dgcatc})$.  

\begin{theo}\label{carac}
There exists a map $\chtop : \ktop\lmo\hp$ called the \emph{topological Chern map} such that the square
\begin{equation}\label{sqcarac}
\xymatrix{ \kn \ar[r]^-{ \ch} \ar[d] & \hcn\ar[d] \\  \ktop \ar[r]^-{ \chtop} &\hp  }
\end{equation}
is commutative in $Ho(Sp^{\dgcatc})$. 
\end{theo}

\begin{proof} 
Comes immediatly from the commutativity of the square (\ref{carst}).
\end{proof}

\subsection{Conjectures}\label{conj}

The following conjectures are all analogs of known facts for smooth proper algebraic varieties. The first one concerns the rational part of the hypothetic noncommutative Hodge structure for a smooth proper dg-category over $\C$. We recall that a dg-category $T$ is said to be proper if its complexes of morphisms are perfect complexes and if the triangulated category $[\m{T}]$ has a compact generator. A dg-category $T$ is said to be smooth if the $T^{op}\tel T$-module $(x,y)\mapsto T(x,y)$ is perfect. 
It is proved in \cite[Cor 2.13]{modob} that a smooth proper dg-category is of finite type, and therefore is equivalent to a homotopically finitely presented, smooth and proper dg-algebra. To prove the following conjectures, it suffices to prove it for smooth and proper dg-algebras. 

\begin{conj}\label{conjres} \emph{(The lattice conjecture)} --- 
Let $T$ be a smooth proper dg-category over $\C$. Then the map
$$\chtop\sm_\s H\C : \ktop(T)\sm_{\s} H\C \lmo \hp(T)$$
is an equivalence. 
\end{conj} 

We remark that the class of dg-categories verifying this lattice conjecture is stable by the operations of filtrant colimits, retracts, quotient, and extension in the Morita homotopy category $Ho(\dgmorc)$. The following conjecture is a generalization of theorem \ref{annupoint} for smooth and proper dg-categories.

\begin{conj} Let $T$ be a smooth proper dg-category over $\C$. Then $\kst_i(T)=0$ for all $i<0$. 
\end{conj}

Denis-Charles Cisinski has recently proved that the vanishing of the negative algebraic K-groups of smooth proper dg-algebras whose cohomology is concentrated in positive degrees. This proof is based on the Schlichting's proof of the vanishing the negative algebraic K-groups of a noetherian abelian category. 

The following conjecture is inspired by Thomason's result cited in the introduction \cite[Thm 4.11]{thomet} and by Friedlander--Walker analog result for certain type of projective varieties (see \cite[Cor 3.8]{fwcomp}). We recall that given a spectrum $E$ and an integer $n\in \Z$, we can, because the stable homotopy category $Ho(Sp)$ is additive, give a meaning to the reduction modulo $n$ of the spectrum $E$ denoted by $E/n$. We state this conjecture with any dg-category -- let us precise that we at least expect very general assumptions on the type of dg-categories verifying this conjecture (as I was told by DC Cisinski), because it would only involve the $\ao$-invariance of algebraic K-theory with torsion coefficients and Gabber rigidity theorem. 

\begin{conj}\label{coefffinis}
Let $T$ be a dg-category over $\C$ and $n>0$ an integer. Then the map $\kn(T)/n\lmo \kst(T)/n$ is an equivalence. 
\end{conj} 

\subsection{Schemes and Deligne cohomology}

The aim is to compare the topological K-theory of the dg-category of perfect complexes on a $\C$-scheme with the usual topological K-theory of its complex points together with their Chern maps, thus completing the point b in theorem \ref{motiv}. The proof is based on Riou's Spanier--Whitehead duality in Morel--Voevodsky's stable homotopy theory of smooth schemes over $\C$. 

At the end we give a comparison result for Deligne cohomology of the dg-category of perfect complexes, showing that we actually recover the commutative Deligne cohomology. 

\begin{nota}
For any $\C$-scheme $X$, we denote by $\ktop(X):=\ktop(\lpe(X))$ the topological K-theory of its dg-category of perfect complexes. 
\end{nota}

\begin{nota}
For a topological space $Y\in Top$, we denote by $\ktopu(Y):=\rhomi_{Ho(Sp)} (\sinf SY_+, \BU)$ its topological K-theory spectrum (nonconnective). If $X\in\schc$ is a $\C$-scheme, we can consider the topological K-theory of its complex points  $\ktopu(sp(X))$. 
\end{nota}

\begin{nota}
For a scheme $X\in\schc$, and any $\Z$-module $A$, we denote by $\hb(X, A)=\rhomi_{Ho(Sp)} (\resp{X}, HA)$ its Betti cohomology with coefficients in $A$. For smooth $X$, we have an isomorphism $\hp(X)\lmo \hb(X, \cuu)$ given by the composite of the antisymmetrisation map and the usual isomorphism between de Rham and Betti cohomology. 
\end{nota}

\begin{nota}
For every topological space $Y$, the usual topological Chern map
$$\chutop : \ktopu(Y) \lmo \hb(Y,\cuu)$$ 
is defined by $\chutop=\rhomi_{Ho(Sp)}(\sinf SY_+,\ch^{\mrm{top}}_\unit))$, where $\ch^{\mrm{top}}_\unit:\BU\lmo H\cuu$ is the topological Chern map of the point (Thm \ref{carac}) which is a map of ring spectra. The identity $\ch^{\mrm{top}}_\unit(\beta)=u$ characterizes uniquely the Chern map in the set of homotopy classes of ring maps $[\BU,H\cuu]$. This latter set is isomorphic to the set $[\BU\sm H\C, H\cuu]$. The standard equivalence between ring spectra and dg-algebras maps the ring spectrum $\BU\sm H\C$ to the polynomial dg-algebra $\C[\beta,\beta^{-1}]$, where the degree of $\beta$ is $2$. The image of $\beta$ is nothing but a non-zero multiple of $u$ ; the usual Chern map corresponds to $\chutop(\beta)=u$. 
\end{nota}

\begin{prop}\label{compvarlisse}
Let $X$ be a separated $\C$-scheme of finite type. Then there exists a canonical isomorphism 
$$\ktop(X)\lmos{\sim}\ktopu(sp(X))$$
in $Ho(Sp)$. Moreover the square, 
\begin{equation}\label{utop}
\xymatrix{ \ktop(X) \ar[r]^-{\chtop} \ar[d]_-\wr & \hp(X) \ar[d]  \\ \ktopu(sp(X)) \ar[r]^-{\chutop} & \hb(X,\C)\sm H\cuu }
\end{equation}
is commutative in $Ho(Sp)$. This implies that the lattice conjecture \ref{conjres} is valid for dg-categories of the form $T=\lpe(X)$ with $X$ any separated $\C$-scheme of finite type. 
\end{prop} 

The proof occupies the sequel of this part. We begin with some notations and reminders. Recall from notation \ref{notaksch} that we have a functor "algebraic K-theory of schemes",
$$\kn : \schc^{op}\lmo Sp.$$
For a scheme $X$ we define a presheaf of K-theory by the formula
$$\ukn(X):=(\spec(A)\longmapsto \kn(X\times_\C \spec(A)).$$
This defines a functor
$$\ukn : \schc^{op}\lmo \spaf.$$
Recall from notation \ref{notaksch} the functor "perfect complexes", 
$$\lpe : \schc^{op}\lmo \dgcatc.$$
For every $X\in \schc$ and every $\spec(A)\in \affc$, there exists a map
\begin{equation}\label{smmon}
\lpe(X)\tel_\C A\lmo \lpe(X\times_{\C} \spec(A))
\end{equation}
in $\dgcatc$ given by pulling back perfect complexes along the projection map $X\times_{\C} \spec(A)\lmo X$. The map (\ref{smmon}) is known to be a Morita equivalence of dg-categories (this uses the fact that we work over a field, and hence that at least one of the two schemes is flat over the base). It can be proved by reducing the assumption to the affine case (using the descent property of $\lpe$) for which it is immediate. We deduce that by proposition \ref{propalgk}, when $\spec(A)$ is smooth, there exists an equivalence 
$$\kn(X\times_\C \spec(A))\simeq \kn(\lpe(X)\tel_\C A)$$ 
in $Sp$. This implies the existence of an isomorphism of presheaves
$$\ukn(X)\simeq \ukn(\lpe(X))$$
in $Ho(\spafliss)$. 

We denote by $\splissna$ the $\ao$-Nisnevish local model structure on $\spafliss$. This is a monoidal model category and we denote by $\sm$ its monoidal product. We denote by $\shc$ Morel--Voevodsky's stable homotopy category of smooth $\C$-schemes. It can be define in the following way. The letter $T:=S^1\sm\gm$ stands for the Tate sphere in $\splissna$, the category $\shc$ is defined here as the homotopy category of symmetric $T$-spectra in $\splissna$,
$$\shc:=Ho(Sp_T \splissna).$$
This is a closed symmetric monoidal category. Riou's statement about Spanier--Whitehead duality says that every object of $\shc$ of the form $\sinf_{T, S^1} X_+$ for $X$ a smooth $\C$-scheme of finite type is (strongly) dualizable in the monoidal category $\shc$ (see \cite{rioudual}).

We choose a presentation of the scheme $\gm=\C\te_{\Z}\Z[t, t^{-1}]$. The invertible function $t$ gives a class $b$ in the group $\kn_1(\gm)$ and we choose $t$ such that the canonical map
$$\kn_1(\gm)\lmo \kn_0(\po)\simeq \Z\oplus \alpha\Z$$ 
sends the class $b$ on $\alpha$ (which was choosen in notation \ref{choixbeta}). 
We denote by $\kh$ the object of $\shc$ whose underlying $S^1$-spectrum is the presheaf $\kn \in \spafliss$, endowed with the $T$-spectrum structure given by the map $T=S^1\sm \gm \lmo \kn$ which corresponds to the class $b\in\kn_1(\gm)$. More precisely, for all $n\geq 0$ the class $b$ gives rise to a map
$$\xymatrix{T\sml \kn \ar[r]^-{b\sml id} &  \kn\sml\kn \ar[r] & \kn}$$
where the last map is given by the ring structure of $\kn$ in $Ho(\spafliss)$. Taking cofibrant-fibrant replacement and applying \cite[Prop 2.3]{cisdes} we obtain a $T$-spectrum $\kh$ with $\kh(Y)_n=\kn(Y)$ for all $n\geq 0$. A fibrant model of $\kh$ in $\shc$ represents Weibel's homotopy invariant algebraic K-theory, as proved by Cisinski (\cite[Thm 2.20]{cisdes}). For all smooth $\C$-scheme of finite type $X$ we denote by
$$\ukh(X):=\rhomi_{\shc} (\Sigma^\infty_{T, S^1} X_+ , \kh).$$
By Yoneda's lemma in $\shc$, this latter is globally equivalent to the presheaf of $T$-spectra 
$$Y\longmapsto \kh(X\times Y),$$
in $\spafliss$. We have an equivalence $\ukh(\unit)\simeq\kh$. The topological realization of $T$ is given by
$$ssp(T)\simeq ssp(S^1\sm \gm) \simeq S^1\sm ssp(\gm)\simeq S^1\sm S^1=S^2.$$
We locally adopt the more precise notation $Sp=\sps$ for symmetric $S^1$-spectra and we denote by $\spss\sps$ the model category of symmetric $S^2$-spectra inside symmetric $S^1$-spectra. The infinite loop space functor $\Omega^\infty_{S^2} : \spss\sps\lmo \sps$ is a Quillen equivalence. There exists a topological realization functor on the level of the category of $T$-spectra $\shc$ ; It is simply defined levelwise. We denote it by
$$\resh{-} : \shc\lmo Ho(\spss\sps).$$
We have a noncommutative square of categories
$$\xymatrix{ \shc \ar[r]^-{\resh{-}} \ar[d]_-{\R\Omega_T^\infty} & Ho(\spss\sps) \ar[d]^-{\R\Omega_{S^2}^\infty} \\ Ho(\splissna) \ar[r]^-{\resp{-} } & Ho(Sp_{S^1} )  }$$
where $\R\Omega_T^\infty$ stands for the derived loop space functor for $T$-spectra. The structure of $S^2$-spectrum of $\resh{\ukh(X)}$ corresponds to the multiplication by the Bott generator $\beta\in \pi_2\resp{\ukn(\unit)}$ inside the $S^1$-spectrum $\resp{\ukn(X)}$. This is because the structure of $T$-spectrum of $\ukh(X)$ is given by the multiplication by the class $b$ and because we have also an equivalence $ssp(T)\simeq S^2$ such that the triangle
$$\xymatrix{  \sinf(S^2)_+\ar[rd]^-{\beta} \ar[d]^-{\wr} & \\  \resp{T} \ar[r]^-{\re{b}} & \resp{\ukn(\unit)}\simeq \bu  }$$
is commutative by the choice of $b$. In consequence, applying the functor $\R\Omega_{S^2}^\infty$ to $\resh{\ukh(X)}$ corresponds to inverting $\beta$ in $\resp{\ukn(X)}$. We therefore have a canonical map
$$\resp{\ukn(X)}[\beta^{-1}] \lmo \R\Omega_{S^2}^\infty\resh{\ukh(X)},$$
which is an isomorphism in $Ho(Sp_{S^1})$ because the underlying $S^1$-spectrum of $\ukh(X)$ is equivalent to $\ukn(X)$.

\begin{proof} of proposition \ref{compvarlisse} ---
We first deal with the smooth case by using Riou's Spanier--Whitehead duality in the homotopy category of smooth schemes, and then we prove it for possible singular schemes by using cdh descent. Let $X$ be a smooth separated $\C$-scheme of finite type. We have canonical ismorphisms in $Ho(\sps)$, 
\begin{align*}
\ktop(X) & =\resp{\ukn(X)}[\beta^{-1}] \\
& \simeq \R\Omega_{S^2}^\infty\resh{\ukh(X)} \\
& = \R\Omega_{S^2}^\infty\resh{\rhomi_{\shc} (\Sigma^\infty_{T, S^1} X_+ , \kh) }
\end{align*}

By Riou's theorems \cite[Thm 1.4+Thm 2.2]{rioudual}, the object $\Sigma^\infty_{T, S^1} X_+$ is strongly dualizable in $\shc$. The functor $\resh{-}$ being monoidal, and because a monoidal functor commutes with the duality functor, we have canonical isomorphisms in $Ho(\sps)$, 
\begin{align*}
\ktop(X) & \simeq \R\Omega_{S^2}^\infty\rhomi_{Ho(\spss\sps)} (\resh{\Sigma^\infty_{T, S^1} X_+} , \resh{\kh}) \\
& \simeq \R\Omega_{S^2}^\infty\rhomi_{Ho(\spss\sps)} (\Sigma^\infty_{S^2, S^1} \re{X}_+ , \resh{\kh}) \\
& \simeq \rhomi_{Ho(\sps)} (\Sigma^\infty_{S^1} \re{X}_+ , \R\Omega_{S^2}^\infty\resh{\kh}) \\
\end{align*}
Besides we have isomorphisms in $Ho(\sps)$,
$$\R\Omega_{S^2}^\infty\resh{\kh}\simeq \ktop(\unit)=\BU.$$
This proves the existence of the expected isomorphism
\begin{equation}\label{complissiso}
\ktop(X)\simeq \rhomi_{Ho(\sps)} (\Sigma^\infty_{S^1} \re{X}_+ , \BU)=\ktopu(sp(X))
\end{equation}
It remains to compare the two Chern map up to homotopy. The square (\ref{utop}) decomposes into
$$\xymatrix@=5pc{\ktop(X)\simeq \resp{\rhomi (X, \kn)}[\beta^{-1}] \ar[r]^-{\resp{\rhomi(X,\ch)}} \ar[d] & \resp{\rhomi (X, \hpa)} \ar[d] \\
\ktopu(sp(X))=\rhomi (\resp{X}, \resp{\kn}[\beta^{-1}])  \ar[r]^-{\rhomi(\resp{X}, \resp{\ch})} \ar@/_1pc/[dr]_-{\rhomi(\resp{X}, \chtop)} & \rhomi(\resp{X}, \resp{\hpa}) \ar[d]^-{\rhomi(\resp{X}, \mcal{P})} \\
 & \rhomi(\resp{X}, H\cuu)=\hb(X, \cuu)  }$$ 
In this diagram, the top square is commutative by functoriality of $\resp{-}$. Moreover the bottom triangle is commutative because it is commutative for $X=\spec(\C)$ by definition of $\chtop$.

Let now $X$ be a non necessarily smooth $\C$-scheme of finite type. We consider the cdh topology on the category $\schc$ and the corresponding homotopy category $Ho(Sp(\schc)^{cdh})$ defined using cdh-local equivalences. We denote by $Ho(Sp(\lissc)^{cdh})$ the smooth version. By using the same argument as in the proof of proposition \ref{eqpreliss} --- mainly, resolution of singularities --- we obtain that the restriction and extension functors induce an equivalence of categories
\begin{equation}\label{extliss}
Ho(Sp(\schc)^{cdh}) \simeq Ho(Sp(\lissc)^{cdh}).
\end{equation}
Haesemeyer \cite[Thm 6.4]{haes} has proven that $\ao$-invariant algebraic K-theory has cdh descent on singular schemes of finite type over $\C$. Since the topological K-theory $\ktop(\lpe(X))$ is the Bott inverted topological realization of the $\ao$-invariant algebraic K-theory defined on singular schemes, we have that topological K-theory has cdh descent too. Moreover $\ktopu(sp(-))$ has also cdh descent because of theorem \ref{di} and of theorem \ref{hypdescpro} which imply that it has étale descent and proper descent respectively. Thus we can consider the isomorphism $\ktop(\lpe(-))\lmo \ktopu(sp(-))$ defined above in (\ref{complissiso}) as a well defined isomorphism in $ Ho(Sp(\lissc)^{cdh})$. By the equivalence (\ref{extliss}), it extends uniquely up to an isomorphism in $Ho(Sp(\schc)^{cdh})$. This proves the comparison result, and with the same argument we obtain a commutative square (\ref{utop}) of Chern characters in $Ho(Sp(\schc)^{cdh})$ (using that periodic cyclic homology has cdh descent by \cite[Cor 3.13]{cycliccdh}). The lattice conjecture is then true for dg-categories $T=\lpe(X)$ with $X$ separated of finite type over $\C$ because it is true for the usual topological Chern character.  
\end{proof}

We can now justify the definition of Deligne cohomology of smooth and proper dg-categories given in the introduction. 

\begin{df} 
The \emph{Deligne cohomology of a smooth proper $\C$-dg-category $T$} is the symmetric spectrum 
$$\hdel(T):=\ktop(T)\ph_{\hp(T)} \hcn(T)$$
where the maps defining the homotopy pullback are those of the square (\ref{sqcarac}).
\end{df} 

\begin{nota} We recall the definition of the Deligne cohomology of a smooth and proper $\C$-scheme $X$. If $A$ is a $\Z$-module, for each integer $p\in\N$ we adopt the standard notation $A(p):=(2i\pi)^p A\subseteq \C$. We denote by $\und{A}_X$ the constant sheaf associated to $A$ on $X$. The $p$th Deligne complex of $X$ is by definition the complex of $\Z$-sheaves
$$\Z(p)_\D (X)= (\und{\Z(p)}_X \lmo \ocal_X \lmos{d} \Omega^1_X \lmos{d} \Omega^2_X \lmos{d} \cdots)$$
where $\Omega_X^q$ is the sheaf of Kähler $q$-forms on $X$, and $d$ is the de Rham differential.The \emph{Deligne cohomology of $X$} is by definition the complex of $\Z$-modules
$$\hdel^* (X, \Z):=\prod_{p\geq 0} \mbb{H}^*(X,\Z(p)_\D(X))[2p]$$
i.e. the hypercohomology of $X$ with coefficients in the total Deligne complex $\prod_{p\geq 0} \Z(p)_\D(X)[2p]$\footnote{The use of the shift comes from the comparison with negative cyclic homology.} (see for example \cite[§1]{evdelignecoh}). This is a reindexed version of the usual Deligne cohomology. The $H^0(\hdel^* (X, \Z))$ gives the group $\prod_{p\geq 0} \mbb{H}^{2p}(X,\Z(p)_\D(X))$. We denote by $\hdel(X,\Z):=H(\hdel^{-*} (X, \Z))$ the associated symmetric spectrum (see §\ref{caracalg} for the functor $H:C(\Z)\lmo Sp$). 

We recall that the $p$th Deligne complex of sheaves $\Z(p)_\D(X)$ is given by the following homotopy quotient
$$\Z(p)_\D(X) = \cone (\und{\Z(p)}_X\oplus F^p\drcx \lmo \drcx)$$
in the derived category $D(X,\Z)$ of complex of $\Z$-sheaves on $X$, where $F^p\drcx=\Omega_X^{\geq p}$ means the $p$th layer of the Hodge filtration on the de Rham complex $\drcx$, which is just the truncation (see \cite[§2.7]{evdelignecoh}). If $E=E^\ast\in D(X,\Z)$ is a cochain complex of $\Z$-sheaves on $X$, we adopt the notation $E^*[u^{\pm 1}]=\prod_{i\in \Z} E^{\ast+2i}=\prod_{i\in\Z}E^\ast[-2i]$ for the $2$-periodic complex associated to $E$. We also note $H(E[u^{\pm 1}]) =:H(E)[u^{\pm 1}]$. 

Because the topological K-theory and the Betti cohomology become isomorphic only with rational coefficients, we will work with Deligne cohomology with rational coefficients which is denoted by $\hdel^* (X, \Q)\in C(\Q)$ and $\hdel(X,\Q)\in Sp$. 
\end{nota}

\begin{rema}\label{remindhkr} Let $X$ be a separated smooth $\C$-scheme of finite type. We denote by $\hp_*(X)$ (resp. $\hcn_*(X)$) the complex $\Z$-modules which calculates the periodic cyclic homology of the scheme $X$ (resp. the negative cyclic homology of $X$). I.e. such that $H(\hp_*(X))=\hp(X)$ and $H(\hcn_*(X))=\hcn(X)$. Recall remark \ref{remhhsch} for the Hochschild homology of schemes. By the Hochschild--Konstant--Rosenberg theorem (see \cite{hkr} or \cite[Thm 3.4.4 p.102]{loday}, and \cite[Ex 2.7]{weibelhodge} for the case of non necessarily affine schemes), there exists an isomorphism 
$$\hp_*(X)\simeq \mbb{H}^{-*}(X,\drcx)[u^{\pm 1}]$$
in $D(\C)$ with the $2$-periodic hypercohomology of $\drcx$ which is nothing but $2$-periodic de Rham cohomology of $X$. The negative theory is itself related to the Hodge filtration on the de Rham complex. That is, by \cite[Thm 3.3]{weibelhodge}, we have an isomorphism 
$$\hcn_\ast (X) \simeq \prod_{p\geq 0} \hcn_\ast^{(p)} (X)$$
in $D(\C)$, where for each integer $p\geq 0$, $\hcn_\ast^{(p)} (X)$ is the $p$th component of the Hodge type decomposition on $\hcn_\ast (X)$, i.e. 
$$\hcn_\ast^{(p)} (X)\simeq \mbb{H}^{-*}(X, F^p\drcx)[-2p]$$
in $D(\C)$. 
\end{rema}

\begin{prop} 
Let $X$ be a separated, smooth and proper  $\C$-scheme of finite type. Then there exists a canonical isomorphism 
$$\hdel(\lpe(X))\sm_\s H\Q \simeq \hdel (X,\Q)$$
in $Ho(Sp)$ between the rational Deligne cohomology of the dg-category $\lpe(X)$ and the Deligne cohomology spectrum of $X$.  
\end{prop} 

\begin{proof}
Because of the expression of the Deligne complexes as cones, for each $p\in\Z$ there exists a pullback/pushout square in the derived category of $\Q$-sheaves on $X$ denoted by $D(X,\Q)$,
$$\xymatrix{ \Q(p)_\D(X)  \ar[r]^-{ } \ar[d]^-{ } & F^p\drcx \ar[d]^-{ } \\ \und{\Q}_X \ar[r]^-{ } & \drcx }$$
By taking the hypercohomology on $X$, applying the shift $[-2p]$, and taking the product on all integer $p$, we obtain a pullback/pushout square
$$\xymatrix{\hdel^{-*}(X,\Q) \ar[r]^-{ } \ar[d]^-{ } & \hcn_*(X) \ar[d]^-{ } \\ \mrm{H_B^{-*}}(X, \Q)[u^{\pm 1}] \ar[r]^-{ } & \hp_*(X) }$$
in $D(\Q)$, where we used remark \ref{remindhkr} to identify hypercohomology with negative and periodic cyclic homology, and
where $\mrm{H_B^{-*}}(X, \Q)$ is the rational Betti cohomology complex. Using the isomorphism $H(\mrm{H_B^{-*}}(X, \Q)[u^{\pm 1}])\simeq \ktop(X)\sm_\s H\Q$ in $Ho(Sp)$ given by proposition \ref{compvarlisse} and the classical Chern isomorphism, we have a pullback/pushout square 
$$\xymatrix{ \hdel (X,\Q) \ar[r]^-{ } \ar[d]^-{ } & \hcn(X) \ar[d]^-{ } \\  \ktop(X)\sm_\s H\Q \ar[r]^-{ } &\hp(X)  }$$
in $Ho(Sp)$, which gives the expected isomorphism. 
\end{proof}


\subsection{Finite dimensional associative algebras}\label{finalg}

In all this part, the base ring is still the complex field $\C$. By default, all algebras are associative unital $\C$-algebras. We say an algebra is finite dimensional if it is so as a $\C$-vector space. The periodic cyclic homology of an algebra is by convention the periodic cyclic homology relatively to the base field $\C$. Our result concerning finite dimensional algebras does not concern the whole spectrum of topological K-theory but what we call "pseudo-connective topological K-theory". Recall that for any dg-category $T$ over $\C$, the connective semi-topological K-theory (definition \ref{defkst}) $\kcst(T)$ is a $\bu$-module. We can therefore invert the Bott generator,
$$\kctop(T):=\kcst(T)[\beta^{-1}]$$
We call this invariant the \emph{pseudo-connective topological K-theory of $T$}. 
We will give below a proof of the following proposition concerning the pseudo-connective topological K-theory of a finite dimensional $\C$-algebra. 

\begin{prop}\label{cralgass}
Let $B$ be a finite dimensional $\C$-algebra. Then the lattice conjecture \ref{conjres} is true for $B$. More precisely the canonical map 
$$\chctop\sm_\s H\C : \kctop(B)\sm_\s H\C \lmo \hp(B)$$ 
is an isomorphism in $Ho(Sp)$. 
\end{prop}

\begin{nota}
We work over the étale site $\affc$ of affine $\C$-scheme of finite type. We have localization functors,
$$Ho(\spr)\lmo Ho(\spret) \lmo Ho(\spretao).$$
The expression "étale stack" stands for an object of the category $Ho(\spret)$. This category is called the homotopy category of stacks. In §\ref{prestru}, we defined a classifying space functor
$$\mrm{B}:Gr(\spr)\lmo \spr,$$
for strict group objects in $\spr$. The classifying stack functor
$$\bcl : Gr(\spret) \lmo \spret,$$
defined by $\bcl(G)=\und{a} \mrm{B}G$ where $\und{a}$ is a fibrant replacement functor in $\spret$. For all presheaf of groups $G$ (simplicially constant) and all $X\in \affc$, the space $\bcl G(X)$ is equivalent to the nerve of the groupoid of $G$-torsors over $X$. In the sequel we will call algebraic stack a $1$-geometric stack in the sense of \cite[Def 1.3.3.1]{hag2} for the standard context of the étale site $\affc$ and the class of morphisms is the class of smooth morphisms of schemes. Therefore we deal with Artin algebraic stacks.  
We still denote by $\afflissc\hookrightarrow \affc$ the inclusion of smooth schemes, $l^*:\spr\lmo \sprliss$ the restriction and $l^*:\spaf\lmo \spafliss$ the restriction for presheaves of spectra. 
\end{nota}

\paragraph{Different stacks associated to a finite dimensional algebra.} 

Let $B$ be a finite dimensional $\C$-algebra. We consider four stacks associated to $B$, which are organized in the following square of stacks, 

$$\xymatrix{ \vect^B \ar@{^{(}->}[r]^-{ } \ar@{^{(}->}[d]^-{ } & \M^B \ar@{^{(}->}[d]^-{ } \\ \vect_B \ar@{^{(}->}[r]^-{ } & \M_B }$$

These stacks are defined in the following way. We first define four presheaves of Waldhausen categories. 
\begin{itemize}
\item $\uproj(B) : \spec(A)\longmapsto \proj(B\te_\C A)$, where $\proj(B\te_\C A)$ is the Waldhausen category of right finitely projective $B\te_\C A$-modules of finite type.
\item $\upsproj(B) : \spec(A)\longmapsto \psproj(B\te_\C A)$, where $\psproj(B\te_\C A)$ is the Waldhausen category of right $B\te_\C A$-modules which are projective of finite type relative to $A$ (i.e. as right $A$-modules).
\item $\uparf(B) : \spec(A)\longmapsto \parf(B\te_\C A)$, where $\parf(B\te_\C A)$ is the Waldhausen category of cofibrant perfect complexes of right $B\te_\C A$-modules (perfect means homotopically finitely presented in the model category of complexes). 
\item $\upspa(B) : \spec(A)\longmapsto \pspa(B\te_\C A)$, where $\pspa(B\te_\C A)$ is the Waldhausen category of cofibrant complexes of right $B\te_\C A$-modules which are perfect relative to $A$ (i.e. as complexes of right $A$-modules). 
\end{itemize}
Taking the nerve of weak equivalences we have by definition, 
\begin{itemize}
\item $\vect^B=Nw\uproj(B) : \spec(A) \longmapsto Nw\proj(B\te_\C A)$. 
\item $\vect_B=Nw\upsproj(B) : \spec(A) \longmapsto Nw\psproj(B\te_\C A)$. 
\item $\M^B=Nw\uparf(B): \spec(A) \longmapsto Nw\parf(B\te_\C A)$. 
\item $\M_B=Nw\upspa(B) : \spec(A) \longmapsto Nw\pspa(B\te_\C A)$. 
\end{itemize}
As we work over a field, $B$ is locally cofibrant as a dg-category, and therefore the tensor product written previously are in fact derived. The stack $\M_B$ is studied by Toën--Vaquié in \cite{modob}. We remark that since $B$ is finite dimensional, it is proper as a dg-category (in the sense of \cite[Def.2.4]{modob}). By Lem 2.8 of loc.cit., a perfect complex of right $B\te_\C A$-modules is therefore perfect relative to $A$. Hence we have monomorphisms $\vect^B \hookrightarrow \vect_B$ and $\M^B\hookrightarrow \M_B$. The monomorphism $\vect^B\hookrightarrow \M^B$ and $\vect_B\hookrightarrow \M_B$ are inclusion of degree $0$ concentrated complexes. These four stacks admit a structure of homotopy coherent commutative monoid given by direct sum of modules and dg-modules repectively. We apply the functor $B_W$ (defined in \ref{monoids}) and we obtain special $\gam$-objects in $\spr$, $\vect^B_\bul$, $\vect_B^\bul$, $\M^B_\bul$, $\M_B^\bul$ whose level $1$ are respectively $\vect^B,\vect_B,\M^B,\M_B$. 
\\

The stack $\vect^B$ is an algebraic stack which is locally finitely presented and smooth (see for example \cite[§1]{modob}). It admits the following description in terms of residual gerbes on global points ; there exists an isomorphism in $Ho(\spret)$, 
\begin{equation}\label{formvect}
\vect^B \simeq \coprod_{M\in \chi} \bcl \uaut(M)
\end{equation}
where $\chi=\pi_0\vect^B(\C)$ is the set of isomorphism class of projective right $B$-modules of finite type, and for all $M\in \chi$, $\uaut(M)$ is the group scheme of $B$-automorphisms of $M$. This formula can be derived out of a more general formula valid for all algebraic stack locally finitely presented over $\C$, whose tangent complex at every point verifies a finiteness property. 

\begin{prop}\label{formgerb}
Let $F$ be a locally finitely presented smooth algebraic stack over $\C$ such that for all global point $E\in \pi_0F(\C)$ the tangent complex $\tc_E F$ verifies $\mrm{H}^0(\tc_E F)=0$. Then the canonical morphism of algebraic stacks,
$$\coprod_{E\in \pi_0F(\C)} \mcal{G}_E\lmo F$$
is an equivalence of stacks, where $\mcal{G}_E$ stands for the residual gerbe of $F$ at $E$. 
\end{prop} 

\begin{proof} 
The morphism is immediately seen to be locally finitely presented. It is a monomorphism of algebraic stacks and therefore a representable morphism. The assumption on the tangent complexes implies that the morphism induces an equivalence on tangent complexes and is therefore an étale morphism of algebraic stacks. Hence it is an open immersion of algebraic stacks. The morphism is an epimorphism in the $\pi_0$ because it is surjective on the complex points of the $\pi_0$.  It is therefore an epimorphism of stacks. We deduce that it is an equivalence. 
\end{proof}

The stacks $\vect^B$ verifies the assumption of proposition \ref{formgerb} because for all $M\in \chi$, we have a quasi-isomorphism $\tc_M \vect^B\simeq \uend(M)[1]$ where $\uend(M)$ stands for the $\C$-vector space of endomorphisms of $M$. This gives the formula (\ref{formvect}). 

\paragraph{Topological K-theory of a finite dimensional algebra.}

Recall that in §\ref{sectkstmt}, we proved that the connective semi-topological K-theory of a dg-category $T$ can be described in terms of the topological realization of the stack $\M^T$ endowed with its homotopy coherent commutative monoid structure. In other words, we have an equivalence natural in $T$, 
$$\kcst(T)\simeq \mcal{B}\regam{\M^T_\bul}.$$
Now if $T=B$ is an algebra, after remark \ref{algebra}, the connective algebraic K-theory of $B$ can be calculated with the presheaf of categories $\uproj(B)$ endowed with its homotopy coherent commutative monoid structure. Using similar arguments as in the proof of proposition \ref{bu}, we have canonical isomorphisms,
\begin{equation}\label{kvect}
\kcst(B)\simeq \resp{\ukc(B)} \simeq \resp{\mcal{B} K^\gam (\uproj(B))}\simeq \mcal{B} \regam{K^\gam (\uproj(B))}\simeq \mcal{B} \regam{\vect^B_\bul}^+.
\end{equation}
in $Ho(Sp)$. Therefore we have canonical isomorphisms, 
\begin{equation}\label{kcsttot}
\kcst(B)\simeq \mcal{B}\regam{\M^B_\bul}\simeq \mcal{B} \regam{\vect^B_\bul}^+
\end{equation}
in $Ho(Sp)$, and
\begin{equation}\label{mbvectb}
\regam{\M^B_\bul}\simeq \regam{\vect^B_\bul}^+
\end{equation}
in $Ho(\gam-SSet)$. 

\begin{proof} of the proposition \ref{cralgass}. ---
If $B$ is semi-simple, then $B$ verifies proposition \ref{cralgass} for the following reasons. By the theory of representations of semi-simple algebras, $B$ is Morita equivalent to a finite product of copies of the unit algebra $\unit$ (i.e. of copies of $\C$). Therefore there exists a Morita equivalence $B\lmos{\sim} \prod_{J} \unit$ in $\dgcatc$ with $J$ a finite set. Pseudo-connective topological K-theory and periodic homology commute both with finite products, and we thus have canonical isomorphisms $\kctop(B)\lmos{\sim} \prod_{J} \kctop(\unit) \simeq \prod_{J} \BU$ and $\hp(B)\lmos{\sim} \prod_J \hp(\unit)\simeq \prod_J H\cuu$ in $Ho(Sp)$. We then have a commutative square in $Ho(Sp)$, 
$$\xymatrix{ \kctop(B) \ar[r]^-{\chctop} \ar[d]^-{\wr} & \hp(B) \ar[d]^-{\wr}  \\ \prod_J \BU \ar[r]^-{\prod_J \chctop} & \prod_J H\cuu }$$
The bottom map is an isomorphism after $\sm_\s H\C$. We deduce that the top map is an isomorphism after $\sm_\s H\C$. 

Now if $B$ is not semi-simple, its radical $rad(B)$ is a nilpotent two-sided ideal of $B$. We denote by $B_0=B/rad(B)$ the smallest semi-simple quotient of $B$. To prove that $B$ verifies proposition \ref{cralgass} it suffices to show that the map induced by base change $\vect^B\lmo \vect^{B_0}$ is an $\ao$-equivalence in $\spretao$. Indeed such an $\ao$-equivalence implies that the map of special $\gam$-objects $\vect^B_\bul\lmo \vect^{B_0}_\bul$ is an isomorphism in $Ho(\gam-\spretao)$. Taking topological realization we deduce that the map $B\lmo B_0$ induces an isomorphism $\kctop(B)\simeq \kctop(B_0)$ in $Ho(Sp)$. On the other hand, Goodwillie proved the invariance of periodic cyclic homology under infinitesimal extension (\cite[Thm 2.5.1]{goodwilliecyclic}) ; indeed the map $B\lmo B_0$ induces an isomorphism $\hp(B)\simeq \hp(B_0)$ dans $Ho(Sp)$. We therefore reduce the proof to the semi-simple case. 

It suffices thus to show that for all finite dimensional $\C$-algebra $B$ and all nilpotent two-sided ideal $I\subseteq B$ the projection map $B\lmo B/I$ induces an $\ao$-equivalence $\vect^B\lmo \vect^{B/I}$. By a classical recurrence argument on the nilpotence degree of $I$ it suffices to show the assertion for a square-zero two-sided ideal. Indeed suppose that for all algebra $C$ and all nilpotent two-sided ideal $J$ of $C$, the projection map $C\lmo C/J$ induces an $\ao$-equivalence $\vect^C\lmo \vect^{C/J}$. If $I$ is a two-sided ideal of $B$ with nilpotence degree $n\geq 2$, we consider the following commutative square of algebras. 
$$\xymatrix{B\ar[r] \ar[rd] & B/I^2=C \ar[d] \\ & B/I\simeq C/I   }$$
Since $I^2$ is a two-sided ideal of $B$ with nilpotence degree $\leq n-1$ and that the two-sided ideal $I.C\subseteq C$ generated by the image of $I$ in $C$ verifies $I.C^2=0$ the arrows $B\lmo C$ and $C\lmo C/I$ induce $\ao$-equivalences on $\vect$. We deduce that $B\lmo B/I$ also induces an $\ao$-equivalence on $\vect$.

Let now $I\subseteq B$ be a square-zero two sided ideal in a finite dimensional $\C$-algebra $B$. We have to show that $B\lmo B/I=B_0$ induces an $\ao$-equivalence on $\vect$. Using formula (\ref{formvect}), we deduce that the map $\vect^B\lmo \vect^{B_0}$ is equivalent to the map,
$$\coprod_{M\in \chi} \bcl \uaut(M)\lmo \coprod_{M\in \chi_0} \bcl \uaut(M_0),$$
where $\chi=\pi_0\vect^B(\C)$, $\chi_0=\pi_0\vect^{B_0}(\C)$, $M_0=M\te_B B_0$ for all $M\in \chi$. 
First, we observe that the set $\chi$ and $\chi_0$ are isomorphic as shown in \cite[Prop 2.12]{bass} (completeness assumptions are verified because $I^2=0$). Since the $\ao$-equivalences are stable by arbitrary sums in $\spretao$, it remains to show that for all $M\in \chi$, the map $\bcl\uaut(M)\lmo \bcl\uaut(M_0)$ is an $\ao$-equivalence. This latter map is the image by the functor $\bcl$ of the map of group schemes $\uaut(M)\lmo \uaut(M_0)$ whose kernel is denoted by $K$. We write $M$ as a direct factor of a free $B$-module $B^r=M\oplus N$, where $N$ is a sub-right-$B$-module of $B^r$ and $r$ a positive integer. We have $B_0^r=M_0\oplus N_0$. If $M$ is free, the automorphism group is a the invertible matrices group $\uaut(B^r)=\ugl_r(B)$. The kernel of the map $\ugl_r(B)\lmo \ugl_r(B_0)$ is given by the multiplicative group $I_r+\uma_r(I)$ where $I_r$ the identity matrix of rank $r$ and $\uma_r(I)$ the additive group scheme of matrices with coefficients in $I$. The scheme $I_r+\uma_r(I)$ is isomorphic to an affine space (of dimension $dim_\C(I).r^2$) and is therefore $\ao$-contractible. If $M$ is just projective, we have a diagram with exact rows, 

$$\xymatrix{ 1\ar[r] & K\ar[r]\ar[d]_-k  &\uaut(M)\ar[r]^-g \ar[d]_-i &\uaut(M_0) \ar[r]\ar[d]_-j &1 \\ 1\ar[r] & I_r+\uma_r(I)\ar[r] & \ugl_r(B)\ar[r]^-f & \ugl_r(B_0)\ar[r] & 1 }$$

The map $i$ sends an automorphism of $M$ on the automorphism of $B^r$ which is the identity on $N$. The map $j$ is defined in the same way. These two maps are closed immersions, and the corresponding square is commutative, i.e. $fi=jg$. We deduce the existence of the map $k$, which is also a closed immersion. The kernel $K$ corresponds to the sub-group scheme of $\C$-automorphisms of $B^r$ which are $B$-linear, restrict to the identity on $N$, and which are in $I_r+\uma_r(I)$. These three conditions can be translated into affine equations in the affine space $I_r+\uma_r(I)$, and we deduce that $K$ is itself isomorphic to an affine space of a certain dimension, and is therefore $\ao$-contractible. This implies that the classifying stack $\bcl K$ is also $\ao$-contractible, i.e. isomorphic to the point $\ast$ in $Ho(\spretao)$. The functor classifying stack $\bcl$ sends exact sequences of group schemes to fibration sequences. We therefore have a fibration sequence in $Ho(\spret_\ast)$, 
$$\bcl K\lmo \bcl \uaut(M)\lmo \bcl \uaut(M_0).$$
Given a group stack $G$, it is known that there exists a Quillen equivalence between the model category of $G$-equivariant stacks and the model category of stacks over $\bcl G$ (voir \cite[Lem 3.20]{kptalg}). The latter result implies the existence of an isomorphism of stacks over $\bcl \uaut(M_0)$, 
$$\xymatrix{\bcl \uaut(M)\ar[rr]^-\sim \ar[rd] &&  [\bcl K/\uaut(M_0)]\ar[ld] \\ & \bcl \uaut(M_0) & }$$
in $Ho(\spret/\bcl \uaut(M_0))$, for a certain action of $\uaut(M_0)$ on $\bcl K$, and the notation $[-/-]$ stands for the quotient stack. Because $\bcl K$ is contractible in $Ho(\spretao)$ and because the localization functor $Ho(\spret)\lmo Ho(\spretao)$ commutes to the operation of quotient, this latter triangle is isomorphic to 
$$\xymatrix{\bcl \uaut(M)\ar[rr]^-\sim \ar[rd] &&  [\ast/\uaut(M_0)]  = \bcl \uaut(M_0) \ar[ld]^-{id}  \\ & \bcl \uaut(M_0) & }$$
in $Ho(\spretao)$. The map $\bcl \uaut(M)\lmo \bcl \uaut(M_0)$ is therefore an isomorphism in $Ho(\spretao)$. This completes the proof that $\vect^B\lmo \vect^{B/I}$ is an $\ao$-equivalence for all square-zero two-sided ideal $I$ and thus of proposition \ref{cralgass}. 
\end{proof}

Remark that in the proof of \ref{cralgass}, we proved the invariance of connective semi-topological K-theory by infinitesimal extension, which is expressed by the following. 

\begin{prop} 
Let $B$ be a finite dimensional associative algebra over $\C$ and $I$ a right nilpotent ideal of $B$. Then the projection map $B\lmo B/I$ induces an isomorphism $\kcst(B)\simeq \kcst(B/I)$ in $Ho(Sp)$. 
\end{prop}

\begin{rema}\label{stab}
Proposition \ref{cralgass} and formula (\ref{kcsttot}) permits to express the periodic cyclic homology of a finite dimensional algebra in terms of the infinite loop space $(\regam{\vect^B_\bul}^+)_1=:\re{\vect^B}^+$. This group completion has the following description. The monoid $\pi_0\re{\vect^B}$ being in general not isomorphic to $\N$ with the usual addition (as it is for $B$ commutative for example), the calculation of this group completion is a bit more complicated than in the commutative case, because we cannot apply Quillen result on the homology of the group completion. Since every projective right $B$-module of finite type is a direct factor of a free $B$-module of finite type, in order to "group complete" the direct sum in $\re{\vect^B}$, it suffices to invert the action of the regular $B$-module $B$. Hence we have an isomorphism, 
$$\re{\vect^B}^+ \simeq \re{\vect^B}[-B]$$
in $Ho(SSet)$, where the latter object is the level $1$ of the localization in the sense of $\gam$-spaces of the special $\gam$-space $\re{\vect^B}$ with respect to the $B$-module $B$. We now want to calculate this localization in terms of the standard colimit
$$\xymatrix{hocolim (\re{\vect^B}\ar[r]^-{\oplus B} & \re{\vect^B}\ar[r]^-{\oplus B} & \re{\vect^B} \ar[r]^-{\oplus B} & \cdots ) }=:\re{\vect^B}^{ST}$$
where the map $\oplus B$ is induced by the endomorphism of $\uparf(B)$ which sends a $B$-dg-module $E$ on $E\oplus B$ and the notation "stable" is in reference to Loday--Quillen stable homology of Lie algebras of matrices. To achieve this formula, we need the language of Lurie's monoidal $\infty$-groupoids rather than $\gam$-spaces ; because of this, we will willingly stay vague with respect to the definitions we use. Hence we consider $\re{\vect^B}$ as a symmetric monoidal $\infty$-groupoid (whose monoidal law is given by the direct sum of modules) and we want to invert the object $B$ with respect to the sum. For this we apply \cite[Cor 4.24]{marco1} (which works more generally for all presentable symmetric monoidal $\infty$-categories, see \cite[Rem 4.7]{marco1} and \cite[Rem 4.26]{marco1} for the $\infty$-groupoid case). Then there exists a canonical isomorphism
$$\re{\vect^B}[-B]\simeq \re{\vect^B}^{ST}$$
in $Ho(SSet)$, provided we show that the object $B$ is a symmetric object in the sense of \cite[Def 4.18]{marco1}. We have to show that there exists an homotopy between the map
$$\xymatrix{B\oplus B\oplus B \ar[r]^-{(123)} & B\oplus B\oplus B }$$
and the identity of $B\oplus B\oplus B$ in the space of complex points of $\ugl_3(B)$, where $(123)$ is the automorphism induced by the cyclic permutation $(123)$. Such an homotopy is given by the composite of an homotopy $(123)\Rightarrow id$ in $\re{\ugl_3(\C)}$ with the canonical map $\re{\ugl_3(\C)}\lmo \re{\ugl_3(B)}$ induced by the structural morphism $\C\lmo B$. 

In conclusion, there exists an isomorphism
$$\re{\vect^B}^+ \simeq \re{\vect^B}^{ST}$$
in $Ho(SSet)$ between the group completion of $\re{\vect^B}$ and the stabilization of $\re{\vect^B}$ (where the colimit is a homotopy one). Taking homotopy groups of the formula of proposition \ref{cralgass}, we obtain the following corollary.  
\end{rema}

\begin{cor}\label{formhp1}
Let $B$ be a finite dimensional associative $\C$-algebra. Then for all $i\geq 0$ the Chern map $\kctop(B)\lmo \hp(B)$ induces an isomorphism of $\C$-vector spaces,
$$ colim_{k\geq 0} \pi_{i+2k} \re{\vect^B}^{ST}\te_\Z \C\simeq\hp_i(B)$$
where the colimit is induced by the action of the Bott generator $\beta$ on homotopy groups, $\pi_i\re{\vect^B}^{ST} \lmos{\times \beta} \pi_{i+2} \re{\vect^B}^{ST}$.
\end{cor}

\paragraph{Consequences in the smooth case.} Let $B$ be a finite dimensional associative $\C$-algebra which is furthermore of finite global dimension. This extra assumption means precisely that $B$ is smooth in the sense of dg-categories (\cite[Def 2.4]{modob}). If $\spec(A)\in \affc$ is \emph{smooth}, then by Quillen resolution theorem \cite[§4 Cor 1]{qui}, the connective algebraic K-theory of $\proj(B\te_\C A)$ is the same as the connective algebraic K-theory of $\psproj(B\te_\C A)$ and also the same as the connective algebraic K-theory of all right $B\te_\C A$-modules of finite type. Indeed since $B\te_\C A$ is smooth, all right $B\te_\C A$-modules of finite type has a finite projective resolution relative to $B\te_\C A$. We therefore have a global equivalence of presheaves of spectra restricted to $\afflissc$, 
\begin{equation}\label{qresol}
l^*\kc(\uproj(B))\simeq l^*\kc(\upsproj(B)).
\end{equation}
By Thm \ref{restliss} we thus have a canonical isomorphism in $Ho(Sp)$, 
\begin{equation}\label{kpsproj}
\kcst(B)\simeq \resp{\kc(\uproj(B))} \simeq \resp{\kc(\upsproj(B))}
\end{equation}

By what precedes, we have the following proposition. 

\begin{prop}\label{cons1}
Let $B$ be a finite dimensional associative $\C$-algebra of finite global dimension. Then we have canonical isomorphisms in $Ho(Sp)$, 
$$\kcst(B)\simeq \mcal{B}\regam{\vect^B_\bul}^+\simeq\mcal{B}\regam{\vect_B^\bul}^+\simeq \mcal{B}\regam{\M^B_\bul}\simeq \mcal{B}\regam{\M_B^\bul}.$$
And therefore canonical isomorphisms in $Ho(SSet)$, 
$$\re{\vect^B}^+\simeq \re{\vect_B}^+\simeq \re{\M^B}\simeq \re{\M_B},$$
where the first two objects are the level $1$ of the corresponding group completion. 
\end{prop} 

\begin{proof}   
We already proved the isomorphism $\kcst(B)\simeq \mcal{B}\regam{\vect^B_\bul}^+$ in $Ho(Sp)$ with the formula (\ref{kvect}). By formula (\ref{kpsproj}), there is an isomorphism  $\kcst(B)\simeq \resp{\kc(\upsproj(B))}$ in $Ho(Sp)$. Proceeding as in the proof of \ref{bu}, with formula (\ref{kvect}) and by the end of remark \ref{kconspec} applied to the Waldhausen category of pseudo-projective modules which has split cofibrations, we have an isomorphism
$\resp{\kc(\upsproj(B))} \simeq \mcal{B}\regam{\vect_B^\bul}^+$. We thus have an isomorphism $\kcst(B)\simeq \regam{\vect_B^\bul}^+$. The isomorphism $\kcst(B)\simeq \mcal{B}\regam{\M^B_\bul}$ is theorem \ref{kstmt}. The Gillet--Waldausen theorem (\cite[Thm 1.11.7]{tt}, see remark \ref{algebra}) gives an isomorphism $\kc(\upsproj(B))\simeq \kc(\upspa(B))$ in $Ho(\spr)$. By theorem \ref{mtps}, we obtain the isomorphisms $\kcst(B)\simeq \resp{\kc(\upsproj(B))} \simeq \resp{\kc(\upspa(B))} \simeq \mcal{B}\regam{\M_B^\bul}$. The second part of the theorem follows directly because $\mcal{B}$ is an equivalence. 
\end{proof}

As in the case of non necessarily smooth associative algebras, we have isomorphisms
$$\xymatrix{ \re{\vect_B}^+ \simeq hocolim (\re{\vect_B}\ar[r]^-{\oplus B} & \re{\vect_B}\ar[r]^-{\oplus B} &\re{\vect_B}\ar[r]^-{\oplus B} & \cdots ) =: \re{\vect_B}^{ST} }$$
in $Ho(SSet)$ (see remark \ref{stab}). We deduce from this the following corollary of proposition \ref{cralgass}. 

\begin{cor}\label{cons2}
Let $B$ be a finite dimensional associative $\C$-algebra of finite global dimension. Then for all $i\geq 0$ the Chern map $\kctop(B)\lmo \hp(B)$ induces an isomorphism of $\C$-vector spaces,
$$ colim_{k\geq 0} \pi_{i+2k} \re{\vect_B}^{ST}\te_\Z \C  \simeq \hp_i(B)$$
where the colimit is induced by the action of the Bott generator $\beta$ on homotopy groups, $\pi_i \re{\vect_B}^{ST}\lmos{\times \beta} \pi_{i+2} \re{\vect_B}^{ST}$. 
\end{cor}


\bibliographystyle{amsalpha}
\bibliography{/users/Anthony/Dropbox/MATH/Latex/Bibliographie/ref}

\end{document}